\documentclass{article}
\usepackage[margin=1.5in]{geometry} 
\usepackage{setspace} 
\usepackage{amsmath, amsfonts, amsthm, amstext, amssymb,mathtools} 
\usepackage[mathscr]{eucal}
\usepackage{xcolor} 
\usepackage{hyperref} 
\usepackage{todonotes} 
\usepackage[nameinlink, capitalize]{cleveref} 
\usepackage{tikz-cd} 
\usepackage[shortlabels]{enumitem} 

\numberwithin{equation}{section} 
\numberwithin{table}{section} 

\theoremstyle{definition}

\newtheorem{definition}[equation]{Definition}
\newtheorem{remark}[equation]{Remark}
\newtheorem{example}[equation]{Example}

\newtheorem{organization}[equation]{Organization}
\newtheorem{acknowledgements}[equation]{Acknowledgements}
\theoremstyle{plain} 

\newtheoremstyle{morespacebeforeandafter}
  {0.3in} 
  {0.3in} 
  {} 
  {} 
  {\bfseries} 
  {.} 
  {.5em} 
  {} 

\theoremstyle{morespacebeforeandafter}

\newtheorem{lemma}[equation]{Lemma}
\newtheorem{proposition}[equation]{Proposition}
\newtheorem{theorem}[equation]{Theorem}
\newtheorem{corollary}[equation]{Corollary}

\def\FF{\mathbb{F}}

\def\ZZ{\mathbb{Z}}

\def\sV{\mathscr{V}}

\begin{document}

\title{Complex cobordism with involutions and geometric orientations}

\author{Jack Carlisle}



\maketitle

\begin{abstract} 
We calculate the cobordism ring $\Omega^{C_2}_*$ of stably almost complex manifolds with involution, and investigate the $C_2$-spectrum $\Omega_{C_2}$ which represents it. We introduce the notion of a geometrically oriented $C_2$-spectrum, which extends the notion of a complex oriented $C_2$-spectrum, and of which $\Omega_{C_2}$ is the universal example. Examples, in addition to $\Omega_{C_2}$,  include the Eilenberg-Maclane spectrum $H \underline{\ZZ}_{C_2}$ and the connective cover $k_{C_2}$ of $C_2$-equivariant $K$-theory. On the algebraic side, we define and study filtered $C_2$-equivariant formal group laws, which are the algebraic structures determined by geometrically oriented $C_2$-spectra. We prove some of the fundamental properties of filtered $C_2$-equivariant formal group laws, as well as a universality statement for the filtered $C_2$-equivariant formal group law determined by $\Omega_{C_2}$.
\end{abstract}

\tableofcontents

\section{Introduction} \label{introduction} 

If $E$ is a commutative ring spectrum, then a complex orientation of $E$ is a cohomology class $x \in \tilde{E}^2(\mathbf{CP}^\infty)$ whose restriction to $\mathbf{CP}^1 \subset \mathbf{CP}^\infty$ corresponds to the unit $1 \in E_0 \cong \tilde{E}^2(\mathbf{CP}^1)$. Such a complex orientation of $E$ determines a well-behaved analogue of chern classes in $E$-cohomology. Important examples of complex oriented spectra include the Eilenberg-Maclane spectrum $H \ZZ$, the complex $K$-theory spectrum $K$, and the complex cobordism spectrum $MU$. In fact, $MU$ is the universal complex oriented spectrum, which means that complex orientations of $E$ correspond to ring spectrum maps $MU \to E$. For this reason, the spectrum $MU$ plays a distinguished role in stable homotopy theory.

Algebraically, complex orientations correspond to {\it formal group laws}. More precisely, a complex orientation of $E$ determines a formal group law $F_E(y,z) \in E_*[[y,z]]$, which encodes much of the structure of the spectrum $E$. For example, the formal group law associated to $H \ZZ$ is $F_{H\ZZ}(y,z) = y + z$, and the formal group law associated to $K$ is $F_{K}(y,z) = y + z - vyz$, where $v \in  K_*$ is the Bott element. In his celebrated theorem, Quillen \cite{Quillen} proved that the formal group law $F_{MU}(y,z) \in MU_*[[y,z]]$ associated to the complex cobordism spectrum $MU$ is universal, which means that formal group laws over a commutative ring $A$ correspond to ring homomorphisms $MU_* \to A$. This result has served as an organizing principle for homotopy theory.

Seemingly unrelated to complex orientations and formal group laws is the geometric complex cobordism ring $\Omega_*$. Elements of $\Omega_*$ are represented by stably almost complex manifolds, and we declare $[M] = 0$ in $\Omega_*$ if there is a stably almost complex manifold $W$ with boundary $\partial W = M$. By work of Pontrjagin and Thom (\cite{Pontrjagin1}, \cite{Pontrjagin2}, \cite{Thom}), there is a ring isomorphism $\Omega_* \overset{\cong}{\longrightarrow} MU_*$ which, combined with Quillen's theorem on the universality of $MU_*$, provides a fascinating link between the topology of manifolds and the algebraic geometry of formal groups.

There is a $G$-equivariant analogue of this story when $G$ is an abelian compact Lie group. In \cite{CGK1} and \cite{CGK2}, the authors develop the theory of complex oriented $G$-spectra, and their associated $G$-equivariant formal group laws. They prove that the $G$-equivariant Thom spectrum $MU_G$, which has been studied extensively ( \cite{GreenleesMay1}, \cite{Loffler1}, \cite{Loffler2}, \cite{May}, \cite{tomDieck}), satisfies the desired homotopical universal property, namely that complex orientations of a $G$-spectrum $E_G$ correspond to ring $G$-spectrum maps $MU_G \to E_G$. Moreover, $MU^G_*$ satisfies the expected algebraic universal property, namely that $G$-equivariant formal group laws over a commutative ring $A$ correspond to ring homomorphisms $MU^G_* \to A$. This was first proved by Hanke and Wiemeler \cite{HW} in the case $G = C_2$, and later by Hausmann \cite{Hausmann} for any abelian compact Lie group $G$ using methods from global homotopy theory \cite{Schwede2}.

There is however, one major result which does not generalize to the $G$-equivariant setting. The {\it geometric} cobordism ring $\Omega^G_*$, whose elements are cobordism classes of stably almost complex $G$-manifolds, does not coincide with the {\it stable }cobordism ring $MU^G_*$ when $G$ is non-trivial. This is related to the fact that transversality is not a generic property in the equivariant setting, so one can not construct an inverse to the equivariant Pontrjagin-Thom map $\Omega^G_* \to MU^G_*$. For this reason, calculating the geometric complex cobordism ring $\Omega^G_*$ has proved difficult. In particular, the stable cobordism ring $MU^G_*$ has been calculated in many cases (see \cite{abramskriz}, \cite{Kriz}, \cite{Sinha}, \cite{Strickland}), but there have been no explicit calculations of $\Omega^G_*$ for $G$ a non-trivial group. Prior to the current work, it is known only that $\Omega^G_*$ is a free $MU_*$-module concentrated in even degrees when $G$ is abelian \cite{Comezana}, and when $G = D_{2p}$ is the dihedral group of order $2p$ \cite{Uribe}.

Since $G$-equivariant geometric complex cobordism is so poorly understood, we restrict to the case $G = C_2$ where we aim to develop a complete picture. One major accomplishment of the present paper is a complete calculation of the $C_2$-equivariant geometric complex cobordism ring $\Omega^{C_2}_*$ (Theorem \ref{geometriccobordism}). Our calculation is based on the observation that there is a $C_2$-spectrum $\Omega_{C_2}$ whose coefficient ring naturally coincides with the geometric cobordism ring $\Omega^{C_2}_*$. We call $\Omega_{C_2}$ the {\it geometric }cobordism spectrum, as opposed to the {\it stable }cobordism spectrum $MU_{C_2}$. We extend our calculation of $\Omega^{C_2}_*$ by calculating the entire $RO(C_2)$-graded coefficients of $\Omega_{C_2}$ (Theorem \ref{fullcobordism}), which are much more complicated than the $\ZZ$-graded part $\Omega^{C_2}_*$. Of particular importance is a certain subring $\Omega^{C_2}_\diamond \subset \Omega^{C_2}_\star$, which we call the {\it extended coefficients} of $\Omega_{C_2}$, or the {\it good range} of $\Omega^{C_2}_\star$ (see section \ref{background} for definition). This subring of $\Omega^{C_2}_\star$ is especially well-behaved, and plays a prominent role in our theory of ``geometric orientations". 

Motivated by our analysis of $\Omega_{C_2}$, we develop a theory of {\it geometrically oriented} $C_2${\it -spectra}, which extends the theory of complex-oriented $C_2$-spectra. A geometric orientation of $E_{C_2}$ is a ring $C_2$-spectrum map $\Omega_{C_2} \to E_{C_2}$, subject to several mild flatness hypotheses (see Definition \ref{quasiorientation}). Our theory of geometrically oriented $C_2$-spectra is interesting because of the wealth of naturally occuring examples. For instance, the Eilenberg-Maclane $C_2$-spectrum $H \underline{\ZZ}_{C_2}$ is geometrically oriented, as is the connective cover $k_{C_2}$ of $C_2$-equivariant $K$-theory. We establish a connection between geometrically oriented $C_2$-spectra and thom isomorphisms for certain $C_2$-equivariant complex vector bundles (see \ref{thomisos} for a precise statement). We also develop a close link between the theory of geometrically oriented $C_2$-spectra and that of complex oriented $C_2$-spectra. More precisely, we prove that by inverting an element $\tau \in E^{C_2}_{\star}$, one can ``stabilize" a geometrically oriented $C_2$-spectrum $E_{C_2}$ to obtain a complex oriented $C_2$-spectrum $\widehat{E}_{C_2}$. We illustrate the general theory by calculating the extended coefficient ring and stabilization of the geometrically oriented $C_2$-spectra $\Omega_{C_2}$, $k_{C_2}$, and $H\underline{\ZZ}_{C_2}$ (Theorem \ref{stable}). For completeness, we calculate the full $RO(C_2)$-graded coefficients of $\Omega_{C_2}$ and $k_{C_2}$ in section \ref{appendixA}.

On the algebraic side, we develop a theory of {\it filtered} $C_2${\it -equivariant formal group laws}, which are the algebraic structures determined by geometrically oriented $C_2$-spectra (Definition \ref{filtered}). The algebraic structure present on a filtered $C_2$-equivariant formal group law is incredibly rich. In particular, any ``complete flag", by which we mean a sequence of $1$s and $\sigma$s with each occuring infinitely many times, determines a direct sum decomposition of the filtered $C_2$-equivariant formal group law.  These direct sum decompositions are related by change of basis matrices, whose entries are represented geometrically by $C_2$-equivariant projective spaces. For this reason, the classes $[\mathbf{CP}(m+n\sigma)] \in \Omega^{C_2}_*$ play a priveleged role in our theory. We analyze these classes and their interaction with filtered $C_2$-equivariant formal group laws in section \ref{changeofbasis}. Finally, we prove an algebraic universality statement for the filtered $C_2$-equivariant formal group law determined by the universal geometrically oriented $C_2$-spectrum $\Omega_{C_2}$ (Theorem \ref{universality}).

\begin{organization} In section 2, we establish notation and make the definitions necessary to state our main theorems, which we do in section 3. In section 4, we calculate the geometric cobordism ring $\Omega^{C_2}_*$, as well as the good range $\Omega^{C_2}_\diamond$ of the $RO(C_2)$-graded coefficients of $\Omega_{C_2}$. In section 5, we introduce our new notion of geometrically oriented $C_2$-spectra, and illustrate the theory by calculating the good range of the $RO(C_2)$-graded coefficients of $H \underline{\ZZ}_{C_2}$ and $k_{C_2}$. On the algebraic side, we introduce the notion of filtered $C_2$-equivariant formal group laws, which are the algebraic structures associated to geometrically oriented $C_2$-spectra, and prove a universality statement for the filtered $C_2$-equivariant formal group law associated to $\Omega_{C_2}$. In section 6, we calculate the full $RO(C_2)$-graded coefficients of $k_{C_2}$ and $\Omega_{C_2}$, which is more difficult than our calculations of $k^{C_2}_\diamond$ and $\Omega^{C_2}_\diamond$. Finally, in the Appendix, we prove a technical lemma needed in section 4, and we give our new definition of ``homological" $C_2$-equivariant formal group laws. We prove a version of Cartier duality in this setting, which confirms that our definition is compatible with the original ``cohomological" formulation of $C_2$-equivariant formal group laws given in \cite{CGK1}.
\end{organization}

\begin{acknowledgements}
I would like to thank my thesis advisor, Igor Kriz, without whose guidance, direction, and inspiration this work would not have been possible.
\end{acknowledgements}

\section{Definitions and background}\label{background}

In this section, we make the definitions necessary to state our results. We begin by recalling some basic notions from representation theory and $C_2$-equivariant homotopy theory. Let $C_2$ be the group of order $2$. We write $\mathbf{R}$ and $\mathbf{R}^\alpha$ for the trivial and sign representations of $C_2$, so the real representation ring of $C_2$ is 
\[RO(C_2) = \ZZ[\alpha]/(\alpha^2 - 1),\]
 where $1 = [\mathbf{R}]$ and $\alpha = [\mathbf{R}^\alpha]$. We write $\mathbf{C}$ and $\mathbf{C}^\sigma$ for the complex trivial and sign representations of $C_2$, so the complex representation ring of $C_2$ is 
 \[R(C_2) = \ZZ[\sigma]/(\sigma^2 - 1),
 \] where $1 = [\mathbf{C}]$ and $\sigma = [\mathbf{C}^\sigma]$. We consider $R(C_2)$ as a subgroup of $RO(C_2)$ by the assignment $m+ n\sigma \mapsto 2m + 2n\alpha$.
 We work primarily with complex $C_2$-representations, and in many cases omit the adjective ``complex". If $V$ is a $C_2$-representation, we write $\dim V$ for the complex dimension of $V$ and $|V| = 2 \dim V$ for the real dimension of $V$. For $m,n \in  \{ 0 , 1 , 2, \dots , \infty \}$, we write $\mathbf{C}^{m,n}$ or $\mathbf{C}^{m+n\sigma}$ for the $C_2$-representation 
 \[
\underbrace{ \mathbf{C} \oplus \cdots \oplus \mathbf{C}}_{m \text{ times}} \oplus \underbrace{\mathbf{C}^{\sigma} \oplus \cdots \oplus \mathbf{C}^{\sigma}}_{n \text{ times}}.
 \]
We write $S^V$ for the one-point compactification of $V$, which is a based $C_2$-space with basepoint $\infty \in S^V$. We write $\mathbf{CP}(V)$ for the $C_2$-space of one-dimensional subspaces of $V$.

Next, we recall some basic notions from $C_2$-equivariant stable homotopy theory. We work in the category $\text{Sp}_{C_2}$ of $C_2$-spectra indexed on the complete complex $C_2$-universe $U = \mathbf{C}^{\infty,\infty}$ in the sense of \cite{LMS}. There are many other point-set models for the category of spectra and $C_2$-spectra, such as orthogonal spectra and symmetric spectra \cite{symmetricspectra}. For a comparison, see \cite{orthogonalspectra}. Our results are independent of the particular point-set model of $C_2$-spectra used, so it is of no substantial consequence that we choose to work in the aforementioned category. A $C_2$-spectrum $E_{C_2}$ assigns to each finite-dimensional subrepresentation $V \subset U$ a based $C_2$-space $E_{C_2}(V)$, together with a coherent family of maps 
\begin{equation}\label{structuremaps}
S^{W-V} \wedge E_{C_2}(V) \to E_{C_2}(W)
\end{equation}
for each inclusion of finite-dimensional sub-representations $V \subset W$ of $U$, where $W-V$ is the orthogonal complement of $V$ in $W$. The maps adjoint to \ref{structuremaps} are required to be homeomorphisms. If they are not, we obtain the definition of a  $C_2$-prespectrum. The inclusion of $C_2$-prespectra into $C_2$-spectra has a left adjoint called ``spectrification", so we can associate to any $C_2$-prespectrum $E_{C_2}$ a spectrum which, by a mild but common abuse of notation, we also denote $E_{C_2}$.

The primary algebraic invariant of a $C_2$-spectrum $E_{C_2}$ is the $C_2$-Mackey functor $\underline{\pi}_*(E_{C_2})$, which we can think of as genuine $C_2$-equivariant analogue of an abelian group. See \cite{thevenaz} for a thorough treatment of Mackey functors. It suffices for our purposes to know that a $C_2$-equivariant Mackey functor $\underline{M}$ is a diagram of the form 
\[ \begin{tikzcd} 
\underline{M}(C_2/e) \ar[rr, bend right = 20,swap,"\text{tr}"] \ar[loop left,"\gamma"] & & \underline{M}(C_2/C_2) \ar[ll, bend right = 20,swap,"\text{res}"] 
\end{tikzcd} \]
such that $\gamma \circ \gamma= 1$, $\gamma \circ \text{res} = \text{res}$, $\text{tr}\circ \gamma = \text{tr}$, and $\text{res} \circ \text{tr} = 1 + \gamma$.
If $E_{C_2}$ is a $C_2$-spectrum, then for any $m \in \ZZ$ we have a  Mackey functor $\underline{M} = \underline{\pi}_{m}(E_{C_2})$ satisfying  
\[
\underline{M}(C_2/C_2) = E^{C_2}_m = [ S^{m} , E_{C_2}]^{C_2}\text{, and}
\]
\[
\underline{M}(C_2/e) = E_m = [C_2/e_+ \wedge S^{m} , E_{C_2}]^{C_2},
\]
where $[X_{C_2},Y_{C_2}]^{C_2}$ denotes the abelian group of maps from $X_{C_2}$ to $Y_{C_2}$ in the $C_2$-equivariant stable homotopy category $\text{Ho}(\text{Sp}_{C_2})$. We can define Mackey functors $\underline{\pi}_{m+n\alpha}(E_{C_2})$ for $m + n \alpha \in RO(C_2)$ similarly, and we write $\underline{\pi}_{\star}(E_{C_2})$ for the  $RO(C_2)$-graded homotopy Mackey functor of $E_{C_2}$.  Since the underlying homotopy groups of the $C_2$-spectra with which we work in this paper are well understood, we focus on calculating the value $E^{C_2}_\star$ of $\underline{\pi}_\star(E_{C_2})$ at $C_2/C_2$. In the present paper, it will be natural to consider the subgroup  $E^{C_2}_\diamond \subset E^{C_2}_\star$ given by 
\[
E^{C_2}_\diamond =  \bigoplus_{m \in \ZZ \text{ and }n \geq 0} \pi^{C_2}_{m-n\sigma}(E_{C_2}).
\]
We call $E^{C_2}_\diamond$ the {\it extended coefficient ring } of $E^{C_2}_*$, or the {\it good range} of $E^{C_2}_\star$.

We now define the geometric and stable cobordism spectra $\Omega_{C_2}$ and $MU_{C_2}$, which are our primary objects of study. If $V$ and $\sV$ are unitary $C_2$-representations, let $\text{Gr}^\sV(V)$
be the $C_2$-space of complex $\dim V$-dimensional subspaces of  $V \oplus \sV$. Let 
\begin{align*}
MU^\sV_{C_2}(V) & = \text{Thom}\left( \xi^\sV(V) \to \text{Gr}^\sV(V) \right)
\end{align*}
be the Thom space of the tautological vector bundle $\xi^{\sV}(V)$ over $\text{Gr}^\sV(V)$. For a fixed $C_2$-representation $\sV$, $MU^{\sV}_{C_2}$ is a $C_2$-prespectrum indexed on $U$, with structure maps
\[ S^{W-V} \wedge MU^\sV_{C_2}(V)  \to MU^\sV_{C_2}(W)\]
induced by the vector bundle maps $ \underline{W-V} \oplus \xi^\sV(V)  \to \xi^\sV(W)$. If $\sV$ is a $C_2$-universe, then $MU^\sV_{C_2}$ is 
a commutative ring spectrum, by which we mean a commutative monoid in the stable homotopy category $\text{Ho}(\text{Sp}_{C_2})$.

\begin{definition}
We define the $C_2${\it -equivariant geometric complex cobordism spectrum} $\Omega_{C_2}$ by 
\[\Omega_{C_2} = MU^{\mathbf{C}^\infty}_{C_2},\]
 and we define the $C_2${\it -equivariant stable complex cobordism spectrum} $MU_{C_2}$ by 
 \[ MU_{C_2} = MU^{\mathbf{C}^{\infty,\infty}}_{C_2}.\] 
 \end{definition}
 The inclusion of $C_2$-universes $\mathbf{C}^\infty \to \mathbf{C}^{\infty,\infty}$ induces a map $\Omega_{C_2} \to MU_{C_2}$, giving $MU_{C_2}$ the structure of an $\Omega_{C_2}$-algebra. In section \ref{thomspectra}, we will prove that the $C_2$-spectrum $\Omega_{C_2}$ represents geometric $C_2$-equivariant complex cobordism, in that the $\ZZ$-graded coefficient ring of $\Omega_{C_2}$ coincides with the ring of cobordism classes of stably almost complex $C_2$-manifolds with involution. In the unoriented case, the analogous fact holds for any finite abelian group $G$, by work of Conner and Floyd \cite{connerfloyd}. In the complex case, this is not known in general.

Next we review the notion of (non-equivariant) formal group laws. We refer the reader to \cite{hazewinkel} and \cite{ravenel} for more information about formal group laws. If $A$ is a commutative ring, then a formal group law over $A$ is a power series $F(y,z) \in A[[y,z]]$ satisfying the expected identity, associativity, and commutativity axioms. For example, if $A$ is any commutative ring, we have the additive formal group law $F(y,z) = y + z$ over $A$, and the multiplicative formal group law $F(y,z) = y + z - yz$ over $A$. If $E$ is a complex oriented spectrum, then $E^*(\mathbf{CP}^\infty) = E^*[[x]]$ and $E^*(\mathbf{CP}^\infty \times \mathbf{CP}^\infty) = E^*[[y,z]]$,
and the pullback of $x$ along the multiplication map $\mathbf{CP}^\infty \times \mathbf{CP}^\infty \to \mathbf{CP}^\infty$ is a formal group law $F_E(y,z) \in E^*[[y,z]]$ over the coefficient ring $E^* \cong E_*$. For example, it turns out that $F_{H\ZZ}(y,z) = y + z$ is the additive formal group law, and $F_{K}(y,z) = y + z - vyz$, where $v \in K_{2}$ is the Bott element. More shockingly, Quillen proved that
\[
F_{MU}(y,z) = \sum_{i,j \geq 0} a_{i,j}y^iz^j \in MU^*[[y,z]]
\]
is the universal formal group law. The elements $a_{i,j} \in MU_*$ generate $MU_*$ as a ring, and it is often convenient to think of $MU_*$ in terms of the presentation $\ZZ[a_{i,j} :  i , j \geq 0]/ \sim$, where we kill the relations enforced by the identity, associativity, and commutativity axioms of a formal group law.

 In our presentation of $\Omega^{C_2}_*$, we reference certain elements $c_{i,j} \in MU_*$ which are related to the elements $a_{i,j} \in MU_*$ by formal group theoretic data. Although the definition of the elements $c_{i,j}$ is provided in our theorem statement, we define these elements in detail here for the reader's convenience. If $u$ and $x$ are variables, we can expand the power series $F_{MU}(u,x) \in MU_*[[u,x]]$ in the variable $x$ as follows,
\begin{align*}
F_{MU}(u,x)
& = u + (\sum a_{1,j}u^j)x + (\sum a_{2,j}u^j) x^2 + (\sum a_{3,j}u^j)x^3 + \cdots \in MU_*[[u]][[x]]
\end{align*} 
The constant term in this power series is $u$, so in the ring $\left( u^{-1} MU_*[[u]] \right)[[x]] = MU_*((u))[[x]]$, the element $F_{MU}(u,x)$ is a unit. Its multiplicative inverse is some power series
\[
\dfrac{1}{F_{MU}(u,x)} = d_0+ d_1x + d_2x^2 + \cdots \in MU_*((u))[[x]]
\]
whose coefficients $d_0,d_1,d_2,\dots \in MU_*((u))$ are Laurent series' in $u$. We define $c_{i,j} \in MU_*$ to be the coefficient of $u^j$ in $d_i$, so that 
\[d_i = \sum_{j \in \ZZ} c_{i,j}u^j \in MU_*((u)).\]
We note that $c_{i,j} = 0$ if $j < -i-1$.

Next, we establish the notation necessary to define {\it filtered }$C_2${\it -equivariant formal group laws}. The category of (coassociative, cocommutative, counital) $A$-coalgebras is symmetric monoidal under the tensor product $ \otimes_A$ with unit $A$. We say $D$ is an $A$-Hopf algebra if $D$ is a group object in the category of $A$-coalgebras. An example of such an object is the group algebra $A[G]$ of a finite abelian group $G$. The multiplication and antipode on the group object $A[G]$ are induced by the multiplication and inverse map on the group $G$. We will be interested in the case $G = C_2^\vee$ is the Pontrjagin dual of $C_2$. 

If $x$ is an  $A$-linear functional on $D$, we write $\langle d , x \rangle $ for the value of $ x$ at $d \in D$, and we write $\cap x$ for the comultiplication-by-$x$ map
\[ \begin{tikzcd} 
D \ar[r,"\Delta"] & D \otimes D \ar[r,"1 \otimes x "] & D \otimes A \cong D.
\end{tikzcd} \]
 If $D$ is an $A$-Hopf algebra equipped with an $A$-Hopf algebra map $A[C_2^\vee] \to D$, we write $x^\sigma \in \text{Hom}_A(D,A)$ for the functional $\langle d , x^\sigma \rangle = \langle \sigma d , x \rangle$, and for any $m,n \geq 0$, we write $x^{m+n\sigma} \in \text{Hom}_A(D,A)$ for the functional
\[
\langle d , x^{m+n\sigma} \rangle = \langle \Delta d , \underbrace{x \otimes \cdots \otimes x}_{m\text{ times}} \otimes \underbrace{x^{\sigma} \otimes \cdots \otimes x^\sigma}_{n\text{ times}}\rangle.
\]
We can now state our definition of $C_2$-equivariant formal group laws, and refer the reader to \ref{AppendixD} for further discussion.
\begin{definition}
A (homological) $C_2$-equivariant formal group law $(A,D)$ consists of a commutative ring $A$, an $A$-Hopf algebra $D$, a morphism  of $A$-Hopf algebras $A[C_2 ^\vee] \to D$, and an $A$-linear functional $x$ on $D$, such that 
\begin{enumerate}
\item the sequence 
\[ \begin{tikzcd} 
0 \ar[r] & A \ar[r,"\eta"] & D \ar[r,"\cap x"] & D \ar[r] & 0
\end{tikzcd} \]
is exact, and 
\item if $d \in D$, then there exists $m,n \geq 0 $ such that 
\[
d \cap x^{m+n\sigma} = 0.
\]
\end{enumerate}
\end{definition}

The final algebraic preliminary we need is the Rees construction, which arises in our calculation of $\Omega^{C_2}_\diamond$. Suppose $A$ is a commutative ring with an increasing filtration
\[
F_\bullet A = (F_0A \subseteq F_1A \subseteq F_2A \subseteq \cdots  \subseteq A) .
\]
The Rees algebra $\text{Rees}(A) \subseteq A[t^{\pm 1}]$ is the subring of $A[t^{\pm 1}]$ consisting of polynomials $\sum f_i t^i$ such that $f_i = 0$ if $i < 0$ and $f_i \in F_iA$ if $i \geq 0$. It is informative to think of $\text{Rees}(A)$ as a deformation with parameter $t$, generic fiber
\[
\text{Rees}(A)/(t - 1) = A
\]
and special fiber
\[
\text{Rees}(A)/(t - 0) = \text{Gr}_\bullet A = \bigoplus_{n \geq 0} F_nA/F_{n-1}A.
\]


\section{Statement of results}

Having developed the necessary background and notation, we can now state our results. Our first major result is a calculation of the ring $\Omega^{C_2}_*$ of stably almost complex manifolds with involution. We give a presentation of $\Omega^{C_2}_*$ as an algebra over the non-equivariant complex cobordism ring $MU_*$, whose structure is well known.

\begin{theorem} \label{geometriccobordism}There is an isomorphism of graded rings 
\begin{align*}
\Omega^{C_2}_* & \cong MU_*\left[d_{i,j},q_j \right] /I
\end{align*}
for $i \geq 1$ and $j \geq 0$, where 
\begin{itemize} 
\item $I \subset MU_*[d_{i,j},q_j]$ is the ideal generated by the relations 
\begin{align*} 
d_{i,j+1}(d_{k,\ell} - c_{k,\ell}) & =  (d_{i,j} - c_{i,j}) d_{k,\ell + 1}  \\
d_{i,j+1}(q_\ell - p_\ell) & = (d_{i,j}-c_{i,j})q_{\ell +1} \\
q_{j+1}(q_\ell - p_\ell) & =  ( q_j - p_j)q_{\ell + 1} \\
q_0 & = 0,
\end{align*} 
for $i,k \geq 1$ and $j,\ell \geq 0$,
\item  $c_{i,j} \in MU_*$ 
is the coefficient of $u^jx^i$ in $\dfrac{1}{F_{MU}(u,x)}\in MU_*((u))[[x]]$,
\item $p_j \in MU_*$ is the coefficient of $x^j$ in $F_{MU}(x,x) \in MU_*[[x]]$, and 
\item  $|d_{i,j}| = |c_{i,j}| = 2(i+j+1)$, and $|q_j| = |p_j| = 2j-2$.
\end{itemize}
\end{theorem} 

Our next major result is a calculation of the extended geometric complex cobordism ring $\Omega^{C_2}_\diamond$, which is the ``good range" of the $RO(C_2)$-graded coefficients of $\Omega_{C_2}$. This ring ends up playing a major role in our new theory of geometric orientations.

\begin{theorem} Let $\Omega_{C_2}$ denote the $C_2$-equivariant geometric complex cobordism spectrum. 
\begin{enumerate}
\item The extended coefficient ring $\Omega^{C_2}_\diamond$ is given by 
\begin{align*}
\Omega^{C_2}_\diamond & =  \dfrac{\Omega^{C_2}_*[\mu,\tau]}{\begin{matrix}\tau(d_{i,j}-c_{i,j}) = \mu d_{i,j+1} \\ \tau(q_j - p_j) = \mu q_{j+1}\end{matrix} }\hspace{0.3in} i \geq 1 \text{ and }j \geq 0,
\end{align*}
where $|\mu| = - \sigma$, and $|\tau| = 2 - \sigma$. Additively,
\[
\Omega^{C_2}_{*-n\sigma} = \widetilde{\Omega}^{C_2}_*(S^{n\sigma}) \cong  \dfrac{\Omega^{C_2}_* \{ 1 , \dots , u^n\}}{\begin{matrix} u^k(d_{i,j}-c_{i,j}) = u^{k+1}d_{i,j+1} \\ u^k(q_j - p_j) = u^{k+1}q_{j+1} \end{matrix} }\hspace{0.3in} i \geq 1 \text{ and }j \geq 0,
\]
where $ 0 \leq k < n$.
\item If we define the {\it euler filtration } of $MU^{C_2}_*$ by letting $F_n MU^{C_2}_*$ be the $\Omega^{C_2}_*$-submodule generated by $1 , \dots , u^n \in MU^{C_2}_*$, then the map
\[ \Omega^{C_2}_\diamond \to MU^{C_2}_\star = MU^{C_2}_*[\tau^{ \pm 1}]\]
identifies $\Omega^{C_2}_\diamond \cong \text{Rees}(MU^{C_2}_*)$ with the Rees algebra of the euler filtration of $MU^{C_2}_*$.
\item The associated graded of $MU^{C_2}_*$ with respect to the euler filtration is 
\[
\text{gr}_\bullet MU^{C_2}_* =  \Omega^{C_2}_*[\mu]/(\mu d_{i,j}, \mu q_j),  \hspace{0.3in}i,j\geq 1.
\]
Additively,
\[ 
\text{gr}_nMU^{C_2}_* \cong 
\begin{cases}
\Omega^{C_2}_* & n = 0\\ 
MU_*[d_1,d_2,\dots] & n > 0.
\end{cases} 
\]
\end{enumerate}
\end{theorem}
Finally, we complete our calculation of the full $RO(C_2)$-graded coefficients of $\Omega_{C_2}$.

\begin{theorem} \label{fullcobordism}
The $RO(C_2)$-graded coefficients of $\Omega_{C_2}$ are listed below. 
\begin{enumerate} 
\item 
\[
\Omega^{C_2}_{*+2n-2n\alpha} = \dfrac{\Omega^{C_2}_* \{ 1 , \dots , u^n\}}{\begin{matrix} u^k(d_{i,j}-c_{i,j}) = u^{k+1}d_{i,j+1} \\ u^k(q_j - p_j) = u^{k+1}q_{j+1} \end{matrix} } \hspace{0.3in} \begin{matrix} i \geq 1 \text{ and }j \geq 0\\ 0 \leq k < n \end{matrix}
\]
\item 
\[ \Omega^{C_2}_{* -2n + 2n\alpha} =  MU_*\{q_1\} \oplus \Omega^{C_2}_* \cap (u^n)   \oplus MU_{*-1}[u]/\left(u^n, \sum_{\ell = 0}^{n-1} p_{j+\ell} u^\ell, \sum_{\ell = 0}^{n-1}  d_{i,j+\ell} u^\ell\right)\]
We provide generators for the ideal $\Omega^{C_2}_* \cap (u^n) \subset \Omega^{C_2}_*$ in Proposition \ref{generators}.
\item \[\Omega^{C_2}_{* + (2n+1) - (2n+1)\alpha} \cong \Omega^{C_2}_{*+2n - 2n\alpha}/q_1\]
\item \[\Omega^{C_2}_{* + (2n+1)\alpha  - (2n+1)} \cong \Omega^{C_2}_{*+2n\alpha  - 2n}/q_1.\]
\end{enumerate}
\end{theorem}

Next, we develop our theory of geometric orientations, which illuminates the relationship between geometric cobordism, $C_2$-equivariant complex orientations, and $C_2$-equivariant formal group laws. Since complex orientations are represented by maps from $MU_{C_2}$, it is natural to ask: what structure on a commutative ring $C_2$-spectrum $E_{C_2}$ is determined by a ring spectrum map $\Omega_{C_2} \to E_{C_2}$? We propose the following definition, which includes several flatness hypotheses which provide us with necessary algebraic control.

\begin{definition} \label{quasiorientation}
Suppose $E_{C_2}$ is a commutative ring $C_2$-spectrum. We say a ring $C_2$-spectrum map $\Omega_{C_2} \to E_{C_2}$ is a {\it geometric orientation} of $E_{C_2}$  if 
\begin{enumerate}
\item the transfer $\text{tr}_e^{C_2}:E_* \to E^{C_2}_*$ is injective, and 
\item $\tau \in \Omega^{C_2}_\diamond$ maps to a non-zero divisor in $E^{C_2}_\diamond$.
\end{enumerate}
\end{definition}
If we have specified such a map $\Omega_{C_2} \to E_{C_2}$, we say $E_{C_2}$ is {\it geometrically oriented}. There are many interesting examples of geometrically oriented $C_2$-spectra.
\begin{proposition} \label{examples}
The following $C_2$-spectra are geometrically oriented.
\begin{enumerate}
\item The Eilenberg-Maclane spectrum $H \underline{R}_{C_2}$ associated to a commutative ring $R$ with no $2$-torsion.
\item The connective cover $k_{C_2}$ of $C_2$-equivariant $K$-theory.
\item The geometric cobordism spectrum $\Omega_{C_2}$.
\end{enumerate}
\end{proposition}

The following result explains how our theory of geometric orientations relates to thom isomorphisms for $C_2$-equivariant vector bundles.

\begin{proposition}
Suppose $E_{C_2}$ is a geometrically oriented $C_2$-spectrum. If $\psi \to X/C_2$ is a complex vector bundle over the orbits of a $C_2$-space $X$, and $\xi = p^*\psi$ is the pullback of $\psi \to X/C_2$ along the projection map $p:X \to X/C_2$, then there is a thom isomorphism
\[
E_{C_2}^*(X) = \widetilde{E}_{C_2}^{*+2\dim \xi}(X^\xi).
\]
\end{proposition}

In section \ref{thomspectra}, we prove that inverting the element $\tau \in \Omega^{C_2}_\diamond$ determines an equivalence $\Omega_{C_2}[1/\tau] \simeq MU_{C_2}$. 
Because of this equivalence, we can associate to any geometrically oriented $C_2$-spectrum $E_{C_2}$ the complex oriented $C_2$-spectrum $\widehat{E}_{C_2}=E_{C_2}[1/\tau]$. Moreover, the coefficients of $\widehat{E}^{C_2}_*$ can be identified as 
\[
\widehat{E}^{C_2}_* \cong E^{C_2}_\diamond / (\tau - 1).
\]
This suggests that the fundamental algebraic invariant of a geometrically oriented $C_2$-spectrum $E_{C_2}$ is its extended coefficient ring $E^{C_2}_\diamond$, since calculating this ring allows us to determine the associated complex oriented $C_2$-spectrum $\widehat{E}_{C_2}$. We illustrate the general theory by calculating the extended coefficient ring and stabilization of the geometrically oriented $C_2$-spectra $H \underline{R}_{C_2}$, $k_{C_2}$, and $\Omega_{C_2}$ from Proposition \ref{examples}

\begin{theorem}\label{stable} We calculate the extended coefficient ring and stabilization of the geometrically oriented $C_2$-spectra $H\underline{R}_{C_2}$, $k_{C_2}$, and $\Omega_{C_2}$ below.
\begin{enumerate} 
\item 
If $R$ is a commutative ring with no $2$-torsion, then the Eilenberg Maclane spectrum $H \underline{R}_{C_2}$ associated to the constant $C_2$-Mackey functor $\underline{R}$ is geometrically oriented. The extended coefficient ring of $H \underline{R}_{C_2}$ is 
\[
H \underline{R}^{C_2}_\diamond = R[\mu,\tau]/(2\mu)
\] 
where $|\mu| = -\sigma$ and $|\tau| = 2-\sigma$. The stabilization of $H \underline{R}_{C_2}$ is  Borel cohomology with coefficients in $R$,
\[
H \underline{R}_{C_2}[1/\tau] \simeq F(EC_{2+},HR).
\]
\item The connective cover $k_{C_2}$ of $C_2$-equivariant $K$-theory is geometrically oriented. The extended coefficient ring of $k_{C_2}$ is 
\[
k^{C_2}_\diamond = \dfrac{R(C_2)[ v , \mu, \tau]}{\begin{matrix} \tau(\sigma - 1) = v\mu\\
\mu(\sigma + 1) = 0 \end{matrix} } 
\]
where $|v| = 2$, $|\mu| = -\sigma$, and $|\tau| = 2 - \sigma$. The stabilization of $k_{C_2}$ is Greenlees' (\cite{Greenlees2}) equivariant connective $K$-theory
\[
k_{C_2}[1/\tau] \simeq ku_{C_2}.
\]
\item The geometric cobordism spectrum $\Omega_{C_2}$ is geometrically oriented. The extended coefficient ring of $\Omega_{C_2}$ is 
\begin{align*}
\Omega^{C_2}_\diamond & =  \dfrac{\Omega^{C_2}_*[\mu,\tau]}{\begin{matrix}\tau(d_{i,j}-c_{i,j}) = \mu d_{i,j+1} \\ \tau(q_j - p_j) = \mu q_{j+1}\end{matrix} }\hspace{0.3in} i \geq 1 \text{ and }j \geq 0,
\end{align*}
where $|\mu| = -\sigma$ and $|\tau| = 2 - \sigma$. The stabilization of $\Omega_{C_2}$ is the stable complex cobordism spectrum,
\[
\Omega_{C_2}[1/\tau] \simeq MU_{C_2}.
\]
\end{enumerate}
\end{theorem}

Next, we develop the algebraic side of our theory. While complex oriented $C_2$-spectra determine $C_2$-equivariant formal group laws, we demonstrate that geometrically oriented $C_2$-spectra determine {\it filtered} $C_2${\it -equivariant formal group laws}. This is the main algebraic definition of the present paper, and is an extension of the notion of $C_2$-equivariant formal group laws as defined in \cite{CGK1}.

Recall that a (homological) $C_2$-equivariant formal group law over a commutative ring $A$ consists, in particular, an $A$-Hopf algebra $D$, and that if $E_{C_2}$ is a complex oriented $C_2$-spectrum, then $(A,D) = (E^{C_2}_*,E^{C_2}_*(\mathbf{CP}^\infty_{C_2}))$ carries the structure of a $C_2$-equivariant formal group law. If $E_{C_2}$ is a geometrically oriented $C_2$-spectrum with stabilization $E_{C_2}[1/\tau] \simeq \widehat{E}_{C_2}$, then the $C_2$-spectra $\widehat{E}_{C_2}$ and $\widehat{E}_{C_2} \wedge \mathbf{CP}^{\infty}_{C_2+}$ are filtered by certain $RO(C_2)$-graded suspensions of $E_{C_2}$ and $E_{C_2} \wedge  \mathbf{CP}^{\infty}_{C_2+}$, respectively. On the algebraic side, this is reflected in a filtration of the $C_2$-equivariant formal group law $(\widehat{E}^{C_2}_* , \widehat{E}^{C_2}_*(\mathbf{CP}^\infty_{C_2}))$. This filtration is structurally rich when viewed in terms of certain geometrically defined $\widehat{E}^{C_2}_*$-module bases of $\widehat{E}^{C_2}_*(\mathbf{CP}^\infty_{C_2})$. For any $m,n \geq 0$, we define 
\begin{align*} 
\Pi_{m+n\sigma} & = [\mathbf{CP}(m+n\sigma) \longrightarrow  \mathbf{CP}^{\infty}_{C_2}] \in  MU^{C_2}_*(\mathbf{CP}^\infty_{C_2}).
\end{align*}
Since any $C_2$-equivariant formal group law $(A,D)$ is equipped with a map 
\[ (MU^{C_2}_*,MU^{C_2}_*(\mathbf{CP}^{\infty}_{C_2})) \to (A,D), \]
 this determines elements $\Pi_{m+n\sigma} \in D$ for any $C_2$-equivariant formal group law $(A,D)$. We prove in section \ref{changeofbasis} that the elements $\Pi_{\rho_1 + \cdots + \rho_i}$ associated to a complete flag $(\rho_i)_{i=1}^\infty$ form a free $A$-module basis for $D$. We will now define a filtered $C_2$-equivariant formal group law, which axiomatizes the properties of the filtrations 
\[ \begin{tikzcd} 
E^{C_2}_* \subset  E^{C_2}_{*+|\sigma| -\sigma} \subset \cdots \subset \widehat{E}^{C_2}_*
\end{tikzcd} \]
and
\[ \begin{tikzcd} 
E^{C_2}_*(\mathbf{CP}^\infty_{C_2}) \subset E^{C_2}_{*+|\sigma|-\sigma}(\mathbf{CP}^\infty_{C_2}) \subset \cdots \subset \widehat{E}^{C_2}_*(\mathbf{CP}^\infty_{C_2}).
\end{tikzcd} \]

\begin{definition} \label{filtered}
A {\it filtered $C_2$-equivariant formal group law} $(F_\bullet A,F_\bullet D)$ consists of a $C_2$-equivariant formal group law $(A,D)$, and filtrations $F_\bullet A$ of $A$ and $F_\bullet{D}$ of $D$, such that 
\begin{enumerate}
\item $\text{Im}(\Omega^{C_2}_* \to A) \subseteq F_0A$,
\item $F_nA$ is generated over $F_0A$ by $1,\dots,u^n$, and 
\item For any complete flag $(\rho_i)_{i=1}^\infty$, we have
\[
F_nD = \left\{ \sum a_i \Pi_{\rho_1 + \cdots + \rho_i} \in D : a_i \in F_{n + \ell_i}A \right\},
\]
where $\ell_i$ is the number of copies of $\sigma$ in $(\rho_1 + \cdots + \rho_{i-1}) \rho_i^{-1}$.
\end{enumerate} 
\end{definition}
Our definition is motivated by the following fact.
\begin{theorem}
If $E_{C_2}$ is a geometrically oriented $C_2$-spectrum with stabilization $\widehat{E}_{C_2} = E_{C_2}[1/\tau]$, then the pair $ (F_\bullet \widehat{E}^{C_2}_*,F_\bullet \widehat{E}^{C_2}_*(\mathbf{CP}^\infty_{C_2}))$ defined by
\[
F_n\widehat{E}^{C_2}_* = E^{C_2}_{*+|n\sigma|-n\sigma}\text{, and}
\]
\[
F_n\widehat{E}^{C_2}_*(\mathbf{CP}^\infty_{C_2}) = E^{C_2}_{*+|n\sigma|-n\sigma}(\mathbf{CP}^\infty_{C_2})
\]
is a filtered $C_2$-equivariant formal group law.
\end{theorem}
Next, we prove the following universality statement, which asserts that the structure of a filtered $C_2$-equivariant formal group law $(F_\bullet A, F_\bullet D)$  is completely determined by $F_0A$ and the filtration on the universal equivariant formal group law $(MU^{C_2}_*,MU^{C_2}_*(\mathbf{CP}^\infty_{C_2}))$.

\begin{theorem} \label{universality}
If $(F_\bullet A, F_\bullet D)$ is a filtered $C_2$-equivariant formal group law, then 
\begin{align*}
F_nA & = F_nMU^{C_2}_* \cdot F_0A\text{, and }\\
F_nD & = F_nMU^{C_2}_* \cdot F_0D.
\end{align*}
\end{theorem}
Finally, we analyze the classes $\pi_{m+n\sigma} = [\mathbf{CP}(m+n\sigma)] \in \Omega^{C_2}_*$, which play a prominent role in our theory of filtered $C_2$-equivariant formal group laws. In the following theorem, we give an algebraic characterization of the elements $\pi_{m+n\sigma}$, and identify these classes in terms of our generators of $\Omega^{C_2}_*$ for some small values of $m$ and $n$. 

\begin{proposition}
The composite 
\[\Omega^{C_2}_* \to MU_* \to H_*(MU) =  \mathbb{Z}[b_i : i \geq 1]\]
maps $\pi_{m+n\sigma}$ to 
\[
(m+n)m_{m+n-1} = \text{coeff}_{x^{m+n-1}} \dfrac{1}{(1+b_1x+b_2x^2+\cdots)^{m+n}},
\]
and the composite 
\[
\Omega^{C_2}_* \to \Phi MU^{C_2}_* \to \widetilde{H}_*(MU \wedge BU_+) =  \mathbb{Z}[b_i,b_i': i \geq 1][u^{\pm 1}]
\]
maps $\pi_{m+n\sigma}$ to the sum
\begin{align*}
& \left(\text{coeff}_{x^m} \dfrac{1}{(1 + b_1x + b_2 x^2 + \cdots )^{m}(1+b'_1x + b_2'x^2 + \cdots)^n}\right)u^{-n}\\
+ & \left(\text{coeff}_{x^n} \dfrac{1}{(1 + b_1x + b_2 x^2 + \cdots )^{n}(1+b'_1x + b_2'x^2 + \cdots)^m}\right)u^{-m}.
\end{align*}
This characterizes the elements $\pi_{m+n\sigma}$.
\end{proposition}

\begin{example}
For several small values of $m$ and $n$, we express the element  \[\pi_{m+n\sigma} = [\mathbf{CP}(m+n\sigma)] \in \Omega^{C_2}_*\] in terms of our generators $d_{i,j},q_j \in \Omega^{C_2}_*$:
\begin{align*}
\pi_{1+\sigma} & = -q_2\\
\pi_{2+\sigma} & = d_{1,0} - a_{1,1}q_2\\
\pi_{2+2\sigma} & = 4d_{1,1}+2q_4 - 2q_2q_3 - q_2^3 + (6b_1^3 -18 b_1b_2 + 6b_3)q_1.
\end{align*}
\end{example}




\section{Equivariant cobordism}\label{thomspectra}
Our goal in this section is to calculate the ring $\Omega^{C_2}_*$ of stably almost complex $C_2$-manifolds with involution, and the good range $\Omega^{C_2}_\diamond$ of the $RO(C_2)$-graded coefficient ring of $\Omega_{C_2}$. In section \ref{equivariantthomspectra}, we prove that the map $\Omega_{C_2} \to MU_{C_2}$ induces an equivalence $\Omega_{C_2}[1/\tau] \simeq MU_{C_2}$, where $\tau \in \Omega^{C_2}_\diamond$ is an element in the $RO(C_2)$-graded coefficients of $\Omega_{C_2}$. This is a spectrum-level analogue of a result of Brocker and Hook in the unoriented case \cite{brockerhook}. We go on to prove that the $\ZZ$-graded coefficient ring of $\Omega_{C_2}$ coincides with the geometric cobordism ring of stably almost complex $C_2$-manifolds. In section \ref{pullbacks} we review known facts about the stable cobordism ring $MU^{C_2}_*$, and calculate a new presentation of this ring, which is essential to our calculation of $\Omega^{C_2}_*$ and $\Omega^{C_2}_\diamond$ in section \ref{extendedsection}.

\subsection{Equivariant Thom spectra}\label{equivariantthomspectra} In this section we prove that the stable cobordism spectrum $MU_{C_2}$ can be obtained from the geometric cobordism spectrum $\Omega_{C_2}$ by inverting an element $\tau$ in the $RO(C_2)$-graded coefficient ring of $\Omega_{C_2}$. Then, we prove that the $\ZZ$-graded coefficient ring $\Omega_{C_2}$ coincides with the geometric cobordism ring of stably almost complex $C_2$-manifolds with involution. 

Recall that the $C_2$-spectrum $\Omega_{C_2}$ assigns to the subrepresentation $\mathbf{C}^\sigma \subset \mathbf{C}^{\infty,\infty}$ of our chosen $C_2$-universe $\mathbf{C}^{\infty,\infty}$ the $C_2$-space
\[\Omega_{C_2}(\mathbf{C}^\sigma) = \text{Thom}\left( \xi^{\mathbf{C}^\infty}(\mathbf{C}^\sigma) \to \text{Gr}^{\mathbf{C}^\infty}(\mathbf{C}^\sigma) \right). \]
There is a point $* \in \text{Gr}^{\mathbf{C}^\infty}(\mathbf{C}^\sigma)$ corresponding to the line 
\[\mathbf{C} = \text{span}(0,1,0,0,\dots) \subset \mathbf{C}^\sigma \oplus \mathbf{C}^\infty,\]
and the inclusion $* \to \text{Gr}^{\mathbf{C}^\infty}(\mathbf{C}^\sigma)$ is covered by a vector bundle map $\mathbf{C} \to \xi^{\mathbf{C}^{\infty}}(\mathbf{C}^\sigma)$. Applying $\text{Thom}(-)$ to this vector bundle map gives us a map $S^2 \to \Omega_{C_2}(\mathbf{C}^\sigma)$, whose homotopy class we call the {\it weak thom class}, denoted
\[
\tau = [ S^2 \to \Omega_{C_2}(\mathbf{C}^\sigma)] \in \Omega^{C_2}_{2-\sigma}.
\]
The desired equivalence $\Omega_{C_2}[1/\tau] \simeq MU_{C_2}$ is a consequence of the following lemma, which identifies the defining diagram of $\Omega_{C_2}[1/\tau]$ with the geometrically defined filtration 
\[
\Omega_{C_2} = MU_{C_2}^{\mathbf{C}^{\infty}} \subset MU_{C_2}^{\mathbf{C}^{\infty+\sigma}}  \subset MU_{C_2}^{\mathbf{C}^{\infty+2\sigma}} \subset \cdots \subset MU_{C_2}.^{\mathbf{C}^{\infty+\infty\sigma}}  = MU_{C_2} 
\]

\begin{lemma}
For each $n \geq 0$, there is an equivalence 
\[ \Sigma^{n\sigma - |n\sigma|}\Omega_{C_2} \simeq MU^{\mathbf{C}^{\infty,n}}_{C_2}\]
 such that the following diagram commutes in $\text{Ho}(C_2 \text{Sp})$.
\begin{equation} \label{equivalencediagram} \begin{tikzcd}
\Omega_{C_2} \ar[r,"\tau"]  \ar[d,"="] & \Sigma^{\sigma - |\sigma|} \Omega_{C_2} \ar[r,"\tau"] \ar[d,"\simeq"] & \Sigma^{2\sigma - |2\sigma|} \Omega_{C_2} \ar[r,"\tau"] \ar[d,"\simeq"] & \dots \\
MU^{\mathbf{C}^\infty}_{C_2}  \ar[r]  & MU^{\mathbf{C}^{\infty,1}}_{C_2}  \ar[r]  & MU^{\mathbf{C}^{\infty,2}}_{C_2}  \ar[r] & \dots \\
\end{tikzcd} 
\end{equation}
\end{lemma}
\begin{proof}
For any $n \geq 0$ and $C_2$-representation $V$, the embedding
\[V \oplus \mathbf{C}^{n\sigma} \oplus  \mathbf{C}^{\infty} \cong V \oplus 0 \oplus \mathbf{C}^{\infty,n} \to V \oplus \mathbf{C}^n \oplus \mathbf{C}^{\infty,n}\]
induces a homotopy equivalence
\[ \begin{tikzcd} 
\text{Gr}^{\mathbf{C}^\infty}(V\oplus \mathbf{C}^{n\sigma}) \ar[r,"\simeq"] & \text{Gr} ^{\mathbf{C}^{\infty,n}}(V\oplus \mathbf{C}^n).
\end{tikzcd} \]
Applying $\text{Thom}(-)$ to the induced map of vector bundles yields a homotopy equivalence
\[ \begin{tikzcd}
\Omega_{C_2}(V\oplus \mathbf{C}^{n\sigma}) \ar[r,"\simeq"] &  MU_{C_2}^{\mathbf{C}^{\infty,n}}(V \oplus \mathbf{C}^n).
\end{tikzcd} \]
The spectrum $V \mapsto MU^{\mathbf{C}^\infty}_{C_2}(V \oplus \mathbf{C}^{n\sigma})$ is a model for $\Sigma^{n\sigma}\Omega_{C_2}$, and the spectrum $V \mapsto MU^{\mathbf{C}^{\infty,n}}_{C_2}( V \oplus \mathbf{C}^{n})$ is a model for $\Sigma^{2n}MU^{\mathbf{C}^{\infty,n}}_{C_2}$, so these maps determine an equivalence $\Sigma^{n\sigma} \Omega_{C_2} \simeq \Sigma^{|n\sigma|} MU^{\mathbf{C}^{\infty,n}}_{C_2}$. Smashing with $S^{-|n\sigma|}$ yields the desired equivalence 
\[\Sigma^{n\sigma - |n\sigma|} \Omega_{C_2} \simeq MU_{C_2}^{\mathbf{C}^{\infty,n}}.\]
 The homotopy commutativity of diagram \ref{equivalencediagram} follows from the homotopy commutativity of the following diagram of based $C_2$-spaces.
\[
\begin{tikzcd}[column sep = 18.0]
 \Omega_{C_2}(V \oplus \mathbf{C}^{n\sigma} ) \wedge S^2 \ar[dd,swap,"\simeq  \wedge 1"] \ar[r,"1 \wedge \tau"] & \Omega_{C_2}(V \oplus \mathbf{C}^{n\sigma} ) \wedge \Omega_{C_2}(\mathbf{C}^\sigma)\ar[r] & MU_{C_2}^{\mathbf{C}^\infty \oplus \mathbf{C}^\infty}(V \oplus \mathbf{C}^{(n+1)\sigma}) \ar[d,"\simeq"] \\
 & & \Omega_{C_2}(V \oplus \mathbf{C}^{(n+1)\sigma}) \ar[d,"\simeq"]  \\
MU_{C_2}^{\mathbf{C}^{\infty,n}}(V \oplus \mathbf{C}^n) \wedge S^2 \ar[r,"i \wedge 1"] & MU_{C_2}^{\mathbf{C}^{\infty,n+1}}(V \oplus \mathbf{C}^n) \wedge S^2 \ar[r] & MU_{C_2}^{\mathbf{C}^{\infty,n+1}}(V \oplus \mathbf{C}^{n+1})  \\
\end{tikzcd}
\]
 \end{proof}

\begin{corollary}
The map $\Omega_{C_2} \to MU_{C_2}$ induces an equivalence
\[
\Omega_{C_2}[1/\tau] \simeq MU_{C_2}.
\] 
\end{corollary} 
\begin{proof}
\begin{align*}
\Omega_{C_2}[1/\tau] & = \text{hocolim} \left(
\Omega _{C_2}\to \Sigma^{\sigma - |\sigma|} \Omega_{C_2} \to \Sigma^{2\sigma - |2\sigma|} \Omega_{C_2} \to \dots
\right)\\
& \simeq \text{hocolim} \left(
MU_{C_2}^{\mathbf{C}^\infty} \to MU_{C_2}^{\mathbf{C}^{\infty,1}} \to M_{C_2}U^{\mathbf{C}^{\infty,2}}  \to \dots
\right)\\
& = MU_{C_2}
\end{align*}
\end{proof}

The final goal of this section is to identify the $\ZZ$-graded coefficient ring of $\Omega_{C_2}$ with the geometric cobordism ring of stably almost complex $C_2$-manifolds with involution, which we temporarily denote $\Omega^{C_2,\text{geo}}_*$. We refer the reader to \cite{Hanke} for a detailed discussion of the geometric cobordism ring $\Omega^{C_2,\text{geo}}_*$ and the equivariant Pontrjagin-Thom construction. We mention only that an element $[M] \in \Omega^{C_2,\text{geo}}_n$ is represented by an $n$-dimensional  $C_2$-manifold $M$ with a complex structure on $TM \oplus \underline{\mathbf{R}}^k$ for some $k \geq 0$. The equivariant Pontrjagin-Thom construction determines a ring map $\Omega^{C_2,\text{geo}}_* \to \Omega^{C_2}_*$, and we will show that this is an isomorphism.

\begin{proposition}
The equivariant Pontrjagin-Thom map 
\[
\Omega^{C_2,\text{geo}}_* \to \Omega^{C_2}_*
\]
is an isomorphism.
\end{proposition}

\begin{proof}
Geometric cobordism defines a $\ZZ$-graded homology theory on $C_2$-spaces, and the Pontrjagin-Thom construction defines a natural transformation from geometric cobordism to $\Omega_{C_2}$-homology. We can evaluate each homology theory on the cofiber sequence 
\[
EC_{2+} \to S^0 \to \widetilde{EC}_2,
\]
which yields the following diagram whose rows are exact.
\[
\begin{tikzcd}
\dots \ar[r] & [-10 pt] \Omega^{C_2,\text{geo}}_*(EC_{2}) \ar[r] \ar[d]  & [-10 pt] \Omega^{C_2,\text{geo}}_* \ar[r] \ar[d] & \widetilde{\Omega}^{C_2,\text{geo}}_*(\widetilde{EC}_2) \ar[d] \ar[r] & [-10 pt]\Omega^{C_2,\text{geo}}_{*-1}(EC_{2}) \ar[r] \ar[d] & [-10 pt]\dots \\
\dots \ar[r] & [-10 pt] \Omega^{C_2}_*(EC_{2}) \ar[r]  &  [-10 pt] \Omega^{C_2}_* \ar[r] & \widetilde{\Omega}^{C_2}_*(\widetilde{EC}_2)  \ar[r] & [-10 pt] \Omega^{C_2}_{*-1}(EC_{2}) \ar[r]  & [-10 pt] \dots 
\end{tikzcd}
\]
The map $\Omega^{C_2,\text{geo}}_*(EC_{2}) \to \Omega^{C_2}_*(EC_{2})$ is an isomorphism since equivariant transversality holds in the presence of  free group actions. By the 5-lemma, it suffices to prove that the map $\widetilde{\Omega}^{C_2,\text{geo}}_*(\widetilde{EC}_{2}) \to \widetilde{\Omega}^{C_2}_*(\widetilde{EC}_{2})$ is also an isomorphism. The geometric fixed point ring $\widetilde{\Omega}^{C_2\text{,geo}}_*(\widetilde{EC}_2) $ is isomorphic to the ring of ``local fixed point data"
\[
\widetilde{\Omega}^{C_2,\text{geo}}_*(\widetilde{EC}_2) \cong \bigoplus_{n \geq 0} MU_{*-2n}(BU(n)).
\]
Elements of this ring are pairs $(F,\xi)$ where $F$ is a manifold and $\xi$ is a vector bundle over $F$. 
The isomorphism $\widetilde{\Omega}^{C_2,\text{geo}}_*(\widetilde{EC}_2) \cong \bigoplus_{n \geq 0} MU_{*-2n}(BU(n))$ takes $[M \to \widetilde{EC}_2]$ to 
\[\bigoplus_i [M^{C_2}_i \to BU(n_i)],\]
 where $M^{C_2}_i$ are the components of the fixed point locus $M^{C_2} \subset M$, and the map $M^{C_2}_i \to BU(n_i)$ classifies the normal bundle $\nu\mid_{M^{C_2}_i}^M$. On the other hand, one can calculate $\widetilde{\Omega}^{C_2}_*(\widetilde{EC}_2) = \Phi \Omega^{C_2}_*$ by calculating the geometric fixed point spectrum $\Phi \Omega^{C_2} = (\Omega_{C_2} \wedge \widetilde{EC}_2)^{C_2}$ of $\Omega_{C_2}$. This can be done at the level of $C_2$-spaces, since $\Omega_{C_2}$ comes from an inclusion $C_2$-prespectrum (see the proof of Lemma \ref{geometricfixedpoints} for further detail). We find that
\[\Phi \Omega^{C_2} \simeq \bigvee_{n \geq 0} \Sigma^{2n} MU \wedge BU(n)_+,\]
so by explicit computation, the map $\widetilde{\Omega}^{C_2,\text{geo}}_*(\widetilde{EC}_{2}) \to \widetilde{\Omega}^{C_2}_*(\widetilde{EC}_{2})$ is an isomorphism.
\end{proof}

\subsection{Calculation of $MU^{C_2}_*$} \label{pullbacks}

In this section we review known facts about the stable cobordism ring $MU^{C_2}_*$, and calculate a new presentation of this ring which will be convenient for our calculation of $\Omega^{C_2}_*$ in the following section. The ring $MU^{C_2}_*$ can be calculated using the Tate diagram. We review the construction of the Tate diagram briefly, and refer the reader to \cite{GreenleesMay2} for a more thorough exposition.  Let $EC_{2}$ be a free $C_2$-space which is non-equivariantly contractible, and consider the cofiber sequence 
\[
EC_{2+} \to S^0 \to \widetilde{EC}_{2}
\]
where the first map collapses $EC_2$ to the non-basepoint of $S^0$. The Tate diagram for $MU_{C_2}$  is
\[ \begin{tikzcd} 
MU_{C_2} \wedge EC_{2+} \ar[r] \ar[d,"\simeq"] & MU_{C_2} \ar[r] \ar[d] & EC_{2+} \wedge \widetilde{EC}_2 \ar[d] \\
 F(EC_{2+} , MU_{C_2}) \wedge EC_{2+} \ar[r]  & F(EC_{2+}, MU_{C_2}) \ar[r] & F(EC_{2+} , MU_{C_2}) \wedge \widetilde{EC}_{2+},
\end{tikzcd} \] 
where both rows are cofiber sequences of $C_2$-spectra, and the vertical maps are obtained from the collapse map $EC_{2+} \to S^0$ by applying $F(-,MU_{C_2})$. We are primarily interested in the right hand square, which at the level of coefficients is 
\[\begin{tikzcd} 
MU^{C_2}_* \ar[r] \ar[d] & \Phi MU^{C_2}_* \ar[d] \\
MU^{hC_2}_* \ar[r] & MU^{tC_2}_*.
\end{tikzcd} \]
The upper right, bottom left, and bottom right corners are the coefficients of the geometric, homotopy, and Tate fixed points of $MU_{C_2}$, respectively. In \cite{Kriz}, Kriz proves that this square is a pullback of rings, and identified the Tate square for $MU^{C_2}_*$ with
\[
\begin{tikzcd}
MU^{C_2}_* \ar[r] \ar[d] & MU_*(BU)[u^{\pm 1}] \ar[d] \\
\dfrac{MU_*[[u]]}{[2]u}\ar[r] & \dfrac{MU_*((u))}{[2]u},
\end{tikzcd}
\]
where $|u| = -2$ and 
\begin{equation}\label{twoseries} [2]u = \sum_{i,j \geq 0} a_{i,j}u^{i+j} = p_0 + p_1u + p_2 u ^2 + \dots \end{equation}

is the $2$-series of the universal formal group law over $MU_*$. The vertical map $MU_*(BU)[u^{\pm 1}] \to MU_*((u))/[2]u$ is given as follows. We know that $MU_*(BU(1)) = MU_* \{ \beta_0 , \beta_1 , \dots \}$ where $\{ \beta_0 , \beta_1, \dots\}$ is dual to $\{1 , x , x^2 , \dots \} \subset MU^*(BU(1)) = MU^*[[x]]$, and that
\[
MU_*(BU) = MU_*[b_1',b_2',\dots] 
\]
where $b_i'$ is the image of $\beta_i \in MU_*(BU(1))$ under the map induced by the inclusion $BU(1) \to BU$. We use the symbols $b_i'$ to distinguish these elements from the coefficients $b_i$ of the exponential series of the universal logarithm. The vertical map $MU_*[b_i'][u^{\pm 1}] \to MU_*[[u]]/[2]u$ is determined by 
\[
ub_i' \mapsto \sum_{j \geq 0} a_{i,j}u^j,
\]
which is the coefficient of $x^i$ in $F_{MU}(u,x) = u + ub_1' x + ub_2' x^2 + \cdots \in MU_*[[u,x]]$.

For reasons that will become clear in the next section, it is convenient for us to use a different presentation of $\Phi MU^{C_2}_*$ in our calculation. More precisely, we will choose a new polynomial basis for $\Phi MU^{C_2}_*$, and emulate Strickland's calculation of $MU^{C_2}_*$ using this new polynomial basis. For any $ i \geq 0$, define $d_i \in MU_*[b_i'][u^{\pm 1}]$ to be such that $d_0 + d_1 x + d_2x^2 + \cdots \in MU_*[b_i'][u^{\pm 1}] [[x]]$ is the multiplicative inverse of $u + ub_1' x +  ub_2' x^2 + \cdots \in MU_*[b_i'][u^{\pm 1}][[x]]$, i.e. such that 
\begin{equation}
(u+ ub'_1 x + ub'_2 x^2 + \cdots )(1 + d_1x + d_2 x^2 + \cdots )  = 1 .
\end{equation}\label{definitionofds}
For instance, we have $d_0=u^{-1}$, $d_1 = -u^{-1}b_1'$, and $d_2 = -u^{-1}(b_1')^2 +u^{-1}b_2'$. Under the identification $\Phi MU^{C_2}_* = MU_*[d_0^{\pm 1}, d_1, d_2 , \dots]$, the map to the Tate fixed points $MU^{tC_2}_* = MU_*((u))/[2]u$ is given by 
\[
d_i \mapsto \sum_{j \in \ZZ} c_{i,j}u^j
\]
which is the coefficient of $x^i$ in $\dfrac{1}{F(u,x)} \in MU_*((u))[[x]]$.

We can now use the pullback square
\[\begin{tikzcd} 
MU^{C_2}_* \ar[r] \ar[d] & MU_*[u^{\pm 1}, d_1, d_2, \dots ] \ar[d]  & d_i \ar[d,mapsto] \\
MU_*[[u]]/[2]u \ar[r] & MU_*((u))/[2]u & \sum_{j \in \ZZ} c_{i,j} u^j.
\end{tikzcd} \]
to calculate $MU^{C_2}_*$. For $i \geq 1$ and $j \geq 0$, let $u,d_{i,j},q_j$ be variables with $|u| = -2$, $|d_{i,j}| = 2(i+j+1)$, and $|q_j| = 2j-2$. Let $J \subset MU_*[u,d_{i,j},q_j ]$ be the ideal generated by the relations
\begin{align*}
d_{i,j}  - c_{i,j} & = ud_{i,j+1}\\
q_j  - p_j & = uq_{j+1}\\
q_0 & = 0
\end{align*}
for $i \geq 1$ and $j \geq 0$. Define $\phi:MU_*[u,d_{i,j},q_j] / J \to MU_*[u^{\pm 1},d_i]$ by 
\begin{equation}\label{generatorimages}
\phi(d_{i,j}) = u^{-j}d_i - \sum_{\ell < j} c_{i,\ell} u^{\ell - j}, \hspace{0.2in}
\phi(q_j) = - \sum_{\ell <j } p_\ell u^{\ell - j } ,
\end{equation}
and $\phi(u) = u$. Define  $\chi: MU_*[u,d_{i,j}, q_j] / J \to MU_*[[u]]/[2]u$ by 
\begin{align*}
\chi(d_{i,j}) = \sum_{\ell \geq 0} c_{i,j+\ell} u^\ell, \hspace{0.2in} 
\chi(q_j) = \sum_{\ell \geq 0} d_{j+\ell} u^\ell,
\end{align*}
and $\chi(u) = u$.

\begin{proposition}
There is an isomorphism of graded rings 
\[
MU^{C_2}_* \cong MU_*[d_{i,j},q_j,u]/J.
\]
\end{proposition}

\begin{proof}
Since the Tate square for $MU^{C_2}_*$ is a pullback of rings, it suffices to show that
\begin{equation}\label{homotopicalsquare}
 \begin{tikzcd}
R \ar[r,"\phi"] \ar[d,swap,"\chi"] & MU_*[u^{\pm 1} , d_1,d_2, \dots ] \ar[d,"\psi"] \\
\dfrac{MU_*[[u]]}{[2]u} \ar[r] & \dfrac{MU_*((u))}{[2]u}
\end{tikzcd} 
\end{equation}
commutes and is a pullback of rings, where $R = MU_*[u,d_{i,j},q_j] / J$. That the diagram commutes is easily verified from the definitions of $\phi$, $\chi$, and $\psi$. To prove that the square is a pullback, it suffices to show that
\begin{enumerate}
\item $\phi$ is an isomorphism after inverting $u$,
\item $\chi$ is an isomorphism after $u$-completion, and 
\item $R$ has bounded $u$-torsion. 
\end{enumerate} 
The proofs of these facts are direct analogues of the arguments in (\cite{Strickland}, Theorem 4), but we include them for completeness. \\

Proof of 1.: Since $\phi(u) = u$ is a unit in $MU_*[u^{\pm 1} , d_i]$, we have an induced map 
\[
\overline{\phi}:u^{-1} R \to MU_*[u^{\pm 1},d_i].
\]
and its inverse is given by $u \mapsto u$, and $d_i \mapsto d_{i,0}+ \sum_{j<0} c_{i,j}u^j$.\\

Proof of 2.: Since $MU_*[[u]]/[2]u$ is complete at $u$, we have an induced map 
\[\widehat{\chi}:R^\wedge_u \to \dfrac{MU_*[[u]]}{[2]u}.\]
If we define $\rho: MU_*[[u]] \to R^\wedge_u$ by $u \mapsto u$, then the composite $\widehat{\chi} \circ \rho$ is the quotient map $MU_*[[u]] \to MU_*[[u]]/[2]u$, so $\widehat{\chi}$ is surjective. By induction on $m \geq 1$, we have the equalities 
\begin{align*}
d_{i,j} - \sum_{ \ell = 0}^{m-1} c_{i,j+\ell}u^\ell & = d_{i,j+m} u^m\\
q_j  - \sum_{\ell = 0}^{m-1} p_{j+\ell} u^\ell & =q_{j+m} u^m
\end{align*}
in $R$. This implies the equalities $d_{i,j} = \sum_{ \ell \geq 0} c_{i,j+\ell} u^\ell$ and $q_j = \sum_{\ell \geq 0} p_{j+ \ell} u^\ell$ in $R^\wedge_u$, so $\rho$ is surjective. Since $q_0 = \sum_{\ell \geq 0} p_\ell u^\ell = [2]u = 0$ in $R^\wedge_u$, $\rho$ factors through a map $\overline{\rho}: MU_*[[u]]/[2]u\to R^\wedge_u$. Since $\overline{\rho}$ is surjective and $\widehat{\chi} \circ \overline{\rho} = 1$, we deduce that $\overline{\rho}$ is an isomorphism with inverse $\widehat{\chi}$.\\

Proof of 3.: It suffices to prove that $u$ is a regular element of $R/q_1$. This is true since we can write $R/q_1 = \varinjlim C_k $where $C_k$ is the ring 
\[ MU_*[u,d_{i,k},q_k : i \geq1]/(q_0, \sum_{\ell = 0}^{k-1} p_{\ell +1} u^\ell + q_k u^k),\]
and $u$ is a regular element of each $C_k$.
\end{proof}



\subsection{Calculation of $\Omega^{C_2}_*$ and $\Omega^{C_2}_\diamond$} \label{extendedsection}
In this section we calculate the geometric cobordism ring $\Omega^{C_2}_*$, as well as the extended coefficient ring $\Omega^{C_2}_\diamond$ of the $C_2$-spectrum $\Omega_{C_2}$. Recall that the inclusion of $C_2$-universes $\mathbf{C}^\infty \to \mathbf{C}^{\infty,\infty}$ induces a ring spectrum map 
\[
\Omega_{C_2} = MU_{C_2}^{\mathbf{C}^\infty} \to  MU_{C_2}^{\mathbf{C}^{\infty,\infty}} = MU_{C_2}.
\]
This induces a map on geometric fixed points, which leads to the diagram 
\begin{equation} \label{pullbacksquare} \begin{tikzcd}
\Omega^{C_2}_* \ar[r] \ar[d] & \Phi \Omega^{C_2}_* \ar[d] \\
MU^{C_2}_* \ar[r] & \Phi MU^{C_2}_*.
\end{tikzcd}
\end{equation}
\begin{lemma}
The square \ref{pullbacksquare} is a pullback of rings.
\end{lemma}

\begin{proof}
Our square sits in the diagram
\[
\begin{tikzcd}
\dots \ar[r] & [-10 pt] \Omega^{C_2}_*(EC_{2}) \ar[r] \ar[d]  & [-10 pt] \Omega^{C_2}_* \ar[r,"\phi_{\Omega}"] \ar[d] & \Phi\Omega^{C_2}_* \ar[d] \ar[r] & [-10 pt]\Omega^{C_2}_{*-1}(EC_{2}) \ar[r] \ar[d] & [-10 pt]\dots \\
\dots \ar[r] & [-10 pt] MU^{C_2}_*(EC_{2}) \ar[r]  &  [-10 pt] MU^{C_2}_* \ar[r,"\phi_{MU}"] & \Phi MU^{C_2}_* \ar[r] & [-10 pt] MU^{C_2}_{*-1}(EC_{2}) \ar[r]  & [-10 pt] \dots 
\end{tikzcd}
\]
whose rows are exact. The map $\Omega^{C_2}_*(EC_{2}) \to MU^{C_2}_*(EC_{2})$ is an isomorphism since $\Omega_{C_2}$ and $MU_{C_2}$ are split $C_2$-spectra and $\Omega_{C_2} \to MU_{C_2}$ is a non-equivariant equivalence. It is proved in \cite{Comezana} that $\Omega^{C_2}_* \to MU^{C_2}_*$ is injective, so Lemma \ref{halemma} implies that the square is a pullback.
\end{proof}

Having realized $\Omega^{C_2}_*$ as the pullback of the diagram \ref{pullbacksquare}, our next goal is to calculate
\[
\Phi \Omega^{C_2}_* \to \Phi MU^{C_2}_*,
\]
which we do in the following lemma.
\begin{lemma}\label{geometricfixedpoints}
There is a ring isomorphism 
\[
\Phi \Omega^{C_2}_* \cong MU_*[u^{-1} , d_1, d_2, \dots ]
\]
such that the following diagram commutes.
\[ \begin{tikzcd} 
\Phi \Omega^{C_2}_* \ar[r,"\cong" ] \ar[d] & MU_*[u^{-1} , d_1, d_2, \dots ] \ar[d] \\
\Phi MU^{C_2}_* \ar[r,"\cong"] & MU_*[u^{\pm 1} , d_1, d_2, \dots ]
\end{tikzcd} \] 
\end{lemma}

\begin{proof}
If $E_{C_2}$ is an inclusion $C_2$-spectrum, which means that all of the adjoint structure maps $E_{C_2}(V) \to \Omega^{W-V} E_{C_2}(W)$ are inclusions of based $C_2$-spaces, then the geometric fixed point spectrum can be calculated at the level of $C_2$-spaces using the formula
\[
(E_{C_2} \wedge \widetilde{EC}_2)^{C_2} (V) = \text{colim}_{W \supset V} \Omega^{(W-V)^{C_2}} E_{C_2}(W)^{C_2}.
\]
Since both $\Omega_{C_2}$ and $MU_{C_2}$ are inclusion $C_2$-prespectra, we can use this formula to calculate that 
\[
\Phi \Omega^{C_2} \simeq \bigvee_{n \geq 0} \Sigma^{2n}MU \wedge  BU(n)_+,
\]
\[
\Phi MU^{C_2} \simeq \bigvee_{n \in \ZZ} \Sigma^{2n}MU \wedge  BU_+,
\]
and the map $\Phi \Omega^{C_2} \to \Phi MU^{C_2}$ is induced by the composites 
\[
BU(n) \to BU \overset{i}{\longrightarrow} BU
\]
where $i$ is the map classifying the additive inverse of stable vector bundles. We have 
\begin{align*}
\Phi\Omega^{C_2}_* & \cong \bigoplus_{n \geq 0} MU_*(BU(n))u^{-n}\\
& \cong \bigoplus_{n \geq 0} MU_* \{ \beta_{i_1} \dots \beta_{i_n} : 0 \leq i_1 \leq \dots \leq i_n \} u^{-n}\\
& \cong MU_*[u^{-1}, u^{-1}b_i': i \geq 1] 
\end{align*}
and the geometric fixed point map $\Phi\Omega^{C_2}_* \to \Phi MU^{C_2}_*$ corresponds to the composite
\[
\begin{tikzcd}
MU_*[u^{-1},u^{-1}b_i'] \subset  MU_*[b_i' ][u^{\pm 1}] \ar[r,"i_*"] & MU_*[b_i' ][u^{\pm 1}] .
\end{tikzcd}
\]
 The $H$-space structure of $BU$ gives $MU_*(BU) = MU_*[b_i']$ the structure of a Hopf algebra over $MU_*$, and the map $i:BU \to BU$ induces its antipode. Since the coproduct on $MU_*(BU)$ satisfies $\Delta b_n' = \sum_{i+j = n} b_i' \otimes b_j'$, it follows that $i_*(ub_i') = d_i$ is the coefficient of $x^i$ in the formal power series 
\[
\dfrac{1}{u + ub_1'x + ub_2'x^2 + \cdots }= u^{-1} + d_1x +d_2x^2 + \cdots .
\]
We deduce that the geometric fixed point map $\Phi \Omega^{C_2}_* \to MU^{C_2}_*$ corresponds to the inclusion 
\[
MU_*[u^{-1} , d_1, d_2, \dots ] \to MU_*[u^{\pm 1}, d_1, d_2, \dots ].
\]
\end{proof} 

We can now deduce the structure of $\Omega^{C_2}_*$ from the pullback square
\[ \begin{tikzcd}
\Omega^{C_2}_* \ar[r] \ar[d] & MU_*[u^{-1},d_i] \ar[d] \\
MU^{C_2}_* \ar[r,"\phi"] & MU_*[u^{\pm 1} ,d_i] .
\end{tikzcd} \]

\begin{theorem} There is an isomorphism of graded rings 
\begin{align*}
\Omega^{C_2}_* & \cong MU_*[d_{i,j},q_j]/I
\end{align*}
where $I$ is generated by the relations  
\begin{align*}
d_{i,j+1}(d_{k,\ell} - c_{k,\ell}) & = (d_{i,j} - c_{i,j})d_{k,\ell+1}\\
d_{i,j+1}(q_\ell - p_\ell) & = (d_{i,j} - c_{i,j})q_{\ell+1}\\
q_{j+1}(q_\ell - p_\ell) & = (q_j - p_j)q_{\ell + 1}\\
q_0 &= 0
\end{align*}
for $i,k \geq 1$ and $j , \ell \geq 0$. 
\end{theorem} 
\begin{proof}
First, we claim that the map 
\[
\Omega^{C_2}_* \to MU^{C_2}_* = MU_*[u,d_{i,j}, q_j]/J
\]
identifies $\Omega^{C_2}_*$ with the $MU_*$-subalgebra of $MU_*[u,d_{i,j},q_j]/J$ generated by $d_{i,j},q_j$ for $i \geq 1$ and $ j \geq 0$. If, for any $f \in MU_*[u^{\pm 1} , d_i]$, we define $\deg_uf$ to be the highest power of $u$ that occurs in $f$, then $MU_*[u^{-1}, d_i] \subset MU_*[u^{\pm 1}, d_i]$ is the inclusion of all elements with $u$-degree $\leq 0$, so the pullback of $\phi$ along this inclusion is
\[
\Omega^{C_2}_* = \{ f \in MU^{C_2}_* : \text{deg}_u\phi(f) \leq 0\}.
\]
Recall that the ideal $J$ is generated by the relations 
\begin{align*}
d_{i,j}  - c_{i,j} & = ud_{i,j+1}\\
q_j  - p_j & = uq_{j+1}\\
q_0 & = 0
\end{align*}
for $i \geq 1$ and $j \geq 0$. 
If $f \in MU^{C_2}_*$, then using these relations, we can write  $f = f_1 + f_2$, where $f_1$ is a polynomial in $\{d_{i,j},q_j\}$, and $f_2$ is a sum of terms of the form $b u^\ell d_{i_1,0} \dots d_{i_k,0}$ where $b \in MU_*$, $\ell \geq 1$, and $i_1,\dots,i_k \geq 1$.
If $f_2 \neq 0$, then 
\[ \deg_u\phi(f) = \deg_u\phi(f_1 + f_2) = \deg_u\phi(f_2) > 0,\]
 so we deduce that if $f \in \Omega^{C_2}_*$ then $f$ can be written as a polynomial in $\{d_{i,j},q_j\}$.

It follows that 
\[\Omega^{C_2}_* \cong MU_*[d_{i,j},q_j]/I\] 
where $I = J \cap MU_*[d_{i,j},q_j]$ is the elimination ideal of $u$. To complete our calculation of $\Omega^{C_2}_*$, we must find generators of the elimination ideal $I$. The relations in $J \subset MU_*[u,d_{i,j},q_j]$ assert that the elements $d_{i,j}- c_{i,j}$ and  $q_j - p_j$ are divisible by the euler class $u$. In particular, for any $F,G \in \{ d_{i,j} - c_{i,j} , q_j - p_j\}$, we have the relation
\[
(F/u) G = F ( G/u)
\]
in $\Omega^{C_2}_*$. These are precisely the relations listed in the statement of the theorem, and  we will prove that these generate the ideal $I$. In order to do this, we need a technical lemma (Lemma \ref{eliminationlemma}) from commutative algebra, which is essentially an application of Buchberger's algorithm. Using the notation of Lemma \ref{eliminationlemma}, the result holds by setting
\[R = MU_*[d_{i,0}, q_0 : i \geq 1] / (q_0),\]
\[ \{ x_1 , x_2, x_3, \dots \} =  \{d_{i,j+1},q_{j+1} : i  \geq 1\text{ and } j \geq 0 \} \]
\[
\{ \pi_1 , \pi_2, \pi_3, \dots \} = \begin{Bmatrix} c_{i,j} - d_{i,j} & i \geq 1 \text{ and }  \\ p_j - q_{j} & j \geq 0 \end{Bmatrix} .
\]
where, if $x_k = d_{i,j+1}$ then $\pi_k = c_{i,j} - d_{i,j}$, and if $x_k = q_{j+1}$ then $\pi_k = p_j - q_j$.
\end{proof}

Having calculated $\Omega^{C_2}_*$, our next goal is to calculate the extended coefficient ring $\Omega^{C_2}_\diamond$. This amounts to calculating $\Omega^{C_2}_{*-n\sigma}$ for each $n \geq 0$. We begin by evaluating $\Omega^{C_2}_{*-n\sigma}(-) \to MU^{C_2}_{*-n\sigma}(-)$ on the cofiber sequence 
\[ EC_{2+}  \to S^{0} \to \widetilde{EC}_{2}  \]
which yields the following diagram whose rows are exact.
\begin{equation} \label{extendeddiagram}
\begin{tikzcd}
\dots \ar[r] & \Omega^{C_2}_{*-n\sigma} (EC_{2})\ar[r] \ar[d] & \Omega^{C_2}_{*-n\sigma} \ar[r] \ar[d] & MU_*[u^{-1},d_i] \ar[r] \ar[d] &  \dots \\
\dots \ar[r] & MU^{C_2}_{*-n\sigma} (EC_{2}) \ar[r] & MU^{C_2}_{*-n\sigma} \ar[r,"\phi_n"] & MU_*[u^{\pm 1},d_i] \ar[r] & \dots
\end{tikzcd}
\end{equation}

\begin{lemma} The square 
\[
\begin{tikzcd}
\Omega^{C_2}_{*-n\sigma} \ar[r] \ar[d] & MU_*[u^{- 1},d_1,d_2,\dots] \ar[d] \\
MU^{C_2}_{*-n\sigma} \ar[r,"\phi_n"] & MU_*[u^{\pm 1}, d_1,d_2,\dots]
\end{tikzcd}
\]
in the diagram above is a pullback of $MU_*$-modules.
\end{lemma}

\begin{proof}
Recall that the thom class $\tau^{-n}  \in MU^{C_2}_{n\sigma - 2n}$ is represented by the map $S^{n\sigma} \to MU(\mathbf{C}^n)$ associated to the vector bundle $\mathbf{C}^{n\sigma} \to *$, and the element $u^n \in MU^{C_2}_{-2n}$ is represented by the composite $S^0 \subset S^{n\sigma} \to MU(\mathbf{C}^n)$. Since the map $MU^{C_2}_*\to MU^{C_2}_*(\widetilde{EC}_{2})$ is given on representatives by taking fixed points, and since 
\[
(S^0 \subset S^{n\sigma} \to MU(\mathbf{C}^n))^{C_2} = (S^{n\sigma } \to MU(\mathbf{C}^n) )^{C_2},
\]
we deduce that the following diagram commutes:
\[
\begin{tikzcd}
MU^{C_2}_{* - n\sigma} \ar[d,swap,"\tau^{-n}"] \ar[r,"\phi_n"] & MU_*[u^{\pm 1},d_1,d_2,\dots] \ar[d,"u^n"]  \\
 MU^{C_2}_{* - 2n} \ar[r,"\phi"] &  MU_*[u^{\pm 1},d_1,d_2,\dots].
\end{tikzcd}
\]
This square sits inside of diagram \ref{extendeddiagram} whose rows are exact. Since $EC_{2+} \wedge S^{n\sigma}$ is free as a based $C_2$-space, and $\Omega_{C_2}$ and $MU_{C_2}$ are split $C_2$-spectra, the map $\Omega^{C_2}_{*-n\sigma} (EC_{2}) \to MU^{C_2}_{*-n\sigma} (EC_{2})$ is an isomorphism. Our square fits into the commutative diagram
\[ \begin{tikzcd} 
 \Omega^{C_2}_{*-n\sigma} \ar[r] \ar[d] & MU_*[u^{-1} , d_1,d_2,\dots] \ar[d,"\iota"] \\
MU^{C_2}_{*-n\sigma} \ar[r,"\phi_n"] \ar[d,swap,"\tau^{-n}"] \ar[d,"\cong"] & MU_*[u^{\pm 1},d_1,d_2,\dots] \ar[d,"u^n"] \ar[d,swap,"\cong"]  \\
MU^{C_2}_{*-2n} \ar[r,"\phi"] & MU_*[u^{\pm 1}, d_1,d_2,\dots] .
\end{tikzcd} \]
and from this description of the map $MU^{C_2}_{*-n\sigma} \to MU_*[u^{\pm 1},d_1,d_2,\dots ]$ it is clear that the maps $\phi_n$ and $\iota$ mutually surject, so by Lemma \ref{halemma} the diagram is a pullback of $MU_*$-modules.
\end{proof}

Combining these results, we can calculate the extended coefficient ring $\Omega^{C_2}_\diamond$ of the $C_2$-spectrum $\Omega_{C_2}$.

\begin{theorem} Let $\Omega_{C_2}$ denote the $C_2$-equivariant geometric complex cobordism spectrum. 
\begin{enumerate}
\item The extended coefficient ring $\Omega^{C_2}_\diamond$ is given by 
\begin{align*}
\Omega^{C_2}_\diamond & =  \dfrac{\Omega^{C_2}_*[\mu,\tau]}{\begin{matrix}\tau(d_{i,j}-c_{i,j}) = \mu d_{i,j+1} \\ \tau(q_j - p_j) = \mu q_{j+1}\end{matrix} }\hspace{0.3in} i \geq 1 \text{ and }j \geq 0,
\end{align*}
where $|\mu| = - \sigma$, and $|\tau| = 2 - \sigma$. Additively,
\[
\Omega^{C_2}_{*-n\sigma} = \widetilde{\Omega}^{C_2}_*(S^{n\sigma}) \cong  \dfrac{\Omega^{C_2}_* \{ 1 , \dots , u^n\}}{\begin{matrix} u^k(d_{i,j}-c_{i,j}) = u^{k+1}d_{i,j+1} \\ u^k(q_j - p_j) = u^{k+1}q_{j+1} \end{matrix} }\hspace{0.3in} i \geq 1 \text{ and }j \geq 0,
\]
where $ 0 \leq k < n$.
\item If we define the {\it euler filtration } of $MU^{C_2}_*$ by letting $F_n MU^{C_2}_*$ be the $\Omega^{C_2}_*$-submodule generated by $1 , \dots , u^n \in MU^{C_2}_*$, then the map
\[ \Omega^{C_2}_\diamond \to MU^{C_2}_\star = MU^{C_2}_*[\tau^{ \pm 1}]\]
identifies $\Omega^{C_2}_\diamond \cong \text{Rees}(MU^{C_2}_*)$ with the Rees algebra of the euler filtration of $MU^{C_2}_*$.
\item The associated graded of $MU^{C_2}_*$ with respect to the euler filtration is 
\[
\text{gr}_\bullet MU^{C_2}_* =  \Omega^{C_2}_*[\mu]/(\mu d_{i,j}, \mu q_j),  \hspace{0.3in}i,j\geq 1.
\]
Additively,
\[ 
\text{gr}_nMU^{C_2}_* \cong 
\begin{cases}
\Omega^{C_2}_* & n = 0\\ 
MU_*[d_1,d_2,\dots] & n > 0.
\end{cases} 
\]
\end{enumerate}
\end{theorem}

\begin{proof}
We have shown that $\Omega^{C_2}_{*-n\sigma}$ sits in the following pullback square:
\[ \begin{tikzcd}
\Omega^{C_2}_{*-n\sigma} \ar[r] \ar[dd] & MU_*[u^{-1},d_1,d_2, \dots]  \ar[d] \\
& MU_*[u^{\pm 1},d_1, d_2, \dots] \ar[d,"u^n"] \\
MU^{C_2}_{*-2n} \ar[r] & MU_*[u^{\pm 1},d_1,d_2,\dots] 
\end{tikzcd} \]
so $\Omega^{C_2}_{*-n\sigma} \subset MU^{C_2}_*$ consists of the elements whose image in $MU_*[u^{\pm 1},d_i]$ have $u$-degree $\leq n$. Using our presentation of $MU^{C_2}_*$, one can check that any such element must be of the form $b_ 0+ b_1 u + \dots + b_nu^n$, so $\Omega^{C_2}_{*-n\sigma}$ is generated over $\Omega^{C_2}_*$ by $1 , \dots , u^n$. We apply lemma \ref{eliminationlemma} from Appendix \ref{appendixC} to obtain the presentation
\[
\Omega^{C_2}_{*-n\sigma} = \dfrac{\Omega^{C_2}_*\{1,\dots,u^n\}}{ \begin{matrix} u^{k-1} (d_{i,j}-c_{i,j}) = u^k d_{i,j+1} \\ u^{k-1}(q_j-p_j) = u^kq_{j+1} \end{matrix} }
\]
where $i \geq 1$, $j\geq 0$, and $1 \leq k < n$. We can then calculate 
\begin{align*}
Gr_nMU^{C_2}_* = \Omega^{C_2}_{*-n\sigma} / \Omega^{C_2}_{* - (n-1)\sigma}  & \cong \dfrac{\Omega^{C_2}_* \{ u^n\}}{\begin{matrix} u^nd_{i,j+1}= 0 \\ u^nq_{j+1} = 0 \end{matrix} } \\
& \cong MU_*[d_{i,0} : i \geq 1] \{u^n\},
\end{align*}
where we have used our presentation of $\Omega^{C_2}_*$ to deduce that there are no relations between the elements $d_{i,0}$.
\end{proof} 

\section{Geometric orientations}

In this section we introduce our new theory of geometrically oriented $C_2$-spectra. In section \ref{defs} we define geometric orientations, provide some examples of geometrically oriented $C_2$-spectra, and prove some of the fundamental properties of such spectra. In sections \ref{ordinary} and \ref{ktheory}, we investigate the examples $H \underline{\ZZ}_{C_2}$ and $k_{C_2}$, which are the ``additive" and ``multiplicative" geometrically oriented $C_2$-spectra. In section \ref{projectivespacessection} we define filtered $C_2$-equivariant formal group laws, which are the algebraic structures determined by geometrically oriented $C_2$-spectra. Finally, in section \ref{changeofbasis}, we investigate the $C_2$-equivariant projective spaces $[\mathbf{CP}(m+n\sigma)] \in \Omega^{C_2}_*$ which control the various direct sum decompositions of a filtered $C_2$-equivariant formal group law.

\subsection{Definition and basic properties}\label{defs}

In this section we develop the foundations of our theory of geometrically oriented $C_2$-spectra. We begin by reviewing the definition of a complex oriented $C_2$-spectrum. Let $U= \mathbf{C}^{\infty,\infty}$ be a complete complex $C_2$-universe, so that $\mathbf{CP}^\infty_{C_2} = \mathbf{CP}(U)$ is the classifying space for $C_2$-equivariant line bundles. For any $m,n \geq 0$, we write $\mathbf{CP}(m+n\sigma) \subset \mathbf{CP}^\infty_{C_2}$ for the sub-projective space associated to the subrepresentation $\mathbf{C}^{m,n} \subset \mathbf{C}^{\infty,\infty}$. We equip $\mathbf{CP}^\infty_{C_2}$ with the basepoint $* = \mathbf{CP}(1) \in \mathbf{CP}^\infty_{C_2}$. For each  $\rho \in \{1,\sigma\}$, we have an inclusion
\[
(S^{\rho^{-1}}, *) \simeq (\mathbf{CP}(1+\rho),\mathbf{CP}(1)) \subset (\mathbf{CP}^\infty_{C_2},\mathbf{CP}(1)) = (\mathbf{CP}^\infty_{C_2}, *).
\]
so if $E_{C_2}$ is complex stable, meaning that we have specified an equivalence $\Sigma^{\sigma - |\sigma|} E_{C_2} \simeq E_{C_2}$, we can restrict a class $x \in E_{C_2}^2(\mathbf{CP}^\infty_{C_2},\mathbf{CP}(1))$ to a class in $ E_{C_2}^2(S^{\rho^{-1}},*) \cong E^0_{C_2}$. We now recall the definition of a complex orientation of a $C_2$-spectrum.

\begin{definition}
If $E_{C_2}$ is a complex stable commutative ring $C_2$-spectrum, then a {\it complex orientation} of $E_{C_2}$ is a cohomology class $x \in E_{C_2}^2(\mathbf{CP}^\infty_{C_2},\mathbf{CP}(1))$ which restricts to $1$ in
\[ E_{C_2}^0 \cong E_{C_2}^2(\mathbf{CP}(1+1),\mathbf{CP}(1)),\]
and some unit in 
\[E_{C_2}^0 \cong E_{C_2}^2(\mathbf{CP}(1 + \sigma), \mathbf{CP}(1)).\]
\end{definition}

Many of our favorite $C_2$-spectra are complex oriented, such as $F(EC_{2+},H\ZZ)$, $K_{C_2}$, and $MU_{C_2}$. In \cite{CGK2}, Cole, Greenlees and Kriz prove that a complex orientation of $E_{C_2}$ is uniquely determined by a commutative ring spectrum map $MU_{C_2} \to E_{C_2}$. It is then natural to ask: what structure is afforded to a $C_2$-spectrum $E_{C_2}$ equipped with a commutative ring spectrum map $\Omega_{C_2} \to E_{C_2}$? We propose the following definition, which includes several flatness hypotheses in order to maintain algebraic control.

\begin{definition}
Suppose $E_{C_2}$ is a commutative ring $C_2$-spectrum. We say a commutative ring spectrum map $\Omega_{C_2} \to E_{C_2}$ is a {\it geometric orientation} of $E_{C_2}$  if 
\begin{enumerate}
\item the transfer $\text{tr}_e^{C_2}:E_* \to E^{C_2}_*$ is injective, and 
\item $\tau \in \Omega^{C_2}_\diamond$ maps to a non-zero-divisor in $E^{C_2}_\diamond$.
\end{enumerate}
\end{definition}
If we have specified such a map $\Omega_{C_2} \to E_{C_2}$, we say $E_{C_2}$ is {\it geometrically oriented}. Many important $C_2$-spectra which fail to be complex oriented are in fact geometrically oriented. We list some examples of geometrically oriented $C_2$-spectra below, and we will investigate these further in the following sections.
\begin{example}
The universal example of a geometrically oriented $C_2$-spectrum is $\Omega_{C_2}$ itself, which is geometrically oriented by the identity map.
\end{example} 
\begin{example}
Suppose $R$ is a commutative ring with no $2$-torsion. Then the $C_2$-equivariant Eilenberg-Maclane spectrum $H \underline{R}_{C_2}$ associated to the constant Mackey functor $\underline{R}$ is geometrically oriented.
\end{example} 
\begin{example}
The connective cover $k_{C_2}= \tau_{\geq 0} K_{C_2}$ of $C_2$-equivariant $K$-theory is geometrically oriented.
\end{example} 

Before discussing the relationship between our new theory of geometric orientations, and the classical theory of complex orientations, we make note of the following result, which establishes the connection between geometric orientations and thom isomorphisms for certain complex vector bundles.

\begin{proposition}\label{thomisos}
Suppose $E_{C_2}$ is a geometrically oriented $C_2$-spectrum. If $\psi \to X/C_2$ is a complex vector bundle over the orbits of a $C_2$-space $X$, and $\xi = p^*\psi$ is the pullback of $\psi \to X/C_2$ along the projection map $p:X \to X/C_2$, then there is a thom isomorphism
\[
E_{C_2}^*(X) = \widetilde{E}_{C_2}^{*+2\dim \xi}(X^\xi).
\]
\end{proposition}

\begin{proof}
Suppose we have a rank $n$ vector bundle $\xi = p^*\psi$ as above. Then since $\psi \to X/C_2$ is a $C_2$-equivariant complex vector bundle over a $C_2$-space with trivial $C_2$-action, the vector bundle $\psi$ is classified by a map $X/C_2 \to \text{Gr}^{\mathbf{C}^\infty}( \mathbf{C}^n)$, and the pullback $\xi = p^*\psi$ is classified by the composite 
\[
X\to X/C_2 \to \text{Gr}^{\mathbf{C}^\infty}(\mathbf{C}^n) .
\]
Taking thom spaces on the corresponding map of vector bundles yields 
\[
X^\xi \to \text{Thom}(\xi^{\mathbf{C}^\infty}(\mathbf{C}^n) \to \text{Gr}^{\mathbf{C}^\infty}( \mathbf{C}^n)) = \Omega_{C_2}(\mathbf{C}^n)
\]
which determines a thom class $t(\xi) \in \widetilde{\Omega}_{C_2}^{2n}(X^\xi)$. Since $E_{C_2}$ is geometrically oriented, we can push forward the class $t(\xi)$ to a class $t(\xi) \in \widetilde{E}_{C_2}^{2n}(X^\xi)$. We can now make use of the Thom diagonal $\delta: X^\xi \to X_+ \wedge X^\xi$. More precisely, we claim that the map $E_{C_2}^*(X) \to \widetilde{E}_{C_2}^{*+2n}(X^\xi)$ which send the class $\omega \in E_{C_2}^*(X)$ to the class of the composite 
\[ \begin{tikzcd} 
X^\xi \ar[r,"\delta"] & X_+ \wedge X^\xi \ar[r, "\omega \wedge t(\xi)"] & E_{C_2} \wedge E_{C_2} \ar[r] & E_{C_2},
\end{tikzcd} \]
is an isomorphism. This follows from the fact that for any point $x \in X$, if we let $\xi_x$ denote the fiber of the vector bundle $\xi$ over $x \in X$, the restriction of $t (\xi) \in \widetilde{E}_{C_2}^{2n}(X^\xi)$ to $\widetilde{E}_{C_2}^{2n}(S^{\xi_x}) \cong \widetilde{E}_{C_2}^{2n}(S^{2n}) \cong E_{C_2}^0$ corresponds to the unit $1 \in E_{C_2}^0$.
\end{proof}

Having established the connection between geometrically oriented $C_2$-spectra and thom isomorphisms for vector bundles, we will now investigate the connection betweeen geometric orientations and complex orientations.  A key observation is that we can associate to any geometrically oriented $C_2$-spectrum $E_{C_2}$ a complex oriented $C_2$-spectrum $\widehat{E}_{C_2}$ in the following way. 
 If $E_{C_2}$ is a geometrically oriented $C_2$-spectrum, then the element $\tau \in \Omega^{C_2}_\diamond$ maps to some element in $\tau \in E^{C_2}_\diamond$. Since $\Omega_{C_2}[1/\tau] \simeq MU_{C_2}$, and $MU_{C_2}$ classifies $C_2$-equivariant complex orientations, inverting the class $\tau \in E^{C_2}_\diamond$ yields a complex oriented $C_2$-spectrum
 \begin{align*} 
 \widehat{E}_{C_2} & = E_{C_2}[1/\tau]  = \text{hocolim}\left( E_{C_2} \overset{\tau}{\longrightarrow}  \Sigma^{\sigma-|\sigma| } E_{C_2} \overset{\tau}{\longrightarrow}  \Sigma^{2\sigma - |2\sigma|} E_{C_2} \overset{\tau}{\longrightarrow} \cdots \right)
 \end{align*}
which we call the {\it stabilization} of $E_{C_2}$. The $C_2$-spectrum $\widehat{E}_{C_2} $ inherits a multiplicative structure from that of $E_{C_2}$, and is filtered by the defining diagram
\[ \begin{tikzcd} 
E_{C_2} \ar[r,"\tau"] & \Sigma^{\sigma-|\sigma| } E_{C_2} \ar[r,"\tau"] & \Sigma^{2\sigma - |2\sigma|} E_{C_2} \ar[r,"\tau"] & \cdots \ar[r] & \widehat{E}_{C_2}.
\end{tikzcd} \]
Since we have assumed that multiplication by $\tau$ is injective in the good range $E^{C_2}_\diamond \subset E^{C_2}_\star$, this determines a filtration of the coefficients of $\widehat{E}_{C_2}$:
\[
E^{C_2}_* \subset E^{C_2}_{*+|\sigma| -\sigma} \subset E^{C_2}_{*+|2\sigma| -2\sigma} \subset \cdots \subset \widehat{E}^{C_2}_*.
\]
We call this the {\it euler filtration} of $\widehat{E}^{C_2}_*$, for reasons that will become clear in Theorem \ref{eulerfiltration}. From this perspective, $E^{C_2}_\diamond$ is the {\it Rees Algebra} of the filtered ring $\widehat{E}^{C_2}_*$, which interpolates between the ``generic fiber" $\widehat{E}^{C_2}_*$, and the ``special fiber"  $\text{gr}_\bullet \widehat{E}^{C_2}_*$, as depicted below.
\[ \begin{tikzcd} 
 & E^{C_2}_\diamond \ar[dl,swap, "/\tau - 1"] \ar[dr,"/\tau - 0"] & \\
 \widehat{E}^{C_2}_*  & & \text{gr}_\bullet\widehat{E}^{C_2}_*
\end{tikzcd} \]
We can think of a map $E^{C_2}_\diamond \to A$ as a deformation of the $C_2$-equivariant formal group law determined by $\widehat{E}^{C_2}_* = E^{C_2}_\diamond/(\tau - 1) \to A/(\tau - 1)$. Having established some basic properties of geometrically oriented spectra, we turn our attention to the examples $H \underline{R}_{C_2}$ and $k_{C_2}$.

\subsection{Ordinary cohomology}\label{ordinary}

The simplest example of a geometrically oriented $C_2$-spectrum is the Eilenberg-Maclane spectrum $H\underline{R}$ associated to a commutative ring $R$ with no $2$-torsion. Recall that the constant Mackey functor $\underline{R}$ is defined by 
 \[ \underline{R}(C_2/e) = R =  \underline{R}(C_2/C_2).\]
  The restriction $\text{res}_e^{C_2}$ is the identity, and the transfer $\text{tr}^{C_2}_e$  is multiplication by $2$. The Eilenberg-Maclane spectrum $H \underline{R}$ represents $\underline{R}$ in the sense that 
\[\underline{\pi}_n(H\underline{R}) = 
\begin{cases}
\underline{R} & n = 0\\
0 & n \neq 0.
\end{cases}
\]
In the following theorem, we calculate the stabilization of the geometrically oriented $C_2$-spectrum $H\underline{R}_{C_2}$. Note that the $RO(C_2)$-graded coefficient ring of $H \underline{R}_{C_2}$ is well known, see for instance \cite{Lewis}.

\begin{theorem} If $R$ is a commutative ring with no $2$-torsion, then the Eilenberg-Maclane spectrum $H \underline{R}$ is geometrically oriented. The extended coefficient ring of $H \underline{R}$ is 
\[
H \underline{R}^{C_2}_\diamond = R[\mu,\tau]/(2\mu)
\] 
where $|\mu| = -\sigma$ and $|\tau| = 2-\sigma$. The stabilization of $H \underline{R}$ is  
\[
H \underline{R}[1/\tau] \simeq F(EC_{2+},HR).
\]
\end{theorem} 

\begin{proof} 
The $RO(C_2)$-graded coefficients of $H \underline{R}_{C_2}$ are well known, and can be calculated using the Tate diagram. The completion map
\[
H \underline{R}_{C_2} \to F(EC_{2+} , H \underline{R}_{C_2}) \simeq F(EC_{2+} , HR)
\]
exhibits $H \underline{R}_{C_2}$ as the connective cover of $F(EC_{2+},HR)$, which is complex oriented. Since $H\underline{R}_{C_2}$ is connective, this determines a geometric orientation $\Omega_{C_2} \to H \underline{R}_{C_2}$. Since the completion map takes $\tau \in H\underline{R}^{C_2}_{2-\sigma}$ to a unit in $F(EC_{2+},HR)^{C_2}_\star$, there is an induced map $H\underline{R}_{C_2}[1/\tau] \to F(EC_{2+},HR)$ which we claim is an equivalence. It is a non-equivariant equivalence since $H \underline{R}_{C_2} \to F(EC_{2+},HR)$ is a non-equivariant equivalence and $\tau$ is non-equivariantly homotopic to $1 \in H\underline{R}_0^{\{e\}}$. Moreover,
 \begin{align*} 
 H \underline{R}[1/\tau]^{C_2}_*  = H \underline{R}^{C_2}_\diamond / (\tau - 1) & =  R [ \mu,\tau ]/(2\mu, \tau - 1) \\
 & \cong R[u]/(2u)\\
 &= F(EC_{2+},HR)^{C_2}_*,
 \end{align*}
 so $H \underline{R}_{C_2}[1/\tau] \to F(EC_{2+},HR)$ induces an isomorphism on $\pi^{C_2}_*(-)$.
\end{proof}

\subsection{Connective $K$-theory}\label{ktheory}
Our next important example of a geometrically oriented $C_2$-spectrum is connective $C_2$-equivariant $K$-theory. Recall that if $E_{C_2}$ is a $C_2$-spectrum, then the connective cover $\tau_{ \geq 0} E_{C_2}$ is a $C_2$-spectrum equipped with a map $\tau_{ \geq 0} E_{C_2} \to E_{C_2}$ such that $\underline{\pi}_n(E_{C_2}) = 0$ if $n < 0$, and $\underline{\pi}_n(\tau_{ \geq 0}E_{C_2}) \to \underline{\pi}_n(E_{C_2})$ is an isomorphism for $n \geq 0$. Connective covers are unique up to canonical isomorphism in the $\text{Ho}(\text{Sp}_{C_2})$. In the case $E_{C_2} = K_{C_2}$, we define $k_{C_2} = \tau_{ \geq 0} K_{C_2}$ to be the connective $C_2$-equivariant $K$-theory spectrum.

There is another important $C_2$-equivariant analogue $ku_{C_2}$ of connective $K$-theory, which was defined and studied by Greenlees in \cite{Greenlees1}, \cite{Greenlees2}, and \cite{Greenlees3}. While $ku_{C_2}$ is not actually the connective cover of $K_{C_2}$, the $C_2$-spectrum $ku_{C_2}$ enjoys many desirable properties: it is complex stable, complex oriented, and Greenlees proves that the coefficient ring 
\begin{align*}
ku^{C_2}_* & = R(C_2)[v,v^{-1}J] \\
& \cong  \ZZ[u,v]/(2u + v u^2)
\end{align*}
classifies multiplicative $C_2$-equivariant formal group laws in the sense of \cite{CGK1}. We write $J \subset R(C_2)$ for the augmentation ideal of $R(C_2)$, which is generated by the element $\sigma - 1 \in R(C_2)$. For this reason, it is the spectrum $ku_{C_2}$ whose properties mirror those of non-equivariant connective $K$-theory. 

The present theory of geometric orientations provides a new and interesting link between these two $C_2$-equivariant analogues of connective $K$-theory. In the next theorem, we calculate the extended coefficient ring of the geometrically oriented $C_2$-spectrum $k_{C_2}$, and prove that the stabilization of $k_{C_2}$ is Greenlees' spectrum $ku_{C_2}$.

\begin{theorem}The connective cover $k_{C_2} = \tau_{ \geq 0} K_{C_2}$ of $C_2$-equivariant $K$-theory is geometrically oriented. The extended coefficient ring of $k_{C_2}$ is 
\[
k^{C_2}_\diamond = \dfrac{R(C_2)[ v , \mu, \tau]}{\begin{matrix} \tau(\sigma - 1) = v\mu\\
\mu(\sigma + 1) = 0 \end{matrix} } 
\]
where $|v| = 2$, $|\mu| = -\sigma$, and $|\tau| = 2 - \sigma$. The stabilization of $k_{C_2}$ is Greenlees' equivariant connective $K$-theory
\[
k_{C_2}[1/\tau] \simeq ku_{C_2}.
\]
\end{theorem}

\begin{proof}
The  $C_2$-spectrum $ku_{C_2}$ lies in a homotopy pullback square
\[ \begin{tikzcd} 
ku_{C_2} \ar[d] \ar[r] & F(EC_{2+},ku)\ar[d] \\
K_{C_2} \ar[r] & F(EC_{2+},K).
\end{tikzcd} \]
The canonical map $k_{C_2} \to K_{C_2}$ and the completion map $k_{C_2} \to F(EC_{2+},k_{C_2})  \simeq F(EC_{2+},ku)$ induce a multiplicative map $k_{C_2} \to ku_{C_2}$, and we will prove that this induces an equivalence $k_{C_2}[1/\tau] \simeq ku_{C_2}$ by calculating the extended coefficient ring of $k_{C_2}$.

We claim that the map $k^{C_2}_{*+|n\sigma| - n\sigma} \to K^{C_2}_*$ is injective with image
\[
J^nv^{-n} \oplus \cdots \oplus Jv^{-1} \oplus R(C_2)[v] \subset R(C_2)[v^{\pm 1}] = K^{C_2}_*.
\]
We prove this claim by induction on $n$, and the base case $n=0$ follows from the definition of connective cover. Applying $k^{C_2}_{*-n\sigma}(-) \to K^{C_2}_{*-n\sigma}(-)$ to the cofiber sequence 
\[
S(\sigma)_+\to S^0 \to S^{\sigma}.
\]
yields the diagram

\begin{equation} \label{lesktheory}
\begin{tikzcd} [column sep = 15.0]
\cdots \ar[r] & k^{C_2}_{*-n\sigma}(S(\sigma)) \ar[r] \ar[d] & k^{C_2}_{*-n\sigma} \ar[r] \ar[d]  & k^{C_2}_{*-(n+1)\sigma} \ar[r] \ar[d] & k^{C_2}_{*-1-n\sigma} (S(\sigma)) \ar[r] \ar[d] &  \cdots \\
\cdots \ar[r] & \mathbb{Z}[v^{\pm 1}] \{2 + uv\} \ar[r]  & R(C_2)[v^{\pm 1}] \ar[r,"u"] &R(C_2)[v^{\pm 1}] \ar[r] & \mathbb{Z}[v^{\pm 1}]  \ar[r] & \cdots \\
\end{tikzcd} 
\end{equation}
whose rows are exact. By applying $k^{C_2}_{*-n\sigma}(-) \to K^{C_2}_{*-n\sigma}(-) $ to the cofiber sequence 
\[
C_{2+}\to S(\sigma)_+ \to \Sigma C_{2+}
\]
we can deduce that $k^{C_2}_{m-n\sigma}(S(\sigma)) \to K^{C_2}_{m-n\sigma}(S(\sigma))$ is an isomorphism for $m \geq 2n$ and $k^{C_2}_{*-n\sigma}(S(\sigma)_+) = 0$ for $m < 2n$. Exactness of the rows in \ref{lesktheory} implies that $k^{C_2}_{m-(n+1)\sigma} \to K^{C_2}_{m - (n+1)\sigma}$ is an isomorphism if $m \geq 2(n+1)$, and $k^{C_2}_{m - (n+1)\sigma} = 0$ if $m$ is odd or $m < 0$. If $0 < 2k \leq 2(n+1)$, then our diagram is 
\[
\begin{tikzcd}
\cdots \ar[r] & 0 \ar[r] \ar[d] & J^{n-k} \ar[r] \ar[d]  & k^{C_2}_{2k-(n+1)\sigma} \ar[r] \ar[d] & 0 \ar[r] \ar[d] &  \cdots \\
\cdots \ar[r] & \mathbb{Z}\{ 1+ \sigma\} \ar[r]  & R(C_2) \ar[r,"\sigma - 1"] &R(C_2) \ar[r] & \mathbb{Z}  \ar[r] & \cdots \\
\end{tikzcd} 
\]
and exactness implies that $k^{C_2}_{2k - (n+1)\sigma} = J^{n+1- k}$. The presentation 
\[
k^{C_2}_\diamond \cong  \dfrac{R(C_2)[v,\mu,\tau]}{\begin{matrix} \tau(\sigma - 1) = v \mu \\  \mu(\sigma + 1) = 0\end{matrix}} 
\]
is obtained by setting $\mu = \sigma- 1 \in J = k^{C_2}_{-\sigma}$, and $\tau = 1 \in R(C_2) = k^{C_2}_{2-\sigma}$. 

Next, we'll prove that the map $k_{C_2}[1/\tau] \to ku_{C_2}$ is an equivalence. This map is a non-equivariant equivalence since $k_{C_2} \to ku_{C_2}$ is a non-equivariant equivalence and $\tau$ is non-equivariantly homotopic to $1 \in k_0^{\{e\}}$. We have 
\begin{align*}
k[1/\tau]^{C_2}_* = k^{C_2}_\diamond/(\tau - 1) & = \ZZ[u,v]/(2u + v u^2)\\
& = ku^{C_2}_*,
\end{align*} 
 so $k_{C_2}[1/\tau] \to ku_{C_2}$ induces an isomorphism on $\pi^{C_2}_*(-)$. We deduce that $k_{C_2}[1/\tau] \to ku_{C_2}$ is an equivalence.
\end{proof}

\subsection{Filtered $C_2$-equivariant formal group laws}\label{projectivespacessection}
In this section we introduce and develop the theory of {\it filtered }$C_2${\it -equivariant formal group laws}, which are the algebraic objects determined by geometrically oriented $C_2$-spectra. Since a filtered $C_2$-equivariant formal group law is a $C_2$-equivariant formal group law equipped with additional structure, we begin by recalling the definition of a $C_2$-equivariant formal group law.

\begin{definition}
A $C_2${\it -equivariant formal group law} $(A,D)$ consists a commutative ring $A$,
an $A$-Hopf algebra $D$, a morphism $A[C_2^\vee] \to D$ of $A$-Hopf algebras, and an $A$-linear functional $x$ on $ D$,
such that 
\begin{enumerate}
\item The sequence 
\[ \begin{tikzcd} 
0 \ar[r] & A \ar[r,"\eta"] & D \ar[r,"\cap x"] & D \ar[r] & 0
\end{tikzcd} \]
is exact, and 
\item if $d \in D$, then there exist $m,n \geq 0$ such that 
\[
d \cap x^{m+n\sigma} = 0.
\]
\end{enumerate}
\end{definition}

If $E_{C_2}$ is a complex oriented $C_2$-spectrum, then the pair $(E^{C_2}_*,E^{C_2}_*(\mathbf{CP}^\infty_{C_2}))$ carries the structure of a $C_2$-equivariant formal group law. The morphism $E^{C_2}_*[C_2^\vee] \to E^{C_2}_*(\mathbf{CP}^\infty_{C_2})$ is obtained by applying $E^{C_2}_*(-)$ to the inclusion
\[
C_2^\vee = \mathbf{CP}(1) \amalg \mathbf{CP}(\sigma) \to  \mathbf{CP}^\infty_{C_2}
\]
and the linear functional $x$ is the map $E^{C_2}_*(\mathbf{CP}^\infty_{C_2}) \to E^{C_2}_*$ obtained by pairing with the complex orientation $x \in \widetilde{E}^2_{C_2}(\mathbf{CP}^\infty_{C_2})$. We will show that when $\widehat{E}_{C_2}$ is the stabilization of some geometrically oriented $C_2$-spectrum $E_{C_2}$, then the $C_2$-equivariant formal group law $(\widehat{E}^{C_2}_*, \widehat{E}^{C_2}_*(\mathbf{CP}^\infty_{C_2}))$ is afforded the additional data of a filtration 
\begin{align*}
F_n\widehat{E}^{C_2}_* & = E^{C_2}_{*+|n\sigma| -n\sigma}\\
F_n\widehat{E}^{C_2}_*(\mathbf{CP}^\infty_{C_2}) & = E^{C_2}_{*+|n\sigma|-n\sigma}(\mathbf{CP}^\infty_{C_2}).
\end{align*}
The interaction of this filtration with the algebraic structure of the $C_2$-equivariant formal group law $(\widehat{E}^{C_2}_*,\widehat{E}^{C_2}_*(\mathbf{CP}^\infty_{C_2}))$ is surprisingly rich and deep. We call the resulting structure a {\it filtered }$C_2${\it -equivariant formal group law}. In order to properly axiomatize this filtration, we must first prove some structural results about $C_2$-equivariant formal group laws. The proposition below, which characterizes the additive and comultiplicative structure of a $C_2$-equivariant formal group law $(A,D)$, is proved in Appendix \ref{AppendixD}.
\begin{proposition} Suppose $(A,D)$ is a $C_2$-equivariant formal group law. Then we can associate to any sequence $\rho_1, \dots , \rho_n \in  C_2^\vee$ an element $\beta(\rho_1, \dots , \rho_n) \in D$, and these elements satisfy the following properties:
\begin{enumerate} \item 
\[
\Delta \beta(\rho_1, \dots , \rho_n) = \sum_{i=1}^n \beta(\rho_1, \dots , \rho_i) \otimes \beta(\rho_i , \dots , \rho_n),
\]
\item If $(\rho_i)_{i=1}^\infty$ is a complete flag, then 
\[
\langle \beta(\rho_1, \dots , \rho_i) , x^{\rho_1 + \cdots + \rho_{j-1}}\rangle = \begin{cases} 1 & i=j \\ 0 & i \neq j.\end{cases}
\]
\item The set $\{ \beta(\rho_1, \dots , \rho_i) : i \geq 1\}$ is a free $A$-module basis for $D$.
\end{enumerate} 
\end{proposition}
It turns out that the filtration afforded to $(A,D) = (\widehat{E}^{C_2}_*, \widehat{E}^{C_2}_*(\mathbf{CP}^\infty_{C_2}))$ is much clearer when viewed in terms of a different, geometrically defined $A$-module basis of $D$. For any $m,n\geq 0$ we have elements
\begin{align*}
\pi_{m+n\sigma} & = [\mathbf{CP}(m+n\sigma)] \in \Omega^{C_2}_*\text{, and}\\
\Pi_{m+n\sigma} & = [\mathbf{CP}(m+n\sigma) \to \mathbf{CP}^\infty_{C_2}] \in \Omega^{C_2}_*(\mathbf{CP}^\infty_{C_2}).
\end{align*}
which map to elements $\pi_{m+n\sigma} \in MU^{C_2}_*$ and $\Pi_{m+n\sigma} \in MU^{C_2}_*(\mathbf{CP}^{\infty}_{C_2})$.\footnote{The class of the map $\mathbf{CP}(m+n\sigma) \to \mathbf{CP}^\infty_{C_2}$ is well defined since the space of equivariant linear isometric embeddings $m+n\sigma \to \mathbf{C}^{\infty,\infty}$ is connected.} Since the pair $(MU^{C_2}_*,MU^{C_2}_*(\mathbf{CP}^\infty_{C_2}))$ is the universal $C_2$-equivariant formal group law, this determines elements $\pi_{m+n\sigma} \in A$ and $\Pi_{m+n\sigma} \in D$ for every equivariant formal group law $(A,D)$. The following theorem asserts that $\{ \Pi_{\rho_1 + \cdots + \rho_i} : i \geq 1\}$ is an $A$-module basis of $D$, and identifies the coefficients of the change of basis matrix from $\{ \Pi_{\rho_1 + \cdots + \rho_i}: i \geq 1 \}$ to the canonical basis $\{ \beta(\rho_1, \dots , \rho_i) : i \geq 1\}$. 

\begin{theorem}\label{geometricbasistheorem}
Suppose $(A,D)$ is a $C_2$-equivariant formal group law.
\begin{enumerate}
\item If $\rho_1, \dots , \rho_n \in C_2^\vee$, then
\[
\Pi_{\rho_1 + \cdots + \rho_n}  = \sum_{i=1}^n \pi_{\rho_i + \cdots + \rho_n} \beta(\rho_1, \dots ,\rho_i).
\]
\item If $(\rho_1,\rho_2,\dots)$ is a complete flag, then the set $\{ \Pi_{\rho_1 + \cdots + \rho_i} : i \geq 1\}$ is a free $A$-module basis for $D$. 
\end{enumerate}
\end{theorem}

We prove this theorem in section \ref{changeofbasis}. Having developed the necessary background, we can now state our main algebraic definition.

\begin{definition}
A {\it filtered $C_2$-equivariant formal group law} $(F_\bullet A,F_\bullet D)$ consists of a $C_2$-equivariant formal group law $(A,D)$, together with a filtration $F_\bullet A$ of $A$ and $F_\bullet{D}$ of $D$ such that 
\begin{enumerate}
\item $\text{Im}(\Omega^{C_2}_* \to A) \subseteq F_0A$,
\item $F_nA$ is generated over $F_0A$ by $1,\dots,u^n \in A$, and 
\item For any complete flag $(\rho_i)_{i=1}^\infty$, 
\[
F_nD = \left\{ \sum a_i \Pi_{\rho_1 + \cdots + \rho_i} \in D : a_i \in F_{n + \ell_i}A \right\},
\]
where $\ell_i$ is the number of copies of $\sigma$ in $(\rho_1 + \cdots + \rho_{i-1})\rho_i^{-1}$.
\end{enumerate} 
\end{definition} 
One of the main theorems of this section is that every geometrically oriented $C_2$-spectrum determines a filtered $C_2$-equivariant formal group law which refines the $C_2$-equivariant formal group law associated to the complex oriented $C_2$-spectrum $\widehat{E}_{C_2}$.
\begin{theorem}
If $E_{C_2}$ is a geometrically oriented $C_2$-spectrum with stabilization $\widehat{E}_{C_2} = E_{C_2}[1/\tau]$, then the pair $ (F_\bullet \widehat{E}^{C_2}_*,F_\bullet \widehat{E}^{C_2}_*(\mathbf{CP}^\infty_{C_2}))$ defined by
\[
F_nE^{C_2}_* = E^{C_2}_{*+|n\sigma|-n\sigma}\text{, and}
\]
\[
F_nE^{C_2}_*(\mathbf{CP}^\infty_{C_2}) = E^{C_2}_{*+|n\sigma|-n\sigma}(\mathbf{CP}^\infty_{C_2})
\]
is a filtered $C_2$-equivariant formal group law.
\end{theorem}
\begin{proof}
It is clear that axiom (1) of a filtered $C_2$-equivariant formal group law is satisfied since the composite $\Omega_{C_2} \to MU_{C_2} \to \widehat{E}_{C_2}$ lifts across the stabilization map $E_{C_2} \to \widehat{E}_{C_2}$. We prove axiom (2) in Proposition \ref{eulerfiltration} and we prove axiom (3) in Proposition \ref{filtration}.
\end{proof}

\begin{proposition}\label{eulerfiltration}
If $E_{C_2}$ is a geometrically oriented $C_2$-spectrum, then for any $n \geq 0$, the $E^{C_2}_*$-module $E^{C_2}_{*+|n\sigma| - n\sigma} \subset \widehat{E}^{C_2}_*$ is generated by the euler classes 
\[\{u^k \in \widehat{E}^{C_2}_{-2k} : 0 \leq k \leq n\}. \]In particular, the $E^{C_2}_*$-algebra $\widehat{E}^{C_2}_*$ is generated by the euler class $u \in \widehat{E}^{C_2}_{-2}$.
\end{proposition} 

\begin{proof}
Note that for any $0 \leq k \leq n$, the submodule $E^{C_2}_{*+|n\sigma| - n\sigma} \subset \widehat{E}^{C_2}_*$ contains the euler class 
\[
u^k = \left[ S^0 \subset S^{k\sigma} \to \Sigma^{k\sigma}E_{C_2} \simeq \Sigma^{|k\sigma|} \Sigma^{k\sigma - |k\sigma|} E_{C_2} \to \Sigma^{|k\sigma|} \widehat{E}_{C_2} \right] \in \widehat{E}^{C_2}_{-|k\sigma|}
\]
associated to the $C_2$-representation $k \sigma$. We prove that these elements generate $E^{C_2}_{*+|n\sigma| - n\sigma}$ by induction on $n \geq 1$. For the base case, we prove that if  $E_{C_2}$ is geometrically oriented, then $E^{C_2}_{*-\sigma} = \widetilde{E}^{C_2}_*(S^\sigma)$ is generated over $E^{C_2}_*$ by $1$ and $u$. The $C_2$-space $S^{\sigma} = S^{2\alpha}$ has a cell structure 
\[ \begin{tikzcd} 
S^0 \ar[r] & S^\alpha \ar[r] \ar[d] & S^{2\alpha} \ar[d]  \\
 & \Sigma C_{2+} & \Sigma^2C_{2+}
\end{tikzcd} \]
and applying $E^{C_2}_*(-)$ yields a spectral sequence $\mathcal{E}$ converging to $E^{C_2}_{*-\sigma}$ with $\mathcal{E}^1$ page
\[
\mathcal{E}^1_{p,q} = 
\begin{cases} 
E^{C_2}_q & p = 0\\
E_{p+q} &  p = 1,2\\
0 & \text{else.}
\end{cases} 
\]
The differential $\mathcal{E}^1_{1,*} \to \mathcal{E}^1_{0,*-1}$ is the transfer, which is injective by assumption, so the differential $\mathcal{E}^1_{2,*} \to \mathcal{E}^1_{1,*-1}$ is zero and the spectral sequence collapses at the $\mathcal{E}^2$ page to 
\[
\mathcal{E}^2_{p,*} = \mathcal{E}^\infty_{p,*} = 
\begin{cases} 
E^{C_2}_*/\text{tr}_e^{C_2} & p = 0\\
E_* & p = 2\\
0 & \text{else.}
\end{cases} 
\]
The unit $1 \in E^{C_2}_*/\text{tr}_e^{C_2}$ represents  $u \in E^{C_2}_{*-\sigma}$. Since $E_{C_2}$ is an $\Omega_{C_2}$-algebra and $\Omega_{C_2}$ is a split $C_2$-spectrum, we know that $E_{C_2}$ is also a split $C_2$-spectrum, which implies that the restriction $E^{C_2}_* \to E_*$ is surjective. We deduce that $\mathcal{E}^{\infty}_{*,*}$ is generated by $1$ and $u$, hence so is the target $E^{C_2}_{*-\sigma}$.

Suppose next that $E^{C_2}_{* - n\sigma}$ is generated as a $E^{C_2}_*$-module by $1, \dots , u^n$. In order to prove that $E^{C_2}_{* - (n+1)\sigma}$ is generated as a $E^{C_2}_*$-module by $1 , \dots , u^{n+1}$, we smash the map $\tau^n : E \to \Sigma ^{n\sigma - |n\sigma|}E$ with the cofiber sequence $S(\sigma)_+ \to S^0 \to S^\sigma$ to obtain the following diagram whose rows are exact.

\[ \begin{tikzcd} [column sep = 10.5]
\cdots \ar[r] & E^{C_2}_*(S(\sigma)) \ar[d] \ar[r] & E^{C_2}_* \ar[r] \ar[d] & E^{C_2}_{* - \sigma} \ar[r] \ar[d] & E^{C_2}_{*-1}(S(\sigma)) \ar[d] \ar[r] & \cdots \\
\cdots \ar[r] & E^{C_2}_{*+|n\sigma| - n \sigma}(S(\sigma)) \ar[r] & E^{C_2}_{*+|n\sigma| - n \sigma} \ar[r] & E^{C_2}_{* + |n\sigma| - (n+1)\sigma} \ar[r] &  E^{C_2}_{*-1 +|n\sigma| - n\sigma}(S(\sigma)) \ar[r] & \cdots 
\end{tikzcd} \] 
Since $S(\sigma)$ is a free $C_2$-space and $\tau^n$ is a non-equivariant equivalence, the maps \[E^{C_2}_*(S(\sigma)) \to E^{C_2}_{*+|n\sigma| - n\sigma}(S(\sigma))\] are isomorphisms. By taking kernels and cokernels of the middle horizontal maps, we obtain the diagram

\[ \begin{tikzcd} 
 0 \ar[r] & K  \ar[d,"\cong"] \ar[r] & E^{C_2}_* \ar[r,"u"] \ar[d] & E^{C_2}_{* - \sigma} \ar[r] \ar[d] & C \ar[d,"\cong"] \ar[r] & 0\\
0 \ar[r] & K'  \ar[r] & E^{C_2}_{*+|n\sigma| - n \sigma} \ar[r,"u"] & E^{C_2}_{* + |n\sigma| - (n+1)\sigma} \ar[r] & C' \ar[r] & 0
\end{tikzcd} \] 
whose rows are exact. By our inductive hypothesis, we know that $1 \in E^{C_2}_{*-\sigma}$ maps to a $E^{C_2}_*$-module generator of $C$, hence the element $1 \in E^{C_2}_{*+|n\sigma| - (n+1)\sigma}$ maps to an $E^{C_2}_*$-module generator of $C'$. We deduce that $E^{C_2}_{*+|n\sigma| - (n+1)\sigma}$ is generated as an $E^{C_2}_*$-module by $1$ and 
\begin{align*}
\text{Im}\left( E^{C_2}_{*+|n\sigma| - n\sigma} \overset{u}{\longrightarrow} E^{C_2}_{*+|n\sigma| - (n+1)\sigma} \right) & = \text{Im} \left( E^{C_2}_* \{ 1, \dots , u^n \}  \overset{u}{\longrightarrow} E^{C_2}_{*+|n\sigma| - (n+1)\sigma} \right) \\
& =  \text{Im}\left( E^{C_2}_{*} \{ u , \dots , u^{n+1}\} \to E^{C_2}_{*+|n\sigma| - (n+1)\sigma}\right),
\end{align*}
from which it follows that $E^{C_2}_{*+|(n+1)\sigma| - (n+1)\sigma}$ is generated over $E^{C_2}_*$ by $1, \dots , u^{n+1}$. 
\end{proof}

\begin{remark}
We mention that the previous result does not necessarily hold for an $\Omega_{C_2}$-algebra $E_{C_2}$ failing the flatness hypotheses of a geometric orientation. For instance, the Eilenberg-Maclane spectrum $H \underline{\mathbb{F}_2}_{C_2}$ is an $\Omega_{C_2}$-algebra, but $H\underline{\mathbb{F}_2}^{C_2}_{*-\sigma}$ has rank $3$ over $\FF_2$, so it can not be generated by $\{1,u\}$ over $H\underline{\mathbb{F}_2}^{C_2}_* = \FF_2$.
\end{remark}

\begin{proposition}\label{filtration}
If $E_{C_2}$ is a geometrically oriented $C_2$-spectrum with stabilization $\widehat{E}_{C_2} = E_{C_2}[1/\tau]$, then for any complete flag $(\rho_i)_{i=1}^\infty$, the map 
 \[
 E^{C_2}_{*+|n\sigma|-n\sigma}(\mathbf{CP}^\infty_{C_2}) \to \widehat{E}^{C_2}_*(\mathbf{CP}^\infty_{C_2}) \cong \bigoplus_{i=1}^\infty \widehat{E}^{C_2}_*\{ \Pi_{\rho_1 + \cdots + \rho_i}\}
 \]
is injective, and identifies $E^{C_2}_{*+|n\sigma|-n\sigma}(\mathbf{CP}^\infty_{C_2})$ with 
\[
\left\{ \sum a_i \Pi_{\rho_1 + \cdots + \rho_i} \in \widehat{E}^{C_2}_*(\mathbf{CP}^\infty_{C_2}) : a_i \in E^{C_2}_{*+|n\sigma| -n\sigma} \subset \widehat{E}^{C_2}_* \right\},
\]
where $\ell_i$ is the number of copies of $\sigma$ in $(\rho_1 + \cdots + \rho_{i-1})\rho_i^{-1}$.
\end{proposition}
\begin{proof}
Choose a complete flag $(\rho_i)_{i+1}^\infty$ and set $V_i = \rho_1 + \cdots + \rho_i$. We can apply $E^{C_2}_{*+|n\sigma|-n\sigma}(-)$ to the diagram
\[ \begin{tikzcd} 
* \ar[r] & \mathbf{CP}(V_1)_+ \ar[r] \ar[d] & \mathbf{CP}(V_{2})_+ \ar[r] \ar[d] & \mathbf{CP}(V_{3})_+ \ar[r] \ar[d] & \mathbf{CP}(V_{4})_+ \ar[r] \ar[d]  &  \cdots \\
&  S^0 & S^{V_1 \rho_2^{-1}} &S^{V_2 \rho_3^{-1}} &S^{V_3 \rho_4^{-1}} & \cdots
\end{tikzcd} \]
which yields a spectral sequence $\mathcal{E}$ with signature 
\[
\mathcal{E}^1_{p,q} = E^{C_2}_{p + |n\sigma| - n\sigma}(S^{V_q \rho_{q+1}^{-1}}) \Rightarrow E^{C_2}_{p+q + |n\sigma| -n\sigma}(\mathbf{CP}^\infty_{C_2}).
\]
By mapping to the spectral sequence associated to the $\widehat{E}_{C_2}$-homology of this diagram, we deduce that the spectral sequence collapses, which leads to the desired direct sum decomposition.
\end{proof}

Our final major result of this section is a universality statement for the $C_2$-equivariant formal group law associated to the universal geometrically oriented $C_2$-spectrum $\Omega_{C_2}$.  This result asserts that the the filtration present on a filtered $C_2$-equivariant formal group law is completely determined by $F_0A$ and the filtration on the universal $C_2$-equivariant formal group law $(MU^{C_2}_*,MU^{C_2}_*(\mathbf{CP}^\infty_{C_2}))$.
\begin{theorem}
If $(F_\bullet A, F_\bullet D)$ is a filtered $C_2$-equivariant formal group law, then 
\begin{align*}
F_nA & = F_nMU^{C_2}_* \cdot F_0A\text{, and }\\
F_nD & = F_nMU^{C_2}_* \cdot F_0D.
\end{align*}
\end{theorem}

\begin{proof}
Both equalities follow from the fact that $F_nA$ is generated over $F_0A$ by the elements $1,\dots , u^n \in F_nMU^{C_2}_*$.
\end{proof}

\subsection{Equivariant projective spaces}\label{changeofbasis}
In this section we identify the geometrically defined classes 
\[\pi_{m + n \sigma} = [ \mathbf{CP}(m+n\sigma)] \in \Omega^{C_2}_*\]
in terms of purely algebraic data (Proposition \ref{identification}). We describe a method for writing the classes $\pi_{m+n\sigma}$ in terms of our generators of $\Omega^{C_2}_*$, and illustrate this method for some small values of $m$ and $n$ (Proposition \ref{examples}). We then prove Theorem \ref{geometricbasistheorem}, which relates the geometrically defined classes $\pi_{m+n\sigma} \in \Omega^{C_2}_*$ and $\Pi_{m+n\sigma} \in \Omega^{C_2}_*(\mathbf{CP}^\infty_{C_2})$ to the algebraic structure of filtered $C_2$-equivariant formal group laws.

We have seen in the previous section that the filtration present on a filtered $C_2$-equivariant formal group law $(F_\bullet A,F_\bullet D)$ is controlled by the euler class $u \in A$ and the geometric classes $\pi_{m+n\sigma} \in A$. For this reason, we'd like to identify the classes $\pi_{m+n\sigma} \in \Omega^{C_2}_*$ in terms of our presentation 
\[
\Omega^{C_2}_* = MU_*[d_{i,j},q_j]/I
\] 
from Theorem \ref{geometriccobordism}, or at least identify these classes in terms of purely algebraic data. Before doing so, we review the non-equivariant case.

Consider the non-equivariant complex projective space $\mathbf{CP}(k) = \mathbf{CP}^{k-1}$ for some $k \geq 0$. We can detect the class $[\mathbf{CP}(k)] \in MU_*$ by applying the Hurewicz homomorphism
\[
MU_* \to H_*MU = \mathbb{Z}[b_1,b_2,\dots],
\]
which is injective. This map encodes characteristic numbers of stably almost complex manifolds, in that the composite
\[
MU_* \to H_*MU \cong H_*(BU) \cong  \text{Hom}(H^*(BU) , \mathbb{Z}) 
\]
is adjoint to the pairing 
\begin{align*}
H^*(BU) \otimes MU_* & \longrightarrow  \ZZ\\
c_I \otimes M & \mapsto \langle c_I(\nu) , [M] \rangle
\end{align*}
where $c_I(\nu)$ is the total chern class of the stable normal bundle $\nu$ of $M$, and $[M] \in H_*(M)$ is the fundamental class of $M$. Since the stable normal bundle $\nu$ of $\mathbf{CP}(k)$ is equal to $-k \{\gamma^1\}$,\footnote{We write $\{ \xi\}$ for the stable equivalence class of a vector bundle $\xi$, and we write $-\nu$ for the $\oplus$-inverse of a stable vector bundle $\nu$.} this implies that the image of $[\mathbf{CP}(k)]$ under the Hurewicz map is the coefficient of $x^{k-1}$ in the power series 
\[
\dfrac{1}{(1+b_1x+b_2x^2+\cdots)^k}.
\]
By the Lagrangian inversion formula, this is equal to $km_{k-1}$ where $x+m_1x^2+m_2x^3+\cdots$ is the functional inverse of $x + b_1x^2 + b_2x^3 + \cdots$. This calculation is originally due to Mischenko  \cite{Mischenko}.

Let's return to the $C_2$-equivariant setting, where we'd like to describe the classes $\pi_{m+n\sigma}=[\mathbf{CP}(m+n\sigma)] \in \Omega^{C_2}_*$ in terms of purely algebraic data. We can use the fact that any class in $\Omega^{C_2}_*$ is determined by its underlying class in $MU_*$, and its image in the geometric fixed point ring 
\[
\Phi MU^{C_2}_* = MU_*[b_1',b_2',\dots][u^{\pm 1}].
\]
This is because the kernel of $\Omega^{C_2}_* \to \Phi^{C_2}MU_*$ is a free $MU_*$-module on $q_1$, and the augmentation $\Omega^{C_2}_* \to MU_*$ maps $q_1$ to $2 \in MU_*$, which is not a zero divisor. We determine the image of $\pi_{m+n\sigma}$ in $MU_*$ and $\Phi MU^{C_2}_*$ in the following proposition. 
\begin{proposition} \label{identification}
The composite 
\[\Omega^{C_2}_* \to MU_* \to \mathbb{Z}[b_i : i \geq 1]\]
maps $\pi_{m+n\sigma}$ to 
\begin{equation}\label{underlyingclass}
(m+n)m_{m+n-1} = \text{coeff}_{x^{m+n-1}} \dfrac{1}{(1+b_1x+b_2x^2+\cdots)^{m+n}},
\end{equation}
and the composite 
\[
\Omega^{C_2}_* \to \Phi MU^{C_2}_* \to \mathbb{Z}[b_i,b_i': i \geq 1][u^{\pm 1}]
\]
maps $\pi_{m+n\sigma}$ to the sum
\begin{equation}\label{fixedpointclass}
 \left(\text{coeff}_{x^m} \dfrac{1}{(1 + b_1x + b_2 x^2 + \cdots )^{m}(1+b'_1x + b_2'x^2 + \cdots)^n}\right)u^{-n}
 \end{equation}
 \begin{equation*}
+ \left(\text{coeff}_{x^n} \dfrac{1}{(1 + b_1x + b_2 x^2 + \cdots )^{n}(1+b'_1x + b_2'x^2 + \cdots)^m}\right)u^{-m}.
\end{equation*}
\end{proposition}

\begin{proof}
The augmentation maps $\pi_{m+n\sigma}$ to $[\mathbf{CP}(m+n)] \in MU_*$, which was determined in our non-equivariant discussion above. Thus, our main task is to determine the image of $\pi_{m+n\sigma}$ in the geometric fixed point ring. Geometrically, the class $\pi_{m+n\sigma} = [\mathbf{CP}(m+n\sigma)] $ maps in the geometric fixed points to
\[
\left[\mathbf{CP}(m) , -\{ \nu\mid_{\mathbf{CP}(m)}^{\mathbf{CP}(m+n)}\}\right]u^{-n}+ \left[\mathbf{CP}(n) , -\{\nu\mid_{\mathbf{CP}(n)}^{\mathbf{CP}(m+n)}\}\right]u^{-m} \in MU_*[b_i'][u^{\pm 1}],\]
where we have used the fact that elements of $MU_*[b_i'] = MU_*(BU)$ are represented by pairs $[M,\xi]$ of a stably almost complex manifold $M$ equipped with a stable vector bundle $\xi$. In order to detect the image of the classes $[\mathbf{CP}(k), -\{\nu\mid_{\mathbf{CP}(k)}^{\mathbf{CP}(k+\ell)}\}]$ in $ MU_*[b_i':i\geq 1]$, we can apply the Hurewicz homomorphism 
\[ MU_*[b_i'] =MU_*(BU) \to \widetilde{H}_*(MU \wedge BU_+) = \ZZ[b_i,b_i' : i \geq 1],\]
which is injective. Much like the case of the Hurewicz homomorphism $MU_* \to H_*(MU)$, we can think of $MU_*(BU) \to \widetilde{H}_*(MU \wedge BU_+)$ as encoding generalized characteristic numbers. More precisely, we have a pairing
\begin{align*}
 MU_*(BU) \otimes H^*(BU) \otimes H^*(BU) & \longrightarrow \ZZ\\
[M,\xi] \otimes c_I \otimes c_J & \longmapsto \langle c_I(\nu) c_J(\xi) , [M] \rangle
\end{align*}
where $\nu$ is the stable normal bundle of $M$ and $[M] \in H_{*}(M)$ is the fundamental class of $M$. This map is adjoint to the map 
\[
MU_*(BU) \to \text{Hom}(H^*(BU) \otimes H^*(BU) , \ZZ),
\]
which corresponds to the Hurewicz homomorphism under the isomorphism 
\begin{align*}
 \text{Hom}(H^*(BU) \otimes H^*(BU) , \ZZ) & \cong \text{Hom}(H^*(BU \times BU) , \ZZ)\\
& \cong   H_*(BU \times BU)\\
&  \cong \widetilde{H}_*(MU \wedge BU_+).
\end{align*} 
 The stable normal bundle of $\mathbf{CP}(k)$ is $\nu =  - k\{\gamma^1\}$, and so
\[
\{\nu \mid_{\mathbf{CP}(k)}^{\mathbf{CP}(k+\ell)} \} = - k \{ \gamma^1\} + (k+ \ell) \{ \gamma^1 \} = \ell\{\gamma^1\}.
\]
From this it follows that $[\mathbf{CP}(k) , -\{\nu \mid_{\mathbf{CP}(k)}^{\mathbf{CP}(k+\ell)} \}] = [\mathbf{CP}(k),-\ell \{\gamma^1\}]$. Since the direct sum of vector bundles corresponds to the product on $MU_*(BU)$, we can deduce that $[\mathbf{CP}(k),-\ell \{\gamma^1\}]$ maps via the Hurewicz homomorphism to the coefficient of $x^{k-1}$ in the power series 
\begin{align*}
 \dfrac{1}{(1 + b_1x + b_2 x^2 + \cdots )^{k}(1+b'_1x + b_2'x^2 + \cdots)^\ell} \in  \ZZ[b_i,b_i' : i \geq 1][[x]],
\end{align*}
which implies the result.
\end{proof}

We can use the previous proposition to express the classes $\pi_{m+n\sigma} \in MU_*[d_{i,j},q_j]/I$ in terms of the generators $d_{i,j}, q_j$ for some small values of $m$ and $n$. In order to do so, our first step is to construct a lift $\tilde{\pi}_{m+n\sigma} \in MU_*[d_{i,j},q_j]$ of the image of $\pi_{m+n\sigma}$ in $\mathbb{Z}[b_i,b_i'][u^{\pm 1}]$. We can do this using the formulas  \ref{underlyingclass}, \ref{fixedpointclass} , and the formula  \ref{generatorimages} of section \ref{pullbacks} .  Since the map $\Omega^{C_2}_* \to \Phi MU^{C_2}_*$ is not injective, the lift $\tilde{\pi}_{m+n\sigma}$ might not be the right one. However, if we have such a lift $\tilde{\pi}_{m+n\sigma}$, then since the kernel of $\Omega^{C_2}_* \to \mathbb{Z}[b_i,b_i'][u^{\pm 1}]$ is $MU_* \{q_1\}$, we can deduce that 
\[
\pi_{m+n\sigma} = \tilde{\pi}_{m+n\sigma} + \gamma q_1
\]
where $\gamma = (m+n)m_{m+n-1} - |\tilde{\pi}_{m+n\sigma}|$, and $|\tilde{\pi}_{m+n\sigma}|$ denotes the image of $\tilde{\pi}_{m+n\sigma}$ under the augmentation $\Omega^{C_2}_* \to MU_*$, which is determined by 
\[
d_{i,j} \mapsto c_{i,j}\text{,} \hspace{0.5in} q_j \mapsto p_j.
\]
We employ this strategy to calculate $\pi_{m+n\sigma}$ for some small values of $m$ and $n$. The validity of these equalities can be verified by considering the image of each side of the equation in $\mathbb{Z}[b_i]$ and $\mathbb{Z}[b_i,b_i'][u^{\pm 1}]$.

\begin{example}\label{examples}$(m,n) = (1,1)$
\[
\pi_{1+\sigma} = - q_2
\]
\end{example}
\begin{example}$(m,n) = (2,1)$
\[
\pi_{2+\sigma}  = d_{1,0} - a_{1,1}q_2
\]
\end{example}
\begin{example}$(m,n) = (2,2)$
\[
\pi_{2+2\sigma} = 4d_{1,1}+2q_4 - 2q_2q_3 - q_2^3 + (6b_1^3 -18 b_1b_2 + 6b_3)q_1
\]
\end{example}

Having analyzed the classes $\pi_{m+n\sigma} \in \Omega^{C_2}_*$, our next goal is to prove Theorem \ref{geometricbasistheorem}. In order to do so, we'll analyze the geometry of equivariant projective spaces, and the equivariant Pontrjagin-Thom construction. If $V$ is a $C_2$-representation, write $\gamma(V)$ for the tautological line bundle on $\mathbf{CP}(V)$. 
Define a function $s: \mathbf{C} \to \mathbf{C}$ by 
\[ 
s(\lambda) = 
\begin{cases} 
0 & \lambda = 0 \\
\frac{1}{\lambda} & \lambda \neq 0. \\
\end{cases} 
\]
\begin{lemma} \label{projectivespacelemma}
Suppose $V$ is a $C_2$ representation and $W = \rho_1 \oplus \cdots \oplus \rho_k$ where each $\rho_i$ is irreducible. Define
\[
\nu = \bigoplus_{i=1}^k \rho_i^{-1} \gamma(V).
\]
Then the map
\begin{align*}
\mathbf{CP}(V \oplus W) / \mathbf{CP}(W)  & \to \text{Th}(\nu \to \mathbf{CP}(V)) \\
[\vec{v} : \lambda_1 : \cdots : \lambda_k] & \mapsto 
\begin{cases} \infty & \vec{v} = 0 \\
([\vec{v}] , s(\lambda_1)\vec{v}, \cdots , s(\lambda_k) \vec{v} ) & \vec{v} \neq 0.
\end{cases} 
\end{align*} 
is an isomorphism of based $C_2$-spaces.
\end{lemma} 

We continue our notation from the previous lemma in the following.

\begin{lemma}\label{homologylemma}
The composite
\[ \begin{tikzcd} 
MU^{C_2}_*(\mathbf{CP}(V \oplus W)) \ar[r] & MU^{C_2}_*(\mathbf{CP}(V \oplus W), \mathbf{CP}(W)) \ar[r,"\cong"] & MU^{C_2}_*\left(D(\nu),S(\nu)\right)
\end{tikzcd} \]
takes $\Pi_{V \oplus W} $ to $\left[D(\nu) \overset{id}{\longrightarrow} D(\nu)\right] \in MU^{C_2}_*(D(\nu),S(\nu))$.
\end{lemma}
\begin{proof}
We will construct maps fitting into the following commutative diagram. 
\[
\begin{tikzcd} 
\mathbf{CP}(V \oplus W) \ar[r] \ar[d,swap,"i_0"] & \mathbf{CP}(V \oplus W) / \mathbf{CP}(W) \ar[d,"\cong"] \\
\mathbf{CP}(V \oplus W) \times [0,1] \ar[r,"F"] & \text{Th}(\nu) \\
D(\nu) \ar[u,"i_1"] \ar[r] & D(\nu)/S(\nu) \ar[u,swap,"\cong"] 
\end{tikzcd} 
\]

The map $i_0$ is defined by $i_0(x) = (x,0)$. Using Lemma \ref{projectivespacelemma}, we can identify 
\[\mathbf{CP}(V \oplus W)\setminus \mathbf{CP}(W) \cong \text{Th}(\nu) \setminus \{ \infty \} = E(\nu) \]
with the total space of $\nu$, so we can consider $D(\nu) \subset E(\nu)$ as a subspace of $ \mathbf{CP}(V \oplus W)$. The map $i_1$ is then defined by $i_1(\vec{v}) = (\vec{v} , 1)$. We define the map $F$ by 

\[
F(x,t)  = 
\begin{cases}
\infty & x \in \mathbf{CP}(W) \text{ or } x \in E(\nu) \text{ has norm } |x| \geq 1/t \\
\tan(\frac{\pi}{2}t |x|) x & x \in E(\nu) \text{ has norm } |x| < 1/t.
\end{cases} 
\]
Then $F$ extends the quotient maps $\mathbf{CP}(V \oplus W) \to \text{Th}(\nu)$ and $D(\nu) \to \text{Th}(\nu)$, and $F$ sends the complement of $P(V \oplus W) \coprod D(\nu) $ in $\partial\left(\mathbf{CP}(V \oplus W) \times [0,1]\right)$ to the basepoint of $\text{Th}(\nu)$, so $F$ is a cobordism between $\mathbf{CP}(V \oplus W) \to \text{Th}(\nu)$ and $D(\nu) \to D(\nu)$. 
\end{proof}

\begin{lemma} \label{cohomologylemma}
The isomorphism 
\[MU^*_{C_2}(\mathbf{CP}(V \oplus W), \mathbf{CP}(W)) \cong MU^*_{C_2}(D(\nu),S(\nu))\]
takes $x^W$ to the thom class $\tau(\nu)$.
\end{lemma} 
\begin{proof} If we write $W = \rho_1 \oplus \cdots \oplus \rho_k$ where each $\rho_i$ is irreducible, then the diagram
\[ \begin{tikzcd} 
\mathbf{CP}(V \oplus W)/\mathbf{CP}(W) \ar[r] \ar[d,swap,"\delta"] & \text{Th}\left( \bigoplus_{i=1}^k \rho_i^{-1} \gamma(V) \right)  \ar[d] \\
 \bigwedge_{i=1}^k \mathbf{CP}(V \oplus \rho_i) / \mathbf{CP}(\rho_i) \ar[d,swap,"\cong" ] \ar[r,"\cong"] &\bigwedge_{i=1}^k\text{Th}\left( \rho_i^{-1} \gamma(V) \right) \ar[d,"\cong"]  \\
 \bigwedge_{i=1}^k\mathbf{CP}(\rho_i^{-1} V \oplus 1)/\mathbf{CP}(1)  \ar[r,"\cong"] \ar[d,swap,"\bigwedge_{i=1}^k x"] &   \bigwedge_{i=1}^k \text{Th}(\gamma(\rho_i^{-1} V)) \ar[d]  \\
 \bigwedge_{i=1}^k \Sigma^2 MU_{C_2}  \ar[r,"="] &  \bigwedge_{i=1}^k \Sigma^2MU_{C_2}
 \end{tikzcd} \]
 commutes, where the composite of the vertical arrows on the left is $x^W$, and the composite of the vertical arrows on the right is $\tau(\nu)$.
\end{proof} 

\begin{proposition} \label{pairingrelation}
If $V$ and $W$ are $C_2$ representations, then 
\[
\langle \Pi_{V \oplus W} , x^W \rangle = \pi_V.
\]
\end{proposition}

\begin{proof}
The Pontrjagin-Thom construction takes $\Pi_V = [\mathbf{CP}(V) \to \mathbf{CP}(V)] \in \Omega^{C_2}_*(\mathbf{CP}(V))$ to the class of a map $f:S^X \to MU_{C_2}(Y) \wedge \mathbf{CP}(V \oplus W)_+$ such that $\mathbf{CP}(V \oplus W) \subset X$ is the preimage of the zero section of $\xi(Y) \to MU_{C_2}(Y)$. By lemmas \ref{homologylemma} and \ref{cohomologylemma}, the element 
\[ \langle \Pi_{V \oplus W} , x^W \rangle = \langle [D(\nu) \to D(\nu)] , \tau(\nu) \rangle \in MU^{C_2}_*\]
 is represented by the composite 
\[ \begin{tikzcd} 
S^X \ar[r,"f"] & MU_{C_2}(Y) \wedge \mathbf{CP}(V \oplus W)_+  \ar[r] & MU_{C_2}(Y) \wedge \mathbf{CP}(V \oplus W)/\mathbf{CP}(W) \ar[d,"\cong"] \\
 & & MU_{C_2}(Y) \wedge \text{Th}(\nu) \ar[d,"id \wedge g"]  \\
 & & MU_{C_2}(Y) \wedge MU(Y') \ar[d] \\
  & & MU_{C_2}(Y \oplus Y'),
\end{tikzcd} \] 
where $g: \text{Th}(\nu) \to MU_{C_2}(Y')$ is obtained by applying $\text{Th}(-)$ to a vector bundle map $\nu \to \xi(Y')$. Since the isomorphism $\mathbf{CP}(V \oplus W) /\mathbf{CP}(W) \cong \text{Th}(\nu)$ identifies $\mathbf{CP}(V)$ with the zero section of $\nu$, this composite is a model for $\pi_V \in MU^{C_2}_*$.
\end{proof} 

We now give the proof of Theorem \ref{geometricbasistheorem}.

\begin{proof}
Result (1) follows from Proposition \ref{pairingrelation} since
\begin{align*}
\Pi_{\rho_1 + \cdots + \rho_n} & = \sum_{i=1}^n \langle \Pi_{\rho_1+\cdots+ \rho_n} , x^{\rho_1 + \cdots + \rho_{i-1}}\rangle \beta(\rho_1, \dots ,\rho_i) \\
& = \sum_{i=1}^n \pi_{\rho_i + \cdots+ \rho_n} \beta(\rho_1, \dots, \rho_i).
\end{align*} 
Result (2) follows from (1) since the matrix expressing $\{ \Pi_{\rho_1 + \cdots + \rho_i} : 1 \leq i \leq n\}$ in terms of the basis $\{ \beta(\rho_1, \dots , \rho_i) : 1 \leq i \leq n\}$ is invertible.
\end{proof} 

\section{$RO(C_2)$-graded calculations} \label{furthercalculations}

In developing our theory of geometrically oriented $C_2$-spectra, we calculated the extended coefficient rings of various $C_2$-spectra, most notably connective $K$-theory $k_{C_2}$, and geometric cobordism $\Omega_{C_2}$. While we only needed the extended coefficient ring of these geometrically oriented $C_2$-spectra to understand their stabilization and the structure of their associated filtered $C_2$-equivariant formal group law, it is of independent interest to understand the full $RO(C_2)$-graded coefficients of these spectra. The purpose of this section is to complete the calculation of $k^{C_2}_\star$ and $\Omega^{C_2}_\star$. The reader will see that  $k^{C_2}_\star$ and $\Omega^{C_2}_\star$ are much more complicated than the extended coefficient rings $k^{C_2}_\diamond $ and $\Omega^{C_2}_\diamond$. In particular, neither $k^{C_2}_\star$ nor $\Omega^{C_2}_\star$ is concentrated in even degrees.

\subsection{The $RO(C_2)$-graded coefficients of $k_{C_2}$}\label{appendixA}

We begin by calculating the $RO(C_2)$-graded coefficients of the connective cover $k_{C_2}$ of $C_2$-equivariant complex $K$-theory. We illustrate the Mackey functor structure explicitly, since it is no more difficult to do so. The labels $\Box, \circ, n,$ and $n/m$ in the statement of our calculation refer to the $C_2$-Mackey functors
\[ \begin{tikzcd} 
\Box = \underline{R} & &  \circ  & & n& & n/m  \\
 \mathbb{Z}[\sigma]/(\sigma^2 - 1) \ar[dd, "1\text{,} \sigma \mapsto 1" , swap, bend right = 20]  & &\mathbb{Z}\{1+ \sigma\}  \ar[dd,  "1+\sigma  \mapsto 2",swap,bend right =20] & & \mathbb{Z}  \ar[dd, bend right = 20] & & \mathbb{Z}/2^{n-m}\mathbb{Z} \ar[dd, bend right = 20] \\
& & & & & &   \\
 \ZZ \ar[uu, "1 \mapsto 1 + \sigma", swap, bend right = 20] & &  \mathbb{Z} \ar[uu, "1 \mapsto 1 + \sigma", swap, bend right = 20]  & &  0 \ar[uu, bend right = 20] & & 0 \ar[uu, bend right = 20]
\end{tikzcd} \]
where the value of each Mackey functor at $C_2/C_2$ is shown on top, and the value at $C_2/e$ is shown on bottom. The reader should think of the Mackey functor $n$ as the $n$th power $\underline{J}^n$ of the augmentation ideal $\underline{J} \subset \underline{R}$, and $n/m$ as the quotient Mackey functor $\underline{J}^n/\underline{J}^m$. 

\begin{theorem}
The $RO(C_2)$-graded coefficients of the connective cover $k_{C_2}$ of equivariant $K$-theory are depicted below.

\[
\begin{tikzpicture}[scale = 0.75]
\coordinate (Origin)   at (0,0);
    \coordinate (XAxisMin) at (-11,0);
    \coordinate (XAxisMax) at (9,0);
    \coordinate (YAxisMin) at (0,-7);
    \coordinate (YAxisMax) at (0,7);

    \draw[thin, gray] [xshift=7.5,yshift=7.5](-9.9,-9.9) grid [xshift=7.5,yshift=7.5](8.9,8.9);
    \draw [very thin, gray,-latex] (0,0) -- (0,10);
    \draw [very thin, gray,-latex] (0,0) -- (10,0);

    \node at (0,11) {$\alpha$};
    \node at (0,0) {$\Box$};
    \node at (2,0) {$\Box$};
    \node at (4,0) {$\Box$};
    \node at (6,0) {$\Box$};
      \node at (8,0) {$\Box$};
    \node at (0,2) {$\Box$};
    \node at (2,2) {$\Box$};
    \node at (4,2) {$\Box$};
    \node at (6,2) {$\Box$};
      \node at (8,2) {$\Box$};
    \node at (0,4) {$\Box$};
    \node at (2,4) {$\Box$};
    \node at (4,4) {$\Box$};
    \node at (6,4) {$\Box$};
      \node at (8,4) {$\Box$};
    \node at (0,6) {$\Box$};
    \node at (2,6) {$\Box$};
    \node at (4,6) {$\Box$};
    \node at (6,6) {$\Box$};
      \node at (8,6) {$\Box$};
    \node at (0,8) {$\Box$};
    \node at (2,8) {$\Box$};
    \node at (4,8) {$\Box$};
    \node at (6,8) {$\Box$};
      \node at (8,8) {$\Box$};
    
    \node at (0,1) {$1$};
    \node at (2,1) {$1$};
    \node at (4,1) {$1$};
    \node at (6,1) {$1$};
        \node at (8,1) {$1$};
    \node at (0,3) {$1$};
    \node at (2,3) {$1$};
    \node at (4,3) {$1$};
    \node at (6,3) {$1$};
        \node at (8,3) {$1$};
    \node at (0,5) {$1$};
    \node at (2,5) {$1$};
    \node at (4,5) {$1$};
    \node at (6,5) {$1$};
        \node at (8,5) {$1$};
    \node at (0,7) {$1$};
    \node at (2,7) {$1$};
    \node at (4,7) {$1$};
    \node at (6,7) {$1$};
        \node at (8,7) {$1$};

    \node at (11,0) {$1$};
    \node at (0,-2) {$1$};
    \node at (2,-2) {$\Box$};
    \node at (4,-2) {$\Box$};
    \node at (6,-2) {$\Box$};
        \node at (8,-2) {$\Box$};
        
    \node at (0,-4) {$2$};
    \node at (2,-4) {$1$};
    \node at (4,-4) {$\Box$};
    \node at (6,-4) {$\Box$};
        \node at (8,-4) {$\Box$};
        
    \node at (0,-6) {$3$};
    \node at (2,-6) {$2$};
    \node at (4,-6) {$1$};
    \node at (6,-6) {$\Box$};
        \node at (8,-6) {$\Box$};
        
            \node at (0,-8) {$4$};
    \node at (2,-8) {$3$};
    \node at (4,-8) {$2$};
    \node at (6,-8) {$1$};
        \node at (8,-8) {$\Box$};
    
    \node at (0,-1) {$1$};
    \node at (2,-1) {$1$};
    \node at (4,-1) {$1$};
    \node at (6,-1) {$1$};
            \node at (8,-1) {$1$};
            
    \node at (0,-3) {$2$};
    \node at (2,-3) {$1$};
    \node at (4,-3) {$1$};
    \node at (6,-3) {$1$};
            \node at (8,-3) {$1$};
            
    \node at (0,-5) {$3$};
    \node at (2,-5) {$2$};
    \node at (4,-5) {$1$};
    \node at (6,-5) {$1$};
            \node at (8,-5) {$1$};

    \node at (0,-7) {$4$};
    \node at (2,-7) {$3$};
    \node at (4,-7) {$2$};
    \node at (6,-7) {$1$};
            \node at (8,-7) {$1$};

     \node at (-3,4) {$1/2$};

     \node at (-3,6) {$1/3$};
     \node at (-5,6) {$2/3$};

     \node at (-3,8) {$1/4$};
     \node at (-5,8) {$2/4$};
     \node at (-7,8) {$3/4$};
     
          \node at (-3,5) {$1/2$};

     \node at (-3,7) {$1/3$};
     \node at (-5,7) {$2/3$};
     
     \node at (-2,2) {$\circ$};
          \node at (-2,4) {$\circ$};
               \node at (-2,6) {$\circ$};
                    \node at (-2,8) {$\circ$};
     \node at (-4,4) {$\circ$};
          \node at (-4,6) {$\circ$};
                    \node at (-4,8) {$\circ$};
          \node at (-6,6) {$\circ$};
                    \node at (-6,8) {$\circ$};  
                    \node at (-8,8) {$\circ$}; 
                    
         \node at (9,-8.8) {$\ddots$};
         \node at (-9,9.2) {$\ddots$};
        \node at (9,8.8) {\reflectbox{\rotatebox[origin=c]{180}{$\ddots$}}};             
\end{tikzpicture}
\]
\end{theorem}

\begin{proof}
We apply $\underline{k}^{-*}(-) \to \underline{ku}^{-*}(-)$ to the cofiber sequence 
\[
S(n\sigma)_+  \to S^0 \to S^{n\sigma}
\]
which yields the diagram
\[
\begin{tikzcd}
\cdots \ar[r]  & \underline{k}^{-*-1}(S(n\sigma)) \ar[d] \ar[r] & \underline{k}_{*}  \ar[d]  \ar[r] & \underline{k}_{*+ n\sigma} \ar[d]  \ar[r] & \underline{k}^{-*} (S(n\sigma)) \ar[d] \ar[r] &  \cdots \\
\cdots \ar[r] &  \underline{ku}^{-*-1}(S(n\sigma)) \ar[r] & \underline{ku}_{*}  \ar[r,"u^n"] &  \underline{ku}_{*+n\sigma} \ar[r] & \underline{ku}^{-*} (S(n\sigma)) \ar[r] & \cdots \\
\end{tikzcd} 
\]
whose rows are exact. The map $\underline{k}^{-*}(S(n\sigma)) \to \underline{ku}^{-*}(S(n\sigma))$ is an isomorphism since $S(n\sigma)$ is free as a based $C_2$-space, and $k_{C_2} \to ku_{C_2}$ is a non-equivariant equivalence. Since $ku_{C_2}$ is complex stable, we know that $\underline{ku}_{*+n\sigma} \cong \underline{ku}_{*+ 2n}$. Exactness of the rows implies that $\underline{k}_{*+n\sigma} \to \underline{ku}_{*+n\sigma}$ is an isomorphism for $* \geq 0$, and $\underline{k}_{*+n\sigma} = 0$ for $* < 2n$. For any $-2n \leq - 2m \leq -2 $, the relevant part of our diagram is
\[ \begin{tikzcd} 
0 \ar[r] \ar[d] & \cdot \ar[d,"\cong"] \ar[r] & \underline{k}_{-2m +n \sigma}  \ar[r] \ar[d] & 0  \ar[r] \ar[d] & \cdot \ar[d,"\cong"] \ar[r]  & \underline{k}_{-2m - 1 + n\sigma} \ar[r] \ar[d] & 0 \ar[d] \\
0 \ar[r] & \cdot \ar[r] & \underline{R} \ar[r,"(\sigma - 1)^n"] & \underline{J}^m \ar[r]  & \cdot  \ar[r] & 0 \ar[r] & 0.
\end{tikzcd} \]
Exactness of the rows implies that 
\begin{align*}
\underline{k}_{-2m + n\sigma} = \text{ker}\left( \underline{R} \overset{(\sigma -1)^n}{\longrightarrow} \underline{J}^m \right) = \circ
\end{align*}
and 
\begin{align*}
\underline{k}_{-2m + n\sigma} = \text{coker}\left( \underline{R} \overset{(\sigma -1)^n}{\longrightarrow} \underline{J}^m \right) = \underline{J}^m/\underline{J}^n \cong m/n.
\end{align*}
Finally, by applying $\underline{k}_{*\pm 2n\alpha}(-)$ to the cofiber sequence $C_{2+} \to S^0 \to S^\alpha$ we can deduce the structure of $\underline{k}_{*\pm(2n+1)\alpha}$ from that of $\underline{k}_{*+2n\alpha}$.
\end{proof}

\subsection{The $RO(C_2)$-graded coefficients of $\Omega_{C_2}$}

Next, we calculate the $RO(C_2)$-graded coefficients of the geometric complex cobordism spectrum $\Omega_{C_2}$. The good range
\[
\Omega^{C_2}_\diamond = \bigoplus_{n \geq 0} \Omega^{C_2}_{*-n\sigma}
\]
was already calculated in section \ref{extendedsection}. To calculate the remaining piece, we need the following lemma, which we have already used several times in this paper.
\begin{lemma} \label{halemma}
If 
\[ \begin{tikzcd} 
0 \ar[r] & A \ar[r,"\iota"] \ar[d,"="] & B \ar[r,"\kappa"] \ar[d,"\beta"] & C \ar[r,"\lambda"] \ar[d,"\gamma"] & D \ar[r,"\mu"] \ar[d,"="] & E \ar[r] \ar[d] & 0\\
0 \ar[r] & A \ar[r,"\phi"] & B' \ar[r,"\chi"] & C' \ar[r,"\psi"] & D \ar[r] & 0 
\end{tikzcd} \]
is a commutative diagram of abelian groups whose rows are exact, then 
\[
B \cong \text{ker}\left( \begin{tikzcd} B' \oplus C \ar[r,"\chi - \gamma"] & C' \end{tikzcd} \right)
\]
and
\[
D \cong \text{coker}\left( \begin{tikzcd} B' \oplus C \ar[r,"\chi - \gamma"] & C' \end{tikzcd} \right).
\]
\end{lemma}
\begin{proof}
This is an elementary diagram chase.
\end{proof} 

We can now finish our calculation of the complete $RO(C_2)$-graded coefficients of the geometric complex cobordism spectrum $\Omega_{C_2}$.
\begin{theorem}
If $n \geq 0$, then 
\begin{enumerate}
\item 
\[\Omega^{C_2}_{*-2n\alpha}  \cong \dfrac{\Omega^{C_2}_* \{ 1 , \dots , u^n\}}{\begin{matrix} u^k(d_{i,j}-c_{i,j}) = u^{k+1}d_{i,j+1} \\ u^k(q_j - p_j) = u^{k+1}q_{j+1} \end{matrix} } \hspace{0.3in} \begin{matrix} i \geq 1 \text{ and }j \geq 0\\ 0 \leq k < n \end{matrix}\]
\item
\[\Omega^{C_2}_{*-(2n+1)\alpha} \cong \Omega^{C_2}_{*-2n\alpha}/q_1.\]
\item
\[ \Omega^{C_2}_{*+2n\alpha} \cong \Omega^{C_2}_{\text{even} + 2n\alpha} \oplus \Omega^{C_2}_{\text{odd} + 2n\alpha}\]
where
\[ 
\Omega^{C_2}_{\text{even} + 2n\alpha} \cong MU_*\{q_1\} \oplus \left((u^n) \cap \Omega^{C_2}_* \right)\]
and
\[
\Omega^{C_2}_{\text{odd}+2n\alpha} \cong  \frac{MU_{*-1}[u] }{ \begin{pmatrix} u^n \; , \; \sum_{\ell = 0}^{n-1} c_{i,j + \ell} u^\ell \; , \; \sum_{\ell = 0}^{n-1} p_{j + \ell} u^\ell \end{pmatrix}}\]
\item
\[\Omega^{C_2}_{*+ (2n+1)\alpha} \cong  \Omega^{C_2}_{*+2n\alpha} /q_1\]
\end{enumerate}
\end{theorem}

\begin{proof}
Our presentation of $\Omega^{C_2}_{*-2n\alpha} = \Omega^{C_2}_{*-n\sigma}$ was calculated in section \ref{extendedsection}. To calculate $\Omega^{C_2}_{*+2n\alpha} = \Omega_{C_2}^{-*-2n\alpha}$, we apply $\Omega_{C_2}^*(-)$ to the diagram 
\[ \begin{tikzcd} 
S(2n\alpha)_+ \ar[r] \ar[d] & S^0 \ar[r] \ar[d] & S^{2n\alpha} \ar[d] \\
S(2n\alpha)_+ \wedge S(\infty \alpha)_+ \ar[r] & S(\infty \alpha)_+ \ar[r] & S^{2n\alpha} \wedge S(\infty \alpha)_+
\end{tikzcd} \]
which yields 
\[ \begin{tikzcd} 
0 & \ar[l] \ar[d] \Omega^{C_2}_{*-1+2n\alpha}  & \ar[l] \frac{MU_*[[u]]}{([2]u,u^n)} \ar[d,"="] & \Omega^{C_2}_* \ar[l] \ar[d] & \ar[l] \Omega^{C_2}_{*+2n\alpha} \ar[d] &\ar[l] MU_*\{q_1\} \ar[d,"="] & \ar[l] 0 \\
& 0 & \ar[l] \frac{MU_*[[u]]}{([2]u,u^n)}  & \ar[l] \frac{MU_*[[u]]}{[2]u} & \ar[l,"u^n"] \frac{MU_*[[u]]}{[2]u}& \ar[l] MU_*\{ q_1\} &  \ar[l] 0
\end{tikzcd} \]
so Lemma \ref{halemma} implies that the even and odd part of $\Omega^{C_2}_{*+2n\alpha}$ are isomorphic to 
\[
\text{ker} \left( \Omega^{C_2}_* \oplus \dfrac{MU_*[[u]]}{[2]u} \to \dfrac{MU_*[[u]]}{[2]u}\right)
\]
and 
\[
\text{coker} \left( \Omega^{C_2}_* \oplus \dfrac{MU_*[[u]]}{[2]u} \to \dfrac{MU_*[[u]]}{[2]u}\right),
\]
respectively. Our presentation of the cokernel is obtained by quotienting $MU_*[[u]]/([2]u,u^n)$ by the image of each of the generators $d_{i,j}, q_j \in \Omega^{C_2}_*$. To calculate a presentation of the kernel, we consider the pullback square 
\[ \begin{tikzcd} 
\Omega^{C_2}_{\text{even} + 2n\alpha} \ar[r] \ar[d] & \Omega^{C_2}_* \ar[d] \\
MU_*[[u]]/[2]u \ar[r,"u^n"] & MU_*[[u]]/[2]u.
\end{tikzcd} \]
Since the kernel of each horizontal arrow is $MU_*\{q_1\}$, we obtain a pullback square
\[ \begin{tikzcd} 
\Omega^{C_2}_{\text{even} + 2n\alpha}/q_1 \ar[r] \ar[d] & \Omega^{C_2}_* \ar[d] \\
(MU_*[[u]]/[2]u)/q_1 \ar[r,"u^n"] & MU_*[[u]]/[2]u.
\end{tikzcd} \]
by killing $q_1$ in the domain of each of the horizontal maps. This square identifies $\Omega^{C_2}_{\text{even}+2n\alpha}$ with the kernel of $\Omega^{C_2}_* \to MU_*[[u]]/[2]u$, and since $MU^{C_2}_* \to MU_*[[u]]/[2]u$ induces an isomorphism 
\[MU^{C_2}_*/(u^n) \cong MU_*[[u]]/([2]u,u^n),\]
the kernel of $\Omega^{C_2}_* \to MU_*[[u]]/([2]u,u^n)$ is the intersection of $(u^n) \subset MU^{C_2}_*$ with $ \Omega^{C_2}_*$. We provide generators for the intersection ideal $(u^n) \cap \Omega^{C_2}_*$ in Proposition \ref{generators}. Finally, to calculate $\Omega^{C_2}_{*\pm(2n+1)\alpha}$, we apply $\Omega^{C_2}_{* \pm 2n\alpha}(-)$ to the cofiber sequence 
\[
C_{2+} \to S^0 \to S^\alpha
\]
which yields 
\[
0 \to \Omega^{\{e\}}_{*\pm 2n\alpha} \to \Omega^{C_2}_{*\pm 2n\alpha} \to \Omega^{C_2}_{*+(2n+1)\alpha} \to 0
\]
and the result follows from the fact that $\Omega^{\{e\}}_{*\pm 2n\alpha} \cong MU_*\{q_1\}$.
\end{proof} 

The only part of $\Omega^{C_2}_\star$ that we have not yet described explicitly is the intersection ideal $(u^n) \cap \Omega^{C_2}_*$. The following proposition gives us a generating set for this ideal. Let $S$ be the set of all monomials in $ \{ d_{i,j} - c_{i,j} ,q_j - p_j: i \geq 1 \text{ and } j \geq 0\}$. Consider the map $\phi: MU_*[d_{i,j},q_j] \to MU_*[[u]]$ determined by 
\begin{align*}
\phi(d_{i,j} - c_{i,j}) & =  c_{i,j+1}u + c_{i,j+2} u^2 + \cdots ,\\
\phi(q_j - p_j) & = p_{j+1}u + p_{j+2}u^2+ \cdots .
\end{align*}
Define a function $\phi_n:S \to MU_*$ by letting $\phi_n(m)$ be the coefficient of $u^n$ in $\phi(m)$, i.e. so that 
\[
\phi(m) = \phi_0(m) + \phi_1(m) u + \phi_2(m) u^2 + \cdots.
\]
We write $|m|$ for the total degree of $m$, so for example $|(d_{i,j}-c_{i,j})| = |(q_j - p_j)| = 1$ and $|(d_{i,j}-c_{i,j})^4(q_\ell - p_\ell)^3| = 7$. 

\begin{proposition}\label{generators}
The ideal $(u^n) \cap \Omega^{C_2}_*$ is generated by the $n$th power $J^n$ of the augmentation ideal of $\Omega^{C_2}_*$, together with the collection of all elements of the form 
\[
\sum_{m\in S, |m| < n} \alpha_m m \in \Omega^{C_2}_* ,
\]
with $\alpha_m \in MU_*$ such that 
\[
\sum_{|m| \leq k} \phi_k(\alpha_m) = 0 \in MU_*/2.
\]
for each $1 \leq k \leq  n-1 $.
\end{proposition}

\begin{proof}
By inspection of the corresponding quotient rings, we have $J = (u) \cap MU^{C_2}_*$, which implies  $J^n \subset (u^n) \cap \Omega^{C_2}_*$. For this reason, it suffices to calculate the kernel of 
\[
\Omega^{C_2}_*/J^n \to MU^{C_2}_*/(u^n).
\]
We can use the fact that  $f \in \Omega^{C_2}_*$ is in $(u^n)$ if and only if its image in each of 
\begin{align*}
 MU^{C_2}_*/(u) , (u)/(u^2) , (u^2)/(u^3), \dots , (u^{n-1})/(u^n) 
\end{align*}
is zero. Given any $f \in \Omega^{C_2}_*/J^n$ we can write 
\[
f =  \alpha_1 + \sum_{|m|<n} \alpha_m m
\]
for some coefficients $\alpha_1,\alpha_m \in MU_*$. The condition that $f$ is zero in each of the associated graded pieces is precisely the condition in the statement of the result, since $MU^{C_2}_*/(u) = MU_*$ and $(u^k)/(u^{k+1}) = MU_*/2 \{ u^k\}$.
\end{proof}

\section{Appendix}

The purpose of this appendix is two-fold. First, in section \ref{appendixC}, we prove a technical lemma from commutative algebra which was needed in order to calculate our presentation of the geometric cobordism ring $\Omega^{C_2}_*$. Second, in section \ref{AppendixD}, we review the theory of $G$-equivariant formal group laws, as defined in \cite{CGK1}, and prove that our new ``homological" formulation of $G$-equivariant formal group laws is equivalent to the original ``cohomological" formulation. We can consider this as an equivariant formal group theoretic version of Cartier duality, which asserts that a formal group is determined by its algebra of (continuous) functions, or by its coalgebra of (compactly supported) distributions.




\subsection{Eliminating the euler class $u$}\label{appendixC}

In this section we prove the main technical result which allows us to calculate the relations among the generators $d_{i,j},q_j$ of the geometric cobordism ring $\Omega^{C_2}_*$. Let $R$ be a domain and consider the ring \[R[u,x_1,x_2,\dots] = R[u,x_i].\] By a monomial in $R[u,x_i]$, we mean an element of the form $u^mx_{i_1}^{n_1} \dots x_{i_k}^{n_k}$. We can order the variables $u,x_1,x_2,\dots$ by $x_1 \prec x_2 \prec \dots \prec u$, and this induces an order on the set of monomials in $R[u,x_i]$. If $q \in R[u,x_i]$ is any polynomial, then we write $M(q)$ for the greatest monomial that occurs in $q$, and we write $LT(q)$ for the leading term of $q$, which is just $M(q)$ together with its coefficient in $R$. 
 
\begin{lemma} Let $R$ be a domain and let $I \subset R[u,x_1,x_2,\dots]$ be the ideal
\[
I = (ux_i+ p_i \; : \; i \geq 1 )
\]
for some $p_1,p_2,\dots \in R[x_1,x_2,\dots]$. Then the intersection ideal $I \cap R[x_1,x_2,\dots]$
 is equal to 
\[
J = (x_ip_j - x_j p_i :  i,j \geq 1).
\]
\end{lemma}
\begin{proof}
We know that $J \subseteq I \cap R[x_i]$ since for any $i,j \geq 1$ we have 
\[
x_ip_j - x_j p_i = x_i(ux_j+ p_j) - x_j(ux_i + p_i)\in I.
\]
It remains to show that $I \cap R[x_i ] \subseteq J.$ Suppose we have $q_1,\dots, q_m \in R[u,x_i]$ and
\[
f = \sum_{t=1}^m q_t(ux_{i_t} + p_{i_t}) \in R[x_i].
\]
We assume without loss of generality that $i_s \neq i_t$ for $s \neq t$. After reordering the terms in the sum we can assume that for some $1< k \leq m$, the terms $q_1(ux_{i_1} + p_{i_1}) , \dots , q_k(ux_{i_k} + p_{i_k})$ have the same leading monomial, and this is greater than the leading monomial in any of the terms $q_{k+1}(ux_{i_{k+1}} + p_{i_{k+1}}) , \dots , q_m(ux_{i_m} + p_{i_m})$. We have 
\[
LT(q_t(ux_{i_t} + p_{i_t})) = LT(q_t)ux_{i_t}
\]
since $u$ is the greatest element in our order. Let $c_t \in R$ be the coefficient of $LT(q_t)$, so that $LT(q_t) = c_t M(q_t)$. By assumption, we have $M(q_t)ux_{i_t} = M(q_s)ux_{i_s}$ for all $1 \leq s,t \leq k$. From this we can deduce the equality 
\[
\dfrac{M(q_t)}{x_{i_s}} = \dfrac{M(q_s)}{x_{i_t}}
\]
for all such $s \neq t$. Since all of the leading terms must cancel as they have $u$-degree $0$, we must have $c_1 + \dots + c_k = 0$. With these two equalities in mind, we can write 
\begin{align*}
\sum_{t=1}^k LT(q_t)(ux_{i_t} + p_{i_t}) = \sum_{t=1}^k LT(q_t)p_{i_t}  & = \sum_{t=1}^k c_tM(q_t)p_{i_t}\\
& = \sum_{t=1}^{k-1} \left( c_1 + \dots + c_t \right)( M(q_t)p_{i_t} - M(q_{t+1}) p_{i_{t+1}})\\
& = \sum_{t=1}^{k-1} \left( c_1 + \dots + c_t \right)\frac{M(q_t)}{x_{i_{t+1}}}( x_{i_{t+1}}p_{i_t} - x_{i_t}p_{i_{t+1}}).
\end{align*}
Call this polynomial $g$, and set $f' = f-g $. Then we have $f = f' + g$ where $M(f')  \prec  M(f)$ and $g \in J$. We can apply this algorithm to $f'$, and after finitely many iterations we will have written $f$ as a sum of elements of $J$, so we deduce that $f \in J$.
\end{proof}
The same algorithm as in the proof above yields the following result.
\begin{lemma}\label{eliminationlemma}
Let $R$ be a domain and let $I \subset R[u,x_1,x_2, \dots]$ be the ideal 
\[ I = (ux_i + p_i : i \geq 1) \]
for some $p_1, p_2, \dots \in R[x_1,x_2, \dots]$. Then the $R[x_i]$-submodule of $R[u,x_i]/I$ generated by $1 , \dots , u^n$ is given by 
\[
\dfrac{R[x_i] \{ 1, \dots , u^n\} }{ \begin{matrix} x_ip_j - x_j p_i \\ u^{k+1} x_i + u^k p_i \end{matrix} } \hspace{0.5in}  i,j\geq 1 \text{ and } 0 \leq k < n.
\]
\end{lemma}

\subsection{Homological equivariant formal group laws}\label{AppendixD}
In this section we develop the theory of ``homological" equivariant formal group laws, and prove its equivalence to the definition given in \cite{CGK1}. While the body of this paper concerns the group $G = C_2$, in this section we work in the generality of an arbitrary finite abelian group $G$. We begin by recalling the definition of a $G$-equivariant formal group law as defined in \cite{CGK1}. Write $G^\vee = \text{Hom}(G,S^1)$ for the Pontrjagin-dual of $G$. Suppose $A$ is a commutative ring and $R$ is a complete topological $A$-algebra. The category of complete topological $A$-algebras is symmetric monoidal under the completed tensor product $\widehat{\otimes}= \widehat{\otimes}_A$ with unit $A$, regarded as a discrete $A$-algebra. For this reason, we can make sense of a cogroup object in the category of complete topological $A$-algebras, which we call a {\it complete topological} $A${\it -Hopf algebra}. An example of a complete topological Hopf algebra is the ring 
\[
A^{G^\vee} = \prod_{G^\vee} A
\] 
of $A$-valued functions on $G^\vee$, which is equipped with the product topology. If $R$ is a complete topological Hopf algebra equipped with a morphism $R \to A^{G^\vee}$, then we can define a $G^\vee$ action on $R$ by $r^\rho = (1 \otimes \text{ev}_{\rho^{-1}}) \Delta r$. For any $V = \rho_1 + \cdots + \rho_k$, we define 
\[
r^V = r^{\rho_1} \cdots r^{\rho_k}.
\]

\begin{definition}
A (cohomological) $G$-equivariant formal group law $(A,R)$ consists of a commutative ring $A$, a complete topological $A$-Hopf algebra $R$, a morphism $R \to A^{G^\vee}$, and an element $x \in R$, such that 
\begin{enumerate}
\item the sequence 
\[ \begin{tikzcd} 
0 \ar[r] & R \ar[r,"x"] & R \ar[r,"\epsilon"] & A \ar[r] & 0
\end{tikzcd} \]
is exact, and 
\item $R = \lim R/(x^{V})$.
\end{enumerate}
\end{definition}
If $E_G$ is a complex oriented $G$-spectrum, then $(A,R) = (E_{G}^*,E_{G}^*(\mathbf{CP}^\infty_G))$ is naturally a $G$-equivariant formal group law: The morphism $E_{G}^*(\mathbf{CP}^\infty_G) \to (E_{G}^*)^{G^\vee}$ is obtained by applying $E_{G}^*(-)$ to the inclusion
\[
G^\vee \cong \coprod_{\rho \in G^\vee} \mathbf{CP}(\rho) \to  \mathbf{CP}^\infty_G
\]
and the coordinate $x \in E^*_{G}(\mathbf{CP}^\infty_G)$ is the complex orientation of $E_G$. 

On the other hand, we can define a dual algebraic structure called a homological equivariant formal group law, which axiomatizes the algebraic structure of $(E^{G}_*,E^{G}_*(\mathbf{CP}^\infty_G))$. Before giving the definition, we'll review some necessary notation. Suppose $D$ is an $A$-Hopf algebra equipped with a map $A[G^\vee] \to D$. If $x$ is an  $A$-linear functional on $D$, we write $\langle d , x \rangle $ for the value of $\cap x$ at $d \in D$, and we write $x$ for the comultiplication-by-$x$ map
\[ \begin{tikzcd} 
D \ar[r,"\Delta"] & D \otimes D \ar[r,"\cap x "] & D \otimes A \cong D.
\end{tikzcd} \]
We can define a $G^\vee$ action on $\text{Hom}_A(D,A)$ by 
\[
\langle d , x^\rho \rangle = \langle \rho^{-1} d , x \rangle,
\]
and for any $V = \rho_1 + \cdots + \rho_k$, we can define $x^{V}$ by 
\[
\langle d , x^{V} \rangle = \langle \Delta d , x^{\rho_1} \otimes \cdots \otimes x^{\rho_k} \rangle.
\]
\begin{definition}
A homological $G$-equivariant formal group law $(A,D)$ consists of a commutative ring $A$, an $A$-Hopf algebra $D$, a morphism $A[G^\vee] \to D$, and an $A$-linear functional $x$ on $D$, such that 
\begin{enumerate}
\item the sequence 
\[ \begin{tikzcd} 
0 \ar[r] & A \ar[r,"\eta"] & D \ar[r,"\cap x"] & D \ar[r] & 0
\end{tikzcd} \]
is exact, and 
\item if $d \in D$, then there exists $V \in \text{Rep}(G)$ such that 
\[
d \cap x^{V} = 0.
\]
\end{enumerate}
\end{definition}
We will prove that the data of a homological $G$-equivariant formal group law is equivalent to that of a cohomological $G$-equivariant formal group law. We can do after proving some basic structural results about homological $G$-equivariant formal group laws. We begin by describing the additive and comultiplicative structure of such objects, which is relatively simple.

\begin{proposition}
If $(A,D)$ is a $G$-equivariant formal group law, then there exists a unique family of elements 
\[
\{\beta(\rho_1, \dots , \rho_n) \in D : n \geq 1, \rho_i \in G^\vee\}
\]
satisfying the following properties:
\begin{enumerate}
\item $\beta(\rho) \in D$ is the image of $\rho \in A[G^\vee]$ under the structure map $A[G^\vee] \to D$.
\item $\beta(\rho_1 , \dots , \rho_n) \cap x^{\rho_1} = \beta(\rho_2,\dots,\rho_n)$, and
\item \[ \epsilon(\beta(\rho_1,\dots,\rho_n)) = \begin{cases} 1 & n =1 \\ 0 & n  >  1.\end{cases}\]
\end{enumerate}
\end{proposition}

\begin{proof}
We construct the elements $\beta(\rho_1, \dots , \rho_n) \in D$ by induction on $n \geq 1$. We define $\beta(\rho_1) \in D$ to be the image of $\rho_1 \in A[G^\vee]$ under $A[G^\vee] \to D$. If we have defined $\beta(\rho_1, \dots , \rho_i)$ for all $ i  <   n$, then we define $\beta(\rho_1,\dots,\rho_n)\in D$ by first choosing any $\beta \in D$ such that $\beta \cap x^{\rho_1} = \beta(\rho_1, \dots , \rho_n)$, and then defining
\[
\beta(\rho_1, \dots , \rho_n)  = \beta - \epsilon(\beta) \beta(\rho_1).
\]

Properties (1), (2), and (3) are satisfied by construction. Suppose that we have another family of elements $\gamma(\rho_1, \dots , \rho_n) \in D$ satisfying properties (1), (2), and (3). We will prove that $\beta(\rho_1, \dots , \rho_n)= \gamma( \rho_1, \dots, \rho_n)$ by induction on $n$. The case $n = 1$ holds by property (1). Now 
\begin{align*}
\left( \beta(\rho_1, \dots , \rho_n) -  \gamma( \rho_1, \dots , \rho_n)  \right) \cap x^{\rho_1} & =\beta(\rho_1, \dots , \rho_n) \cap x^{\rho_1}- \gamma(\rho_1, \dots , \rho_n) \cap x^{\rho_1}\\
& = \beta(\rho_2, \dots , \rho_n)  - \gamma(\rho_2, \dots , \rho_n) = 0
\end{align*}
so $\beta(\rho_1, \dots , \rho_n) - \gamma(\rho_1, \dots , \rho_n) = a \beta(\rho_1) $ for some $a \in A$. But we compute 
\begin{align*}
a = \epsilon(a\beta(\rho_1)) = \epsilon(\beta(\rho_1, \dots , \rho_n) - \gamma(\rho_1, \dots , \rho_n) ) & = \epsilon(\beta(\rho_1, \dots , \rho_n)) - \epsilon(\gamma(\rho_1, \dots , \rho_n))\\
& = 0
\end{align*}
so $\beta(\rho_1, \dots , \rho_n) - \gamma(\rho_1, \dots , \rho_n)$.
\end{proof}

It turns out that the elements $\beta(\rho_1, \dots , \rho_i)$ associated to a complete flag $(\rho_i)_{i=1}^\infty$ form a free $A$-module basis for $D$.

\begin{lemma} \label{free}
If $(A,D)$ is a $G^\vee$-equivariant formal group law and $(\rho_i)_{i=1}^\infty$ is a complete flag, then the set
\[
 \{ \beta(\rho_1, \dots , \rho_n) : n \geq 1\}
\]
is an $A$-linear basis for $D$.
\end{lemma}

\begin{proof}
The elements $\beta(\rho_1, \dots , \rho_n) \in D$ determine an $A$-module map 
\[ \psi:  A\{\beta(\rho_1, \dots , \rho_n) : n \geq 1 \} \to D\]
 which we claim is an isomorphism. First, let's show that $\psi$ is surjective. Since $(\rho_i)_{i = 1}^\infty$ is a complete flag, we know that for any $d \in D$, there is some $n \geq 1$ such that $d \cap x^{\rho_1 + \cdots \rho_n} = 0$. If $d \cap x^{\rho_1} = 0$ then $d =  a \beta(\rho_1)$ for some $a \in A$, so $d$ is in the image of $\psi$. Suppose next that $d \cap x^{\rho_1 + \cdots + \rho_n} = 0$. Then $d \cap x^{\rho_1 + \cdots + \rho_{n-1}}$ is in the kernel of $\cap x^{\rho_n}$, so $d \cap x^{\rho_1+ \cdots + \rho_{n-1}} = a \beta(\rho_n)$ for some $a \in A$. Then 
 \begin{align*}
 (d - a\beta(\rho_1, \dots , \rho_n)) \cap x^{\rho_1+ \cdots + \rho_{n-1}}
 & = a \beta(\rho_n) - a \beta(\rho_n) \\
 & = 0,
 \end{align*}
so by induction $d - a \beta(\rho_1,\dots,\rho_{n-1})$ is in the image of $\psi$, hence so is $d$. Next lets show that $\phi$ is injective. Suppose $a_1,\dots, a_n \in A$ and 
\[ d = a_1 \beta(\rho_1) + \dots + a_n\beta(\rho_1, \dots , \rho_n) = 0\]
 in $D$. If $n =1$, then $d = a_1\beta(\rho_1)$ which is zero in $D$ if and only if $a_1 = 0$. Suppose inductively that $n > 1$. Then $d \cap x^{\rho_1+\cdots +\rho_{n-1}} = a_n\beta(\rho_n)= 0$, so $a_n = 0$, and by induction this implies that $a_0 = \dots = a_{n -1} = 0$.
\end{proof}

Next, we prove that the elements $\beta(\rho_1, \dots , \rho_i) \in D$ are dual to the linear functionals $x^{\rho_1 + \cdots + \rho_{i-1}}$. If $n=0$, then the symbol $x^{\rho_1 + \dots + \rho_n}$ is understood to mean the counit $x^ 0 = \epsilon :D \to A$.

\begin{lemma}
If $(A,D)$ is a $G^\vee$-equivariant formal group law and $(\rho_i)_{i=1}^\infty$ is a complete flag, then for any $d \in D$ we have 
\begin{align*}
d & = \sum_{i \geq 1 } \left\langle d , x^{\rho_1 + \cdots + \rho_{i-1}} \right\rangle \beta(\rho_1, \dots , \rho_i).
\end{align*}
\end{lemma}
\begin{proof}
Since $(\rho_i)_{i=1}^\infty$ is a complete flag, we know that if $d \in D$ then $d \cap x^{\rho_1 + \cdots + \rho_n} = 0$ for some $n \geq 1$. Suppose first that $d \cap x^{\rho_1} = 0$. Then $d = a \beta(\rho_1)$ for some $a \in A$, and we can compute 
\[
a = \epsilon( a \beta(\rho_1)) = \epsilon(d) = \left\langle d , x^0 \right\rangle,
\]
so $d = \left\langle d , x^0 \right\rangle \beta(\rho_1)$. If $i > 0$, then 
\[
\left\langle d , x^{\rho_1 + \cdots + \rho_i} \right\rangle = \left\langle d \cap x^{\rho_1}, x^{\rho_2 + \cdots + \rho_i} \right\rangle = \left\langle 0 , x^{\rho_2 + \cdots + \rho_i}\right\rangle = 0,
\]
so the formula holds in the case $n= 1$. Suppose inductively that $d \cap x^{\rho_1 + \cdots + \rho_n} = 0$. Then $d \cap x^{\rho_1 + \cdots + \rho_{n-1}}  \in \text{ker}(\cap x^{\rho_n})$, so $d \cap x^{\rho_1 + \cdots + \rho_{n-1}} = a \beta(\rho_n)$ for some $a \in A$, and applying $\epsilon$ shows that $a = \left\langle d, x^{\rho_1 + \cdots + \rho_{n-1}}\right\rangle$. We now have 
\[
d - \left\langle d , x^{\rho_1 + \cdots + \rho_{n-1}} \right\rangle \beta(\rho_1 , \dots , \rho_n)   \in \text{ker}(\cap x^{\rho_1 + \cdots + \rho_{n-1}}),
\] 
so by induction we have 
\begin{align*}
d - \left\langle d , x^{\rho_1 + \cdots + \rho_{n-1}} \right\rangle  \beta(\rho_1 , \dots , \rho_n) = \sum_{i \geq 1} a_i \beta(\rho_1, \dots , \rho_i)
\end{align*}
where 
\begin{align*}
a_i & = \left\langle d - \left\langle d , x^{\rho_1 + \cdots + \rho_{n-1}}) \right\rangle  \beta(\rho_1, \dots , \rho_n), x^{\rho_1 + \cdots + \rho_{i-1}}\right\rangle \\
& = \begin{cases}
\left\langle d , x^{\rho_1 + \cdots + \rho_{i-1}}\right \rangle & i < n\\
0 & i \geq n.
\end{cases}
\end{align*}
so
\begin{align*}
d = \sum_{i = 1}^n \left\langle d, x^{\rho_1 + \cdots + \rho_{i-1}} \right\rangle \beta(\rho_1,\dots,\rho_i)
\end{align*}
Our final step is to observe that $\left\langle d , x^{\rho_1 + \cdots + \rho_{i-1}} \right\rangle = 0$ if $i > n$.
\end{proof}

\begin{lemma}
If $(A,D)$ is a $G^\vee$-equivariant formal group law and $(\rho_i)_{i=1}^\infty$ is a complete flag, then for any $d' \otimes d'' \in D \otimes D$ we have 
\[
d' \otimes d'' = \sum_{i,j \geq 1} \left\langle d' \otimes d'' , x^{\rho_1 + \cdots + \rho_{i-1}} \otimes x^{\rho_i + \cdots + \rho_{i+j-1}}) \right\rangle \beta(\rho_1 , \dots , \rho_i) \otimes \beta(\rho_i , \dots , \rho_j).
\]
\end{lemma}
\begin{proof}
If $d' \otimes d'' \in D \otimes D$, then 
\begin{align*}
d' \otimes d'' & = \left( \sum_{i \geq 1} \left\langle d' , x^{\rho_1 + \cdots + \rho_{i-1}} \right\rangle \beta(\rho_1, \dots , \rho_i) \right) \otimes d''\\
& = \sum_{i \geq 1} \left\langle d' , x^{\rho_1 + \cdots + \rho_{i-1}} \right\rangle \beta(\rho_1, \dots , \rho_i) \otimes d''\\
& = \sum_{i \geq 0} \left\langle d' , x^{\rho_1 + \cdots + \rho_{i-1}} \right\rangle  \beta(\rho_1 , \dots ,\rho_i) \otimes \left( \sum_{j \geq 1} \left\langle d'', x^{\rho_i , \dots , \rho_{i+j-1}}) \right\rangle \beta(\rho_i , \dots , \rho_{i+j})\right) \\
& = \sum_{i,j \geq 1} \left\langle d',x^{\rho_1 + \cdots + \rho_{i-1}} \right\rangle \left\langle d'', x^{\rho_i + \cdots + \rho_{i+j-1}} \right\rangle \beta(\rho_1 , \dots , \rho_i) \otimes \beta(\rho_i , \dots , \rho_{i+j})\\
& = \sum_{i,j \geq 1} \left\langle d' \otimes d'' , x^{\rho_1 + \cdots + \rho_{i-1}} \otimes x^{\rho_i + \cdots + \rho_{i+j-1}} \right\rangle \beta(\rho_1 , \dots , \rho_i) \otimes \beta(\rho_i , \dots , \rho_{i+j}).
\end{align*} 
and the result holds for sums of simple tensors by $k$-linearity.
\end{proof}

The preceding result allows us to determine the comultiplicative structure of $D$.

\begin{lemma}
The coproduct $\Delta:D \to D \otimes D$ is determined by 
\[
\Delta \beta(\rho_1, \dots , \rho_n) = \sum_{i=1}^n \beta(\rho_1, \dots , \rho_i) \otimes \beta(\rho_i , \dots , \rho_n).
\]
\end{lemma} 

\begin{proof}
First, note that if $(\rho_i)_{i=1}^\infty$ is a complete flag, then 
\[
\left\langle \beta(\rho_1, \dots , \rho_i) , x^{\rho_1 + \cdots + \rho_{j-1}} \right\rangle = 
\begin{cases} 
1 & i-j \\
0 & i \neq j.
\end{cases}
\]
We compute 
\begin{align*}
\Delta \beta(\rho_1, \dots , \rho_n) & = \sum_{i,j \geq 1} \left\langle \Delta \beta(\rho_1, \dots , \rho_n) , x^{\rho_1 + \cdots + \rho_{i-1}} \otimes x^{\rho_i + \cdots + \rho_{i+j-1}} \right\rangle \beta(\rho_1, \dots , \rho_i) \otimes \beta(\rho_i + \cdots + \rho_{i+j})\\
& = \sum_{i,j \geq 1} \left\langle \beta(\rho_1, \dots , \rho_n) , x^{\rho_1 + \cdots + \rho_{i+j-1}} \right\rangle \beta(\rho_1, \dots , \rho_i) \otimes \beta(\rho_i + \cdots + \rho_{i+j})\\
& = \sum_{i=1}^n \beta(\rho_1 , \dots , \rho_i) \otimes \beta(\rho_i , \dots , \rho_n).
\end{align*}
\end{proof}

Having developed some basic properties of homological $G^\vee$-equivariant formal group laws, we can prove our Cartier duality theorem for equivariant formal group laws.

\begin{theorem}
If $A$ is a commutative ring, then the functors
\[ \begin{tikzcd} 
\begin{Bmatrix} \text{Cohomological }G\text{-equivariant} \\ \text{ formal group laws over }A \end{Bmatrix}   \ar[rr,shift left = 4,"\text{Hom}^\text{cts}_A(- \text{,} A)"]
& & \ar[ll,shift left = 4,"\text{Hom}_A(- \text{,} A)"] \begin{Bmatrix} \text{Homological }G\text{-equivariant} \\ \text{ formal group laws over }A \end{Bmatrix} 
\end{tikzcd} \]
are inverse equivalences of categories.
\end{theorem}
\begin{proof}
That the dual of a cohomological (resp. homological) $G$-equivariant formal group law carries the structure of a homological (resp. cohomological) $G$-equivariant formal group law follows from the fact that 
\[
\text{Hom}^\text{cts}_A(R \widehat{\otimes} R , A) \cong \text{Hom}^\text{cts}_A(R,A) \otimes \text{Hom}^\text{cts}_A(R,A)
\]
resp. 
\[
\text{Hom}_A(D \otimes D , A) \cong \text{Hom}_A(D,A) \; \widehat{\otimes} \; \text{Hom}_A(D,A).
\]
The assignments
\begin{align*}
R & \to \text{Hom}_A(\text{Hom}^\text{cts}_A(R,A),A)\\
r & \mapsto \text{ev}_r
\end{align*} 
and
\begin{align*}
D & \to \text{Hom}_A^\text{cts}(\text{Hom}_A(D,A),A)\\
d & \mapsto \text{ev}_d
\end{align*} 
define natural isomorphisms. This can be verified by observing that 
\[
R \cong \prod_{i=1}^\infty A\{ x^{\rho_1 + \cdots + \rho_{i-1}}\}\text{ and } D \cong \bigoplus_{i=1}^\infty A\{\beta(\rho_1 , \dots , \rho_i) \},\]
so the maps $r \mapsto \text{ev}_r$ and $d \mapsto \text{ev}_d$ are isomorphisms at the level of $A$-modules.
\end{proof}

\end{document}